\documentclass[english,article]{siamart190516} 

\usepackage[utf8]{inputenc}
\usepackage[english]{babel}
\usepackage{color}
\usepackage{amsmath}
\usepackage{amssymb}
\usepackage{amsfonts}
\usepackage{bbm}
\usepackage{pgfplots}
\usetikzlibrary{patterns}
\usepgfplotslibrary{fillbetween}
\usetikzlibrary{intersections}
\usepackage{subfig}
\usepackage{enumerate}
\usepackage{mathtools}
\usepackage[shortlabels]{enumitem}
\usepackage{url}

\pgfplotsset{compat=1.14}

\newsiamthm{conjecture}{Conjecture}

\DeclareMathOperator*{\argmin}{arg\,min}

\DeclareMathOperator{\sign}{sign}
\DeclareMathOperator{\rank}{rank}
\DeclareMathOperator{\nullity}{nullity}

\newcommand{\norm}[2]{\left \Vert{#1} \right \Vert_{#2}}

\newcommand{\F}[0]{\mathbb{F}}
\newcommand{\R}[0]{\mathbb{R}}
\newcommand{\C}[0]{\mathbb{C}}
\newcommand{\Fn}[0]{\mathbb{F}^n}
\newcommand{\Fnn}[0]{\mathbb{F}^{n \times n}}
\newcommand{\Rn}[0]{\mathbb{R}^n}
\newcommand{\Rnn}[0]{\mathbb{R}^{n \times n}}
\newcommand{\Cn}[0]{\mathbb{C}^n}
\newcommand{\Cnn}[0]{\mathbb{C}^{n \times n}}
\newcommand{\Fp}[1]{\mathbb{F}^{#1}}
\newcommand{\Rp}[1]{\mathbb{R}^{#1}}

\newcommand{\Fpq}[2]{\mathbb{F}^{#1 \times #2}}

\newcommand{\G}[1]{\mathcal{G}^{\mathbb{F}}(#1)}
\newcommand{\Ga}[1]{\mathcal{G}_{\text{aff}}^{\mathbb{F}}(#1)}

\newcommand{\Img}[1]{\text{Im}(#1)}
\newcommand{\Ker}[1]{\text{Ker}(#1)}
\newcommand{\Aff}[1]{\text{Aff}\left( #1 \right)}
\newcommand{\Span}[1]{\text{span}(#1)}
\newcommand{\Ind}[2]{\textbf{Ind}(#1,#2)}
\newcommand{\Indsimple}[2]{\textbf{Ind}_{#1}(#2)}
\newcommand{\T}[0]{\mathcal{T}}
\newcommand{\X}[0]{\mathcal{X}}
\newcommand{\Y}[0]{\mathcal{Y}}
\newcommand{\K}[0]{\mathcal{K}}
\newcommand{\omin}[0]{\omega_{\text{min}}}
\newcommand{\xbo}[0]{x_{b,\omega}}
\newcommand{\Uni}[1]{\text{U}(#1)}
\newcommand{\Sym}[1]{\text{Sym}(#1)}
\newcommand{\Pos}[1]{\text{P}^{+}(#1)}
\newcommand{\Gl}[1]{\text{GL}(#1)}
\newcommand{\Gr}[2]{\text{Gr}^{\mathbb{F}} (#1,#2)}
\newcommand{\Gra}[2]{\text{Graff}^{\mathbb{F}} (#1,#2)}
\newcommand{\Gam}[2]{\mathbf{\Gamma}_{#1}(#2)}
\newcommand{\evalmin}[1]{\lambda_{\text{min}}(#1)}
\newcommand{\D}[2]{\mathbf{D}_{#1}(#2)}
\newcommand{\dD}[3]{\partial \mathbf{D}_{#1}(#2,#3)}
\newcommand{\M}[0]{\mathcal{M}}
\newcommand{\Msym}[0]{\mathcal{M}_{\text{sym}}}
\newcommand{\MsymT}[0]{\mathcal{M}_{\text{sym,T}}}
\newcommand{\Minv}[0]{\mathcal{M}_{\text{inv}}}
\newcommand{\Msyminv}[0]{\mathcal{M}_{\text{sym,inv}}}
\newcommand{\Mpos}[0]{\mathcal{M}_{\text{pos}}}
\newcommand{\W}[0]{\widehat{W}}

\renewcommand{\S}[0]{\mathcal{S}}
\renewcommand{\dim}[1]{\text{dim}(#1)}
\renewcommand{\vec}[1]{\textbf{vec}(#1)}

\definecolor{darkgreen}{rgb}{0,0.5,0}

\title{The index of invariance and its implications for a parameterized least squares problem
    \thanks{Both authors contributed equally.
    \funding{Rahul Sarkar was partially supported by a fellowship from Schlumberger, and Léopold Cambier was partially supported by a fellowship from Total S.A., for the duration of this work.}}
}

\author{
L\'{e}opold Cambier\thanks{Institute for Computational and Mathematical Engineering, Stanford University
  (\email{lcambier@stanford.edu}, \url{www.stanford.edu/\~lcambier}).}
    \and
Rahul Sarkar\thanks{Institute for Computational and Mathematical Engineering, Stanford University
  (\email{rsarkar@stanford.edu}, \url{www.stanford.edu/\~rsarkar}).}
}

\begin{document}

\maketitle

\begin{abstract}
We study the least squares problem $x_{b,\omega} := \argmin_{x \in \S} \|(A + \omega I)^{-1/2} (b - Ax)\|_2$, with $A = A^*$, for a subspace $\S$ of $\F^n$ ($\F = \R$ or $\C$), and $\omega > -\evalmin{A}$.
We show that there exists a subspace $\Y$ of $\Fn$, independent of $b$, such that $\{x_{b,\omega} - x_{b,\mu} \mid \omega,\mu > -\evalmin{A}\} \subseteq \mathcal{Y}$, where $\dim{\mathcal{Y}} \leq \dim{\S + A\S} - \dim{\S} = \Indsimple{A}{\S}$, a quantity which we call the index of invariance of $\S$ with respect to $A$.
In particular if $\S$ is a Krylov subspace, this implies the low dimensionality result of Hallman \& Gu (2018) \cite{hallman2018lsmb}.
The least squares problem also has the property that when $A$ is positive, and $\S$ is a Krylov subspace, it reduces to the conjugate gradient problem for $\omega = 0$, and to the minimum residual problem in the limit $\omega \rightarrow \infty$.
We study several properties of $\Indsimple{A}{\S}$ in relation to $A$ and $\S$.
We show that in general, the dimension of the affine subspace $\X_b$ containing the solutions $x_{b,\omega}$ can be smaller than $\Indsimple{A}{\S}$ for all $b$.
However, we also exhibit some sufficient conditions on $A$ and $\S$, under which a related set $\X := \Span{\{x_{b,\omega} - x_{b,\mu} \mid b \in \F^n, \omega,\mu > -\evalmin{A}\}}$ has dimension equal to $\Indsimple{A}{\S}$.
We then study the injectivity of the map $\omega \mapsto x_{b,\omega}$, leading us to a proof of the convexity result from \cite{hallman2018lsmb}. We finish by showing that sets such as $M(\S,\S') := \{A \in \F^{n \times n} \mid \S + A\S = \S'\}$, for nested subspaces $\S \subseteq \S' \subseteq \Fn$, form smooth real manifolds, and explore some topological relationships between them.
\end{abstract}


\section{Introduction}
\label{sec:intro}

It was recently shown in \cite{hallman2018lsmb} that for any $b \in \mathbb{R}^m$, and a full column rank matrix $A \in \mathbb{R}^{m \times n}$ with $m \geq n$, the solution $x_{b,\omega} \in \mathbb{R}^n$ to the \textit{LSMB problem} for $\omega \geq 0$
\begin{equation}
\label{eq:lsmb-obj}
    \argmin \;\; \norm{(A^{\ast}A + \omega I)^{-\frac{1}{2}} A^{\ast} (b-Ax)}{2} \;\;\;\; \text{s.t. } x \in \mathcal{K}_k (A^{\ast} A, A^{\ast} b),
\end{equation}
is a convex combination of the LSQR \cite{paige1982lsqr} and LSMR \cite{fong2011lsmr} solutions (all notations are formally introduced in \Cref{ssec:defn}). Here $\omega \in \mathbb{R}$ is an arbitrarily chosen parameter, and 
\begin{equation}
\label{eq:least-sq-krylov}
    \mathcal{K}_k (A^{\ast} A, A^{\ast} b) := \Span{ \{A^{\ast}b,(A^{\ast} A) A^{\ast}b,(A^{\ast} A)^2 A^{\ast}b,\dots,(A^{\ast} A)^{k-1}A^{\ast}b\}}
\end{equation}
is the Krylov subspace over which we minimize \eqref{eq:lsmb-obj}. It was also noted in \cite{hallman2018lsmb} that when $\omega = 0$, one recovers the LSQR solution, while if $\omega \rightarrow \infty$, $x_{b,\omega}$ converges to the LSMR solution. Thus, the LSMB problem is a generalization of LSQR and LSMR. Furthermore, it was shown that for all $\omega \geq 0$, the iterates $x_{b,\omega}$ are convex combinations of $x_{b,0}$ and $x_{b,\infty}$. Thus by varying $\omega$, one obtains the set $\{x_{b,\omega} \mid \omega \geq 0\}$ contained in the line passing through the LSQR and LSMR solutions, which is a one-dimensional affine subspace of $\mathcal{K}_k (A^{\ast} A, A^{\ast} b)$.

Since $A$ is a full rank matrix, $A^{\ast} A$ is a positive matrix, and conversely any positive matrix $B \in \mathbb{R}^{n \times n}$ can be decomposed as $B = L^{\ast} L$ for some full rank matrix $L \in \mathbb{R}^{n \times n}$, for e.g. using the Cholesky decomposition. 
Thus, from the above, we immediately deduce the result that for any positive matrix $B \in \mathbb{R}^{n \times n}$ and $b' \in \mathbb{R}^n$, the set of solutions $\{x_{b',\omega} \in \mathbb{R}^n \mid \omega \geq 0\}$ to the problem
\begin{equation}
\label{eq:spd-obj-krylov}
    \argmin \;\; \norm{(B + \omega I)^{-\frac{1}{2}} (Bx - b')}{2} \;\;\;\; \text{s.t. } x \in \mathcal{K}_k (B, b'),
\end{equation}
also lie on a line. Moreover, as we will show later in \Cref{ssec:injectivity}, these are convex combinations of CG \cite{hestenes1952methods} and MINRES \cite{paige1975solution} solutions, the convexity being a direct consequence of the convexity result in \cite{hallman2018lsmb} as stated above. In fact the other direction is also true: if we knew that the solutions $x_{b',\omega}$ to \eqref{eq:spd-obj-krylov} lied on a line, then the corresponding result for the LSMB problem \eqref{eq:lsmb-obj} follows by simply substituting $A^{\ast} A$ and $A^{\ast} b$ in places of $B$ and $b'$ respectively in \eqref{eq:spd-obj-krylov}.

Below we briefly summarize our contributions, the main result of the paper, and how the paper is organized. The notations and objects appearing in \Cref{ssec:contrib,ssec:summary} are all formally introduced and defined in \Cref{sec:prelim}; however some of them are restated below for convenience.

\subsection{Contributions}
\label{ssec:contrib}
In this paper, we generalize the previously mentioned one-dimensional affine subspace result by studying the minimization problem
\[ x_{b,\omega} := \argmin_{x \in \T} \|(A + \omega I)^{-\frac{1}{2}} (Ax-b)\|_2, \]
with $\T \in \Gra{p}{n}$ ($\F = \R$ or $\F = \C$) an arbitrary affine subspace, $A \in \Gl{n} \cap \Sym{n}$, $b \in \Fn$, and $-\evalmin{A} < \omega \in \R$, where $\evalmin{A} \in \R$ is the smallest eigenvalue of $A$.
Also for any subspace $\S$ of $\Fn$, we define $\Indsimple{A}{\S} := \dim{\S + A\S} - \dim{\S}$, which we call the index of invariance of $\S$ with respect to $A$. We then prove a number of results:
\begin{enumerate}[(i)]
    \item We study (\Cref{sec:prelim}) the index of invariance in some detail and prove a number of its properties, such as upper-bounds, its relationship to a tridiagonal block decomposition of $A$, how it relates to $\Indsimple{A^{-1}}{\S}$, $\Indsimple{A^s}{\S}$, $\Indsimple{A}{\S^\perp}$ and subadditivity.
    \item We show (\Cref{sec:main-result}) that there exists a subspace $\mathcal{Y}$ such that for all $b \in \Fn$,  $\{x_{b,\omega} - x_{b,\mu} \mid \omega, \mu \geq - \evalmin{A}\} \subseteq \mathcal{Y}$, where $\dim{\mathcal{Y}} \leq \Indsimple{A}{\S}$, immediately generalizing the result from $\cite{hallman2018lsmb}$, since $\Indsimple{A}{\S} = 1$ for Krylov subspaces. We give an expression for $\mathcal{Y}$ as a function of $A$ and $\S$. This theorem is formally stated in \Cref{ssec:summary}.
    \item We then study (\Cref{sec:converse}) the tightness of the previously mentioned bound when one varies $\omega$, keeping $b$ fixed. Let $\X_b$ be the affine hull of $\{x_{b,\omega} \mid \omega > -\evalmin{A}\}$, for a fixed $b \in \Fn$.
    \begin{itemize}[-]
        \item We show that the 0-dimensional case is special, as $\dim{\X_b} = 0$ for all $b \in \Fn$, if and only if $\Indsimple{A}{\S} = 0$.
        \item We show however that there exist $A$ and $\S$ such that, for all $b \in \Fn$,  $\dim{\X_b} = 1$, while $\Indsimple{A}{\S}$ can be arbitrarily large.
        \item We finally show that the set $\{b \in \Fn \mid \dim{\X_b} = 0\}$ is non-trivial and has Lebesgue measure zero.
    \end{itemize}
    We continue by studying instead a related set $\X := \Span{\{x_{b,\omega} - x_{b,\mu} \mid b \in \F^n, \omega,\mu > -\evalmin{A}\}}$,
    where we find some sufficient conditions on $A$ and $\S$ ensuring $\dim{\X} = \Indsimple{A}{\S}$.
    \item We finish (\Cref{sec:related-results}) by studying some applications of our results. 
    We study the injectivity of the map $\omega \mapsto x_{b,\omega}$, which leads us to a proof of the convexity result from \cite{hallman2018lsmb} in the real Krylov case.
    We also study some new matrix manifolds that arise in connection with the index of invariance.
\end{enumerate}

\subsection{The main result}
\label{ssec:summary}

The main result of the paper is the following theorem which is proved in \Cref{sec:main-result}:

\begin{theorem} 
\label{thm:main_result}
Let $\F$ denote the field $\C$ or $\R$. Let $\Gra{p}{n}$ and $\Gr{p}{n}$ denote the set of $p$-dimensional affine subspaces and subspaces of $\Fn$ respectively, and let $\mathbf{\Gamma}_p : \Gra{p}{n} \rightarrow \Gr{p}{n}$ denote the vector bundle projection map. For $1 \leq p \leq n$, let $A \in \Sym{n} \cap \Gl{n}$ be an $n \times n$ invertible Hermitian matrix over $\F$, $\T \in \Gra{p}{n}$, and $b \in \Fn$. Define $\omin := -\evalmin{A}$, $\S := \Gam{p}{\T}$, $q := \dim{\S+A\S} - \dim{\S}$, and for all $\omega \in (\omin,\infty)$
\[ x_{b,\omega} := \argmin_{x \in \T} \| (A + \omega I)^{-\frac{1}{2}} (Ax - b) \|_2. \]
Then there exists a subspace $\Y \in \Gr{q}{n}$, independent of $b$, such that $x_{b,\omega} - x_{b,\mu} \in \Y$ for all $\omega, \mu \in (\omin,\infty)$.
When $q \geq 1$, if $V \in \Fpq{n}{p}$, $V' \in \Fpq{n}{q}$ are chosen such that $\begin{bmatrix}V & V'\end{bmatrix}$ is semi-unitary, $\Img{V} = \S$, and $\Img{\begin{bmatrix}V & V'\end{bmatrix}} = \S + A\S$, then $\Y = \Img{V (H^\ast H)^{-1} B^\ast}$ (not depending on the choice of $V, V'$), where $H^\ast = \begin{bmatrix}T & B^\ast\end{bmatrix}$, with $T = V^\ast A V$, and $B = V'^\ast A V$.
\end{theorem}

\subsection{Outline of the paper}
\label{subsec:outline}

The rest of the paper is structured as follows. 
In \Cref{sec:prelim} we first introduce some definitions and notations, motivate and formally state the problem. We then prove a number of properties related to the index of invariance.
In \Cref{sec:main-result} we prove the main result of the paper. 
In \Cref{sec:converse} we study the converse of the main result in some detail, i.e. under what conditions is the bound from \Cref{sec:main-result} tight.
\Cref{sec:related-results} explores the question of injectivity and some topological consequences of our results.
We finally finish by stating some open problems in \Cref{sec:future-work}.
\section{Preliminaries}
\label{sec:prelim}

\subsection{Definitions and notation}
\label{ssec:defn}

Relevant definitions and notations to be used throughout the paper are introduced here. Some additional aspects of topology and real analytic functions are required in \Cref{sec:converse,sec:related-results}, which we don't introduce below for brevity; but we use standard terminology. The unfamiliar reader is referred to \cite{rudin1964principles, armstrong2013basic} for a comprehensive treatment of these topics. In \Cref{sec:related-results} we have a few results on smooth manifold embeddings, and we use terminology consistent with \cite{lee2013smooth}. We use $\F$ to represent the field we work over, which can be either $\C$ or $\R$.

\subsubsection{Matrix notations}
\label{sssec:hilbert-spaces}

We define $\Fpq{p}{q}$ to be the set of all $p \times q$ matrices with $\F$-valued entries. Given any $A \in \Fpq{p}{q}$, $\Img{A}:= \{Ax \mid x \in \Fp{q}\}$ will denote its range, while $\Ker{A}:= \{x \in \Fp{q} \mid Ax=0\}$ will denote its kernel or nullspace. The rank of $A$ is defined as the dimension of its range, which we denote $\rank(A)$. The transpose of $A$ is denoted $A^\top$, while the adjoint of $A$ is denoted by $A^\ast$, and it is the complex conjugate transpose (resp. transpose) of $A$ when $\F = \C$ (resp. $\R$). $\overline{A}$ will denote the complex conjugate of $A$ without the transpose, unless specified otherwise. When $\F = \C$, $\Re(A)$ and $\Im(A)$ will denote the real and imaginary parts of $A$ respectively. The $(i,j)$ entry of $A$ will be denoted by $A_{ij}$. $A$ is said to be \textit{semi-unitary} if all its columns are orthonormal. If $p=q$, the adjugate of $A$ denoted $\text{adj}(A)$ is the transpose of the cofactor matrix of $A$. Sometimes we will refer to matrices as \textit{operators}.

If $A \in \Fnn$,  $A$ is said to be \textit{Hermitian} if $A = A^\ast$. When $\F = \C$, $A$ is said to be \textit{positive} if $x^\ast Ax > 0$ (i.e. real and positive) for all non-zero $x \in \Fn$ (this in fact ensures that $A$ is Hermitian), while if $\F = \R$ we additionally require that $\F$ is Hermitian for it to be positive. We say that $A$ is \textit{unitary} if $A^\ast A = I$, where $I \in \Fnn$ denotes the identity matrix, which also implies $AA^\ast = I$. The symbol $I$ will be used to denote square identity matrices of other shapes as well, but the shape will always be clear from context. The set of all Hermitian matrices, positive matrices, unitary matrices, and invertible matrices in $\Fnn$ will be denoted using the symbols $\Sym{n}$, $\Pos{n}$, \Uni{n}, and $\Gl{n}$ respectively, and recall that $\Pos{n}\subseteq \Gl{n}$ and $\Uni{n} \subseteq \Gl{n}$ ($\F$ will be clear from context). We recall that Hermitian matrices have real eigenvalues and positive matrices have positive eigenvalues. For any $A \in \Sym{n}$, $\evalmin{A}$ will denote its smallest eigenvalue. We also recall that if $A \in \Sym{n}$, then it admits a spectral decomposition $A = U \Lambda U^\ast$ for some $U \in \Uni{n}$ and a real diagonal matrix $\Lambda \in \Fnn$, which allows us to define\footnote{The spectral decomposition guarantees that $\Lambda$ has positive diagonal entries, and is unique up to conjugation by permutation matrices --- so $A^q$ defined this way is unique.} whenever $A \in \Pos{n}$, its power $A^q := U \Lambda^q U^\ast \in \Pos{n}$ for any $q \in \R$, where $(\Lambda^q)_{ij} = \Lambda_{ij}^q$ (sign always chosen to be positive) for all $1 \leq i,j \leq n$.

For any $A \in \Fpq{p}{q}$, and for any $1 \leq k \leq l \leq p$ and $1 \leq k' \leq l' \leq q$, we will define the block matrix $A_{k:l,k':l'} \in \Fpq{(l-k+1)}{(l'-k'+1)}$ as
\begin{equation}
    \left( A_{k:l,k':l'} \right)_{ij} = A_{k+i-1, k'+j-1}, \;\; \forall \;\; 1 \leq i \leq (l-k+1),  1 \leq j \leq (l'-k'+1).
\end{equation}
We also define the function $\textbf{vec}: \F^{p \times q} \to \F^{pq}$ by $\vec{A}_{i+(j-1)p} = A_{ij}$, where $1 \leq i \leq p$, and $1 \leq j \leq q$. Informally, $\vec{A}$ stacks the columns of $A$ on top of each other into a vector. For completeness, we define the span of zero vectors to be $\{0\}$.
Similarly, assuming $A \in \F^{p \times m}, B \in \F^{m \times q}$ and $C = AB$, if $m = 0$ we define $C = 0 \in \F^{p \times q}$.

\subsubsection{Subspaces and affine subspaces}
\label{sssec:affine-subspaces}
An affine subspace $\T$ of $\Fn$ is a set such that if $x \in \T$, $\S_x := \{x - y \mid y \in \T\}$ is a subspace of $\Fn$. Clearly $\S_x = \S_{x'}$ for distinct $x, x' \in \T$, so one can unambiguously associate a subspace $\S := \{x - y \mid y \in \T\}$ with $\T$, for any arbitrarily chosen $x \in \T$. We define the dimension of $\T$ as $\dim{\T}:= \dim{\S}$. The notation $\Gra{k}{n}$ (resp. $\Gr{k}{n}$) denotes the set of all $k$-dimensional affine subspaces (resp. subspaces) of $\Fn$. Defining $x_0 := \argmin_{x \in \T} \norm{x}{2}$ (this minimizer exists\footnote{Existence follows by choosing any point $x \in \T$ and defining the compact set $\T_1 := \{y \in \T \mid \norm{y}{2} \leq \norm{x}{2}\}$, which has the property that for any $y \in \T \setminus \T_1$, $\norm{y}{2} > \norm{x}{2}$, and so the minimization can be performed over $\T_1$.} and is unique because in the case $\F=\C$, if $\Cn \ni x = x_1 + i x_2$ with $x_1, x_2 \in \Rn$, the map $(x_1,x_2) \mapsto \norm{x}{2}^2$ is smooth and strictly convex, while if $\F = \R$, the map $\Rn \ni x \mapsto \norm{x}{2}^2$ is also smooth and strictly convex), it follows that every $\T \in \Gra{k}{n}$ can be represented uniquely as
\begin{equation}
\label{eq:graff-to-gr}
\T = x_0 + \S := \{x_0 + y \mid y \in \S\},
\end{equation}
where
$\S \in \Gr{k}{n}$, $x_0 \in \T$, and $\norm{x_0}{2} \leq \norm{x}{2}$ for all $x \in \T$. It will be useful to represent $\S$ appearing in \eqref{eq:graff-to-gr} by the map $\mathbf{\Gamma}_k: \Gra{k}{n} \rightarrow \Gr{k}{n}$; thus we can rewrite \eqref{eq:graff-to-gr} as $\T = x_0 + \Gam{k}{\T}$. An affine subspace $\T$ is a subspace if and only if $x_0 = 0$ in its representation \eqref{eq:graff-to-gr}. Defining $\G{n} := \bigcup_{k=0}^{n} \Gr{k}{n}$ and $\Ga{n} := \bigcup_{k=0}^{n} \Gra{k}{n}$, we have the set inclusions $\G{n} \subseteq \Ga{n}$, and $\Gr{k}{n} \subseteq \Gra{k}{n}$ for all $0 \leq k \leq n$.

If $\S \in \G{n}$, $\T \in \Ga{n}$, and $A \in \Fnn$, we define $A\S := \{Ax \mid x \in \S\}$, and $A\T := \{Ax \mid x \in \T\}$, which is a subspace and an affine subspace of $\Fn$ respectively. The orthogonal complement of $\S$ is defined as $\S^{\perp} := \{x \in \Fn \mid x^\ast y = 0, \; \forall \; y \in \S \}$. For $\S, \S' \in \G{n}$, we define the sum $\S + \S' := \{x + y \mid x \in \S, y \in \S'\}$, and if $\S \cap \S' = \{0\}$, this sum is a direct sum denoted as $\S \oplus \S'$. If $\mathcal{U} \in 2^{\G{n}}$ is infinite (possibly uncountable), we define $\sum \mathcal{U} = \bigl \{ \sum_{i=1}^{m} x_i \mid x_i \in \S \in \mathcal{U}, \; 0 \leq m < \infty \bigr \}$, and it is also a subspace. Intersections of subspaces (resp. affine subspaces), possibly uncountable, is a subspace (resp. affine subspace). If $\mathcal{X} \subseteq \Fn$, the \textit{affine hull} $\Aff{\mathcal{X}}$ of $\mathcal{X}$ is the intersection of all affine subspaces of $\Fn$ containing $\mathcal{X}$ and is an affine subspace. The \textit{linear hull} or span of $\mathcal{X}$, denoted $\Span{\mathcal{X}}$, is the intersection all subspaces of $\Fn$ containing $\mathcal{X}$. When $A \in \Fpq{p}{q}$, the linear hull of the set of its columns equals its range $\Img{A}$. For $A \in \Fnn$, $b \in \Fn$, and $\mathbb{Z} \ni k \geq 1$, we define the Krylov subspace $\K_k(A,b) := \Span{\{b, Ab, \dots, A^{k-1}b\}}$.

\subsubsection{Index of invariance}
\label{sssec:index}

We now introduce the most important quantity relevant for this paper which plays a key role in the proofs.

\begin{definition}
\label{def:index}
Let $A \in \Fnn$, and $\S \in \G{n}$ be a subspace of $\Fn$. We define the \textit{index of invariance} of $\S$ with respect to $A$ to be the codimension of $\S$ in $\S + A\S$. Formally we will represent this quantity as a map $\textbf{Ind} : \G{n} \times \Fnn \rightarrow \mathbb{Z}_{\geq 0}$, and so $\Ind{\S}{A} := \dim{\S + A\S} - \dim{S}$. In most of our applications $A$ will be fixed, in which case we will use the compressed notation $\Indsimple{A}{\S}$ to mean $\Ind{\S}{A}$, and treat it as a map $\textbf{Ind}_{A}: \G{n} \rightarrow \mathbb{Z}_{\geq 0}$. In this setting, we will refer to $\Indsimple{A}{\S}$ as simply the \textit{index} of $\S$.
\end{definition}

If $A \in \Fnn$, a subspace $\S \in \G{n}$ is called an \textit{invariant subspace} of $A$ or simply $A-$\textit{invariant} if $A\S \subseteq \S$. Thus it can be seen from \Cref{def:index} that $\Ind{\S}{A} = 0$ if and only if $\S$ is $A-$invariant. Another interesting example is that of a Krylov subspace $\K_k(A,b)$ that is not $A-$invariant, in which case it can be verified that $\Ind{\K_k(A,b)}{A} = 1$. We study several interesting properties of the index of invariance below in \Cref{ssec:properties-index}.

\subsubsection{Strong orthogonality}
\label{sssec:strong-orthogonality}
Finally, we introduce a notion of orthogonality of vectors that will be used in \Cref{ssec:zero-dim}, that is much stronger than the usual notion of orthogonality.

\begin{definition}
\label{def:strong-orthogonality}
Let $\Fn = \S_1 \oplus \dots \oplus \S_t$ be an orthogonal direct sum decomposition for $t$ orthogonal subspaces $\S_1, \dots, \S_t$. Let $\pi_i : \Fn \rightarrow \S_i$ denote the orthogonal projection map on $\S_i$, for all $1 \leq i \leq t$. Then two vectors $u, v \in \Fn$ are called \textit{strongly orthogonal} with respect to $\S_1, \dots, \S_t$ if and only if $\pi_i(u)$ is orthogonal to $\pi_i(v)$ for each $i$.
\end{definition}

It is important to note that if two vectors $u,v$ are strongly orthogonal with respect to $\S_1, \dots, \S_t$, then they are orthogonal (but the converse is not true). This is because the direct sum $\S_1 \oplus \dots \oplus \S_t = \Fn$, so we have $u = \sum_{i=1}^{t} \pi_i(u)$, and $v = \sum_{i=1}^{t} \pi_i(v)$, from which it follows that $u^\ast v = \sum_{i=1}^{t} (\pi_i(u))^\ast \pi_i(v)$, the other terms vanishing due to orthogonality of the subspaces, and finally by strong orthogonality we get $u^\ast v = 0$.
\subsection{Problem statement}
\label{ssec:problem-statement}

In this subsection, after we formally state the problem in the next paragraph, we will define a few quantities that will be used in its analysis and prove some easy facts.

Let $\F$ be $\C$ or $\R$. Let $A \in \Sym{n} \cap \Gl{n}$ be a Hermitian invertible operator, $\T \in \Gra{p}{n}$ be an affine subspace of dimension $1 \leq p \leq n$, and let $\omin = - \evalmin{A}$. For any $b \in \Fn$ and $\omega \in (\omin, \infty)$, we define $x_{b,\omega}$ to be the solution to the following minimization problem (the fact that this minimizer exists and is unique is proved in \Cref{lemma:x-bw}):
\begin{equation} 
\label{eq:x_b_w} 
x_{b,\omega} := \argmin_{x \in \T} \|(A + \omega I)^{- \frac{1}{2}}(b - Ax) \|_2,
\end{equation}
and in addition, we also define an affine subspace and a subspace
\begin{equation}
\label{eq:X-X_b}
\X_b := \Aff{\{\xbo \mid \omega > \omin\}}, \text{ and } \X := \sum_{b \in \Fn} \Gam{\dim{\X_b}}{\X_b}.
\end{equation}
We seek to resolve the following questions: What is the maximum dimension of $\X_b$ and $\X$? Conversely, does the dimensions of $\X_b$ and $\X$ say anything about the quantity $\Indsimple{A}{\Gam{p}{\T}}$?

\subsubsection{Characterizing the solution}
\label{sssec:solution-formula}

We start by giving an explicit solution for $x_{b,\omega}$. In fact we give the solution for a slightly more general case in the next lemma, and $x_{b,\omega}$ is obtained by setting $s = -1$ in the lemma (i.e. $x_{b,\omega} = x_{b,\omega, -1}$).

\begin{lemma} 
\label{lemma:x-bw}
Let $A \in \Sym{n} \cap \Gl{n}$, $b \in \Fn$, $\omin = - \evalmin{A}$, and $\T \in \Gra{p}{n}$. Let $\T = x_0 + \S$ be any representation of $\T$ for some $x_0 \in \T$, and $\S = \Gam{p}{\T}$. Then for any $\omega \in (\omin, \infty)$, and for any $s \in \mathbb{R}$, the problem
\begin{equation}
\label{eq:general-min-prob}
    \argmin_{x \in \T} \|A_\omega^{s/2}(b - Ax) \|_2
\end{equation}
has a unique solution $x_{b,\omega,s}$ given by
\begin{equation}
\label{eq:x_b_w_explicit}
    x_{b,\omega,s} := x_0 + V \left( V^* A A_\omega^{s} A V\right)^{-1} V^* A A_\omega^{s} (b - Ax_0)
\end{equation}
where $A_\omega := A + \omega I$, and $V \in \Fpq{n}{p}$ is any full rank matrix whose columns span $\S$. \eqref{eq:x_b_w_explicit} is well defined as it is independent of the choices $x_0$ and $V$.
\end{lemma}

The proof of this lemma is given in \Cref{app:appA}, and we simply note here that $\omega > \omin$ ensures that $A_\omega \in \Pos{n}$. Because of the freedom in the choices of $x_0$ and $V$ in \Cref{lemma:x-bw}, from now on unless otherwise specified, we will always assume that $V$ is semi-unitary, and $x_0$ is chosen so that it satisfies the unique representation of $\T$ in \eqref{eq:graff-to-gr}. It will also suffice to study the special case when $\T$ is a subspace, because of the following easy corollary of \Cref{lemma:x-bw}.

\begin{corollary}
\label{cor:general-min-prob-subspace}
Under the assumptions of \Cref{lemma:x-bw}, defining $b' := b - Ax_0$,
\begin{equation}
\label{eq:reduction-eqv}
    \argmin_{x \in \T} \|A_\omega^{s/2}(b - Ax) \|_2 = x_0 + \argmin_{x \in \S} \|A_\omega^{s/2}(b' - Ax) \|_2.
\end{equation}
\end{corollary}

\begin{proof}
This follows from \eqref{eq:x_b_w_explicit}, because when $\T$ is a subspace $x_0 = 0$.
\end{proof}

As a consequence, if $\X_b$ and $\X$ are the affine subspaces defined in \eqref{eq:X-X_b} for $\T$, and $\X'_b$ and $\X'$ are the corresponding affine subspaces for $\S$, then $\X_b = x_0 + \X'_{(b - Ax_0)}$ and $\X = x_0 + \X'$, which also implies $\dim{\X'_{(b-Ax_0)}} = \dim{\X_b}$, and $\dim{\X'} = \dim{\X}$. Thus from now on, unless specified otherwise, we will assume that $\T = \Gam{p}{\T} = \S$, and with this the expressions for $ x_{b,\omega,s}$ and $ x_{b,\omega}$ become
\begin{equation}
\label{eq:x_b_w_explicit_S}
    x_{b,\omega,s} = V \left( V^* A A_\omega^{s} A V\right)^{-1} V^* A A_\omega^{s} b, \text{ and } x_{b,\omega} = V ( V^* A A_\omega^{-1} A V)^{-1} V^* A A_\omega^{-1} b.
\end{equation}

The expression for $x_{b,\omega}$ will be studied in some detail in this paper, and so to make things easier we make the following definition:

\begin{definition}
\label{def:Dww'}
Using the notation and assumptions of \Cref{lemma:x-bw}, we define two maps $\mathbf{D}_{A} : (\omin, \infty) \rightarrow \Fpq{n}{n}$, and $\partial \mathbf{D}_{A} : (\omin, \infty) \times (\omin, \infty) \rightarrow \Fpq{n}{n}$ as
\begin{equation}
\label{eq:Dww'}
    \D{A}{\omega} = V ( V^* A A_\omega^{-1} A V)^{-1} V^* A A_\omega^{-1}, \;\; \dD{A}{\omega}{\mu} = \D{A}{\omega} - \D{A}{\mu}.
\end{equation}
For any $\omega, \mu \in (\omin, \infty)$, $\D{A}{\omega}$ and $\dD{A}{\omega}{\mu}$ represent linear maps $\Fn \rightarrow \S$, and is independent of the choice of $V$ (the proof of independence is essentially contained in the proof of \Cref{lemma:x-bw}).
\end{definition}

\subsubsection{Motivation}
\label{sssec:Motivation}

In order to gain some motivation about why we study the problem, we start with the following observation\footnote{Note that \Cref{lemma:solutions-invariant-S} holds for all $s \in \R$, even though we are only interested in the case $s=-1$.}, that holds under the assumptions mentioned above.

\begin{lemma}
\label{lemma:solutions-invariant-S}
If $\S$ is $A-$invariant, $x_{b,\omega,s}$ given by \eqref{eq:x_b_w_explicit_S} is independent of $\omega$ and $s$.
\end{lemma}

\begin{proof}
Since $A$ is invertible and $\Indsimple{A}{\S} = 0$, $A\S = \S$; so applying the spectral theorem to the restriction map $A|_\S$ (which is Hermitian because $A$ is), we can conclude that $\S$ is spanned by eigenvectors of $A$. Thus, we can choose $V$ such that $AV = V \Lambda$, where the columns of $V$ are eigenvectors of $A$, and $\Lambda$ is a diagonal matrix with real entries (the eigenvalues). Thus for any $\omega \in \R$, $A_\omega V = V (\Lambda + \omega I)$, while the definition of $A_\omega^s$ shows that $A^s_\omega V = V (\Lambda + \omega I)^s$ for all $\omega \in (\omin,\infty)$ and $s \in \R$, where the bounds on $\omega$ ensure that $A_\omega \in \Pos{n}$ so that $A_\omega^s$ is well-defined\footnote{Indeed if $A = UDU^\ast$ is the spectral decomposition of $A$ with $U \in U(n)$, $A^s_\omega = U (D+\omega I)^s U^\ast$ from definition, so if $u$ is any eigenvector of $A$ (not necessarily a column of $U$) such that $Au = \lambda u$, then $A^s_\omega u = (\lambda + \omega)^s u$.}. Plugging into \eqref{eq:x_b_w_explicit_S} gives $x_{b,\omega,s} = V \left( \Lambda (\Lambda + \omega I)^s \Lambda \right)^{-1} \Lambda (\Lambda + \omega I)^s V^* b = V \Lambda^{-1} V^\ast b$.
\end{proof}

\Cref{lemma:solutions-invariant-S} suggests that when $\Indsimple{A}{\S} = 0$, $\dim{\X_b} = 0$ for all $b \in \Fn$. A natural question that arises then is what happens when $\Indsimple{A}{\S} > 0$. As indicated in \Cref{sec:intro}, a simple consequence of the results in \cite{hallman2018lsmb} is that in the case $\F = \R$, if $B \in \R^{n \times n}$ a positive matrix, $b \in \R^n$, and $\S = \K_k(B,b)$ is a real Krylov subspace, the set of solutions $\{x_{b,\omega} \mid \omega \geq 0 \}$ with $x_{b,\omega}$ defined as the solution to \eqref{eq:spd-obj-krylov} belong to a $1$-dimensional affine subspace. Recall from \Cref{sssec:index} that for a Krylov subspace $\S = \K_k(A,b)$ that is not $A-$invariant, $\Indsimple{A}{\S} = 1$. Based on these two known results, we are faced with the possibility that the conjecture $\dim{\X_b} \leq \Indsimple{A}{\S}$ for all $b \in \Fn$, might be true for $\F = \C$ or $\R$.

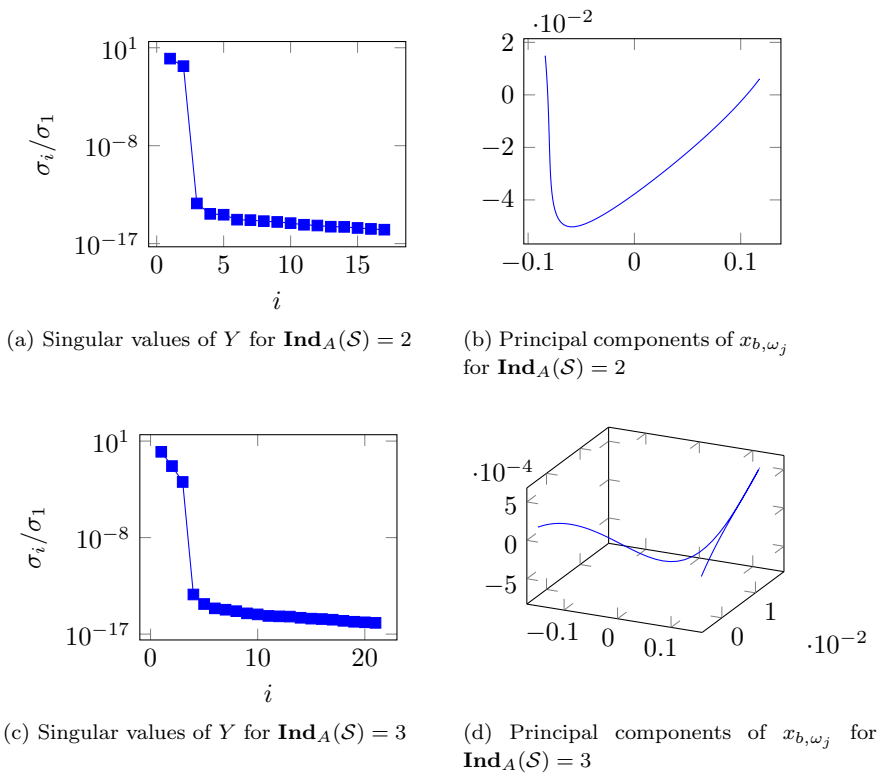
\begin{figure}[ht]
    \centering
    \begin{minipage}[b]{6cm}
        \subfloat[][Singular values of $Y$ for $\Indsimple{A}{\S}=2$]{
            \label{fig:svd_dim2}
            \begin{tikzpicture}
                \begin{semilogyaxis}[ylabel={$\sigma_i/\sigma_1$},xlabel={$i$},width=5cm]
                    \addplot[color=blue,mark=square*] table[x=i,y=s] {data/example_2d_svds_A.txt};
                \end{semilogyaxis}
            \end{tikzpicture}
        }
    \end{minipage}
    \begin{minipage}[b]{6cm}
        \subfloat[][Principal components of $x_{b,\omega_j}$ for $\Indsimple{A}{\S}=2$]{
            \label{fig:x_b_w_plot_dim2}
            \begin{tikzpicture}
                \begin{axis}[width=5cm,xlabel={\phantom{LOL}}]
                    \addplot[color=blue] table[x=x,y=y] {data/example_2d_coords_A.txt};
                \end{axis}
            \end{tikzpicture}
        }
    \end{minipage} \\
    \vspace{0.5cm}
    \begin{minipage}[b]{6cm}
        \subfloat[][Singular values of $Y$ for $\Indsimple{A}{\S}=3$]{
            \label{fig:svd_dim3}
            \begin{tikzpicture}
                \begin{semilogyaxis}[ylabel={$\sigma_i/\sigma_1$},xlabel={$i$},width=5cm]
                    \addplot[color=blue,mark=square*] table[x=i,y=s] {data/example_3d_svds_A.txt};
                \end{semilogyaxis}
            \end{tikzpicture}
        }
    \end{minipage}
    \begin{minipage}[b]{6cm}
        \subfloat[][Principal components of $x_{b,\omega_j}$ for $\Indsimple{A}{\S}=3$]{
            \label{fig:x_b_w_plot_dim3}
            \begin{tikzpicture}
                \begin{axis}[width=5cm,xlabel={\phantom{LOL}}]
                    \addplot3[no marks,blue] table[x=x,y=y,z=z] {data/example_3d_coords_A.txt};
                \end{axis}
            \end{tikzpicture}
        }
    \end{minipage}
\caption{Illustration of the low dimensionality of the affine subspace $\X_b = \Aff{\{x_{b,\omega} \mid \omega \geq 0\}}$. Left plots show the singular values of the centered matrix $Y$ computed as $Y_{ij} = X_{ij} - K^{-1} \sum_k X_{ik}$ where $X = [x_{b,\omega_1} \, \dots \, x_{b,\omega_K}]$. A sharp drop in the singular values indicates that the set $\{x_{b,\omega_j}\}_{j=1}^K$ lives in a low dimensional affine subspace. Right plots show the projection of $\{x_{b,\omega_j}\}_{j=1}^K$ over that low dimensional subspace.}
\label{fig:curves}
\end{figure}

We now describe a numerical experiment that also illustrates and confirms our intuition. Working over $\F = \R$, given a positive matrix $A \in \R^{N \times N}$ (built as the finite difference discretization with a 5-points stencil of a Poisson equation on a square domain, with $N = 529$), we create two experiments by building $\S$ as the sum of two (resp. three) real Krylov subspaces, i.e. $\K_{11}(A,b_1) + \K_{6}(A,b_2)$ (resp. $\K_{11}(A,c_1) + \K_{6}(A,c_2) + \K_{4}(A,c_3)$). The vectors $b_1, b_2, c_1, c_2, c_3 \in \Rp{N}$ were chosen as random Gaussian vectors, but such that $\Indsimple{A}{\S} = 2$ (resp. 3), and $\dim{\S} = 17$ (resp. 21). The vector $b \in \Rp{N}$ was also initialized as a random Gaussian vector.

To check the dimension of the solution set $\X_b$, we then build a matrix $X = \begin{bmatrix} x_{b,\omega_1} & \dots & x_{b,\omega_K} \end{bmatrix} \in \R^{N \times K}$, with the columns of $X$ computed using \eqref{eq:x_b_w_explicit_S} and $K=200$, and where $\omega_j = 10^{-3+6(j-1)/(K-1)}$, for all $1 \leq j \leq K$. We then perform a principal component analysis on $X$: we compute and subtract the mean across each $N$ dimensions, building $Y$ such that $Y_{ij} = X_{ij} - K^{-1} \sum_k X_{ik}$, for all $1 \leq i \leq N, \; 1 \leq j \leq K$. \Cref{fig:svd_dim2} (resp. \Cref{fig:svd_dim3}) shows the singular values of $Y$ in the $\Indsimple{A}{\S} = 2$ (resp. $\Indsimple{A}{\S} = 3$) cases. The sharp drop at the third (resp. fourth) singular value indicates that $Y$ is rank two (resp. three), which indicates that $\X_b$ may belong to a low dimensional affine subspace of dimension 2 (resp. 3). \Cref{fig:x_b_w_plot_dim2} (resp. \Cref{fig:x_b_w_plot_dim3}) shows the solution set $\{x_{b,\omega_j}\}_{j=1}^K$ projected over the leading two (resp. three) eigenvectors of $Y$ for the two experiments.

\subsubsection{A property of the minimization problem}
\label{sssec:minprob-prop}

It is worth noting a property of the minimization problem \eqref{eq:general-min-prob} that we now state, reminding the reader that we have already assumed that $\T = \Gam{p}{\T} = \S$. The result uses \Cref{lemma:index-bounds} proved in the next subsection, and the Pythagorean theorem: if $\mathcal{A} \in \G{n}$, $x \in \Fn$, and $y \in \mathcal{A}$ is the orthogonal projection of $x$ in $\mathcal{A}$, then $\norm{x}{2}^2 = \norm{y}{2}^2 + \norm{x - y}{2}^2$, and $z^\ast (x - y) = 0$ for all $z \in \mathcal{A}$; thus $x-y \in \mathcal{A}^\perp$.

Let $\S_1$ be the largest\footnote{Equivalently $\S_1$ is the sum of all $A-$invariant subspaces contained in $\S$.} $A-$invariant subspace such that $\S_1 \subseteq \S$, $\S_2$ be the smallest $A-$invariant subspace\footnote{Equivalently $\S_2$ is the intersection of all $A-$invariant subspaces containing $\S$.} such that $\S \subseteq \S_2$, and define $\S' := \S_1^\perp \cap \S$. Note that $\S_2$ always exists because $\Fn$ is $A-$invariant and contains $\S$, but $\S_1$ could be trivial. We thus have $\S_1 \subseteq \S \subseteq \S_2$, and $\S = \S_1 \oplus \S'$, a direct sum of orthogonal subspaces\footnote{$\S_1 \cap \S' = \{0\}$ because $\S_1 \cap \S_1^\perp = \{0\}$.}. The latter is true because if $x \in \S$, and $x_1$ is the orthogonal projection of $x$ on $\S_1$, then $x = x_1 + (x - x_1)$, with $x_1 \in \S_1$, and $x - x_1 \in \S_1^\perp$ by the Pythagorean theorem, and moreover $x - x_1 \in \S$ as both $x,x_1 \in \S$. Let $b \in \Fn$ be decomposed as $b = b_1 + b_2 + (b - b_1 - b_2)$, where $b_1$ is the orthogonal projection of $b$ on $\S_1$, and $b_2$ is the orthogonal projection of $b - b_1$ on $\S_2$; so again by the Pythagorean theorem $b - b_1 \in \S_1^\perp$, and $b - b_1 - b_2 \in \S_2^\perp$. Now consider the minimization problem \eqref{eq:general-min-prob}: $\argmin_{x \in \S} \|A_\omega^{s/2}(b - Ax) \|_2$. Writing $\S \ni x = y + z$, for $y \in \S_1$ and $z \in \S'$, and remembering that this representation is unique by the property of direct sums, we can equivalently express the minimization problem as $\argmin_{y \in \S_1,\; z \in \S'} \|A_\omega^{s/2}(b - Ay - Az) \|_2$. Next notice that 
\begin{equation}
\label{eq:orthogonal-sum-S1}
\begin{split}
    \|A_\omega^{s/2}(b - Ay - Az) \|_2^2 &= \|A_\omega^{s/2}(b_1 - Ay) + A_\omega^{s/2}(b - b_1 - Az) \|_2^2 \\
    &= \|A_\omega^{s/2}(b_1 - Ay) \|_2^2 + \|A_\omega^{s/2}(b - b_1 - Az) \|_2^2
\end{split}
\end{equation}
using the Pythagorean theorem. This is because $\S_1, \S_1^\perp$ are both $A,A_\omega^{s/2}-$invariant by \Cref{lemma:index-bounds}(iv),(v) (this uses $A_\omega^{s/2} \in \Pos{n}$) --- so as both $b_1, y \in \S_1$, we have $A_\omega^{s/2}(b_1 - Ay) \in \S_1$; similarly both $b-b_1,z \in \S_1^\perp$ implies $A_\omega^{s/2}(b - b_1 - Az) \in \S_1^\perp$, and \eqref{eq:orthogonal-sum-S1} follows. A final simplification happens by noticing that $\S' \subseteq \S_2$, and since $\S_2, \S_2^\perp$ are both $A,A_\omega^{s/2}-$invariant (again by \Cref{lemma:index-bounds}(iv),(v)), we have $A_\omega^{s/2} (b - b_1 - b_2) \in \S_2^\perp$ and $A_\omega^{s/2}(b_2 - Az) \in \S_2$, and so by another application of the Pythagorean theorem
\begin{equation}
\label{eq:orthogonal-sum-S2}
\begin{split}
    \|A_\omega^{s/2}(b - b_1 - Az) \|_2^2 &= \|A_\omega^{s/2}(b - b_1 - b_2) + A_\omega^{s/2}(b_2 - Az) \|_2^2 \\
    &= \|A_\omega^{s/2}(b - b_1 - b_2) \|_2^2 + \|A_\omega^{s/2}(b_2 - Az) \|_2^2.
\end{split}
\end{equation}
Thus, we have decoupled the variables $y$ and $z$, into two separate minimization problems, which can be solved independently, and we have proved
\begin{lemma}
\label{lemma:soln-sum-simpler-probs}
The solution to the minimization problem \eqref{eq:general-min-prob} satisfies the identity
\begin{equation}
\label{eq:soln-sum-simpler-probs}
    x_{b,\omega,s} = \argmin_{y \in \S_1} \|A_\omega^{s/2}(b_1 - Ay) \|_2 + \argmin_{z \in \S'} \|A_\omega^{s/2}(b_2 - Az) \|_2.
\end{equation}
\end{lemma}

\Cref{lemma:soln-sum-simpler-probs} allows us to get an upper bound on $\dim{\X_b}$, and already gives the first hints that $\X_b$ is a low dimensional affine subspace. This is stated in the next corollary.

\begin{corollary}
\label{cor:weak-bound}
$\dim{\X_b} \leq \dim{\S'}$, for all $b \in \Fn$.
\end{corollary}

\begin{proof}
By \Cref{lemma:solutions-invariant-S}, the first term in the right hand side of \eqref{eq:soln-sum-simpler-probs} is independent of $\omega, s$, and so is a fixed point $y_b \in \S_1$ for a given $b \in \Fn$; the second term always is in $\S'$. Thus $x_{b,\omega,s} \in y_b + \S'$ for all $\omega \in (\omin, \infty)$, and $s \in \mathbb{R}$, and so $\X_b \subseteq y_b + \S'$. The conclusion follows as $b$ is arbitrary.
\end{proof}

One should note that $\Indsimple{A}{\S} \leq \dim{\S'}$, because $\S + A\S = \S + A\S_1 + A\S' = \S + \S_1 + A\S' = \S + A\S'$, and so $\Indsimple{A}{\S} = \dim{\S + A\S} - \dim{\S} \leq \dim{A\S'} \leq \dim{\S'}$. It turns out that because of this reason the bound provided by \Cref{cor:weak-bound} is weak, which will be strengthened in \Cref{sec:main-result}.

\begin{remark}
Indeed for a Krylov subspace $\S = \K_p(A,b)$ that is not $A-$invariant, $\S_1 = \{0\}$\footnote{If $\S_1$ was not trivial, it must have an eigenvector $x \neq 0$ satisfying $Ax = \lambda x$, $\lambda \neq 0$, as $A$ is invertible. Expanding $x$ in the Krylov basis as $x = \sum_{i=0}^{p-1} c_i A^i b$, and using that $\{A^i b\}_{i=0}^{p}$ is linearly independent because $\K_p(A,b)$ is not $A-$invariant, gives $c_i = 0$ for all $0 \leq i \leq p-1$.}, and so $\S' = \S$ and the bound gives $\dim{\X_b} \leq p$, while as we have already mentioned, we know from \cite{hallman2018lsmb} that $\dim{\X_b} \leq 1$ when $\F=\R$. On the other hand, for invariant subspaces the bound is tight as $\S_1 = \S$, so $\dim{\X_b} \leq 0$.
\end{remark}

\subsection{Properties of the index of invariance}
\label{ssec:properties-index}

We now prove some facts about the index of invariance, defined previously in \Cref{sssec:index}. This subsection is self-contained, and the assumptions established in \Cref{ssec:problem-statement} will not be assumed here, but we assume $n \geq 1$. We start with two lemmas that characterize the relationship between the index of invariance and bases of the subspaces involved in its definition.

\begin{lemma} 
\label{lemma:basic_properties}
Let $A \in \Fnn$, $\S \in \Gr{p}{n}$ and $q = \Indsimple{A}{\S}$.
Then
\begin{enumerate}[(i)]
\item $\Indsimple{A}{\S} \leq \min \{ \dim{\S}, \; n - \dim{\S} \} \leq \lfloor n/2 \rfloor$.
\item If $p \geq 1$, there exists semi-unitary $V \in \Fpq{n}{p}$ such that $\Img{V} = \S$, and when $q \geq 1$ also, there exists $V' \in \Fpq{n}{q}$ such that $\begin{bmatrix} V & V' \end{bmatrix}$ is semi-unitary, $\Img{\begin{bmatrix} V & V' \end{bmatrix}} = \S + A\S$, and $\Img{V'} = \S^\perp \cap (\S + A\S)$.
\end{enumerate}
\end{lemma}

\begin{proof}
\begin{enumerate}[(i)]
\item Since $\dim{A\S}~\leq~\dim{\S}$, $\dim{\S + A\S} \leq 2 \dim{\S}$, and so $\Indsimple{A}{\S} = \dim{\S+A\S} - \dim{\S} \leq \dim{\S}$. Furthermore, since $\dim{\S + A\S} \leq n$, $\Indsimple{A}{\S} \leq n - \dim{\S}$. We conclude by noting that $\lfloor n/2 \rfloor \geq \min\{\dim{\S},\; n - \dim{\S}\} \in \mathbb{N}$.

\item Since $\S$ is of dimension $p$, the existence of $V$ follows from using the Gram-Schmidt process on any basis of $\S$. Now assume $q \geq 1$. Since $\dim{\S + A\S} = \dim{\S} + q$, one can find $q$ independent vectors $\{x_i\}_{i=1}^q$ in $\S+A\S$ not in $\S$, and let $X = \begin{bmatrix} x_{1} & \dots & x_{q} \end{bmatrix} \in \Fpq{n}{q}$. Then, applying the Gram Schmidt process to $\begin{bmatrix} V & X \end{bmatrix}$ gives the semi-unitary matrix $\begin{bmatrix}V & V'\end{bmatrix}$. The columns of $V'$ are orthogonal to $\S$ because $\begin{bmatrix}V & V'\end{bmatrix}$ is semi-unitary, so $\Img{V'} \subseteq \S^\perp \cap (\S + A\S)$. Also $\S + A\S = \S \oplus (\S^\perp \cap (\S + A\S))$, thus $\dim{\S^\perp \cap (\S + A\S)} = q = \dim{\Img{V'}}$, so in fact $\Img{V'} = \S^\perp \cap (\S + A\S)$.
\end{enumerate}
\end{proof}

\begin{lemma} 
\label{lemma:rank_B}
Let $A \in \Fnn$ be any operator, and $\S \in \Gr{p}{n}$ for $p \geq 1$. Let $\begin{bmatrix} V & V' \end{bmatrix}$ be semi-unitary such that $\Img{V} = \S$, and $\Img{\begin{bmatrix} V & V' \end{bmatrix}} = \S + A\S$. Then $\Indsimple{A}{\S} = 0$ if and only if there exist $T \in \Fpq{p}{p}$ such that $A V = V T$. Otherwise the following are equivalent:
    \begin{enumerate}[(i)]
        \item $\Indsimple{A}{\S} = q \geq 1$.
        \item There exist $T \in \Fpq{p}{p}$, $B \in \Fpq{q}{p}$ and $\rank(B)=q$, such that $A V = V T + V' B$, and $T,B$ are uniquely determined by $A, V, V'$.
    \end{enumerate}
\end{lemma}

\begin{proof} Notice that from \Cref{lemma:basic_properties}(ii), $V$ and $V'$ always exist (the latter only existing when $q \geq 1$). Let $\Indsimple{A}{\S} = q$. The $q=0$ case is clear, so assume $q \geq 1$. We first prove (i)$\rightarrow$(ii). Since $\Img{AV} = A\S \subseteq \S + A\S = \Img{\begin{bmatrix} V & V'\end{bmatrix}}$, we have $AV = VT + V'B$, for some $T \in \mathbb{F}^{p \times p}$ and $B \in \mathbb{F}^{q \times p}$, which are uniquely determined because $\begin{bmatrix} V & V'\end{bmatrix}$ is full rank. From \Cref{lemma:basic_properties}(i) we have $q \leq p$.
Now assume $B$ is not of full rank $q$. Then one can decompose $B$ (such as using the singular value decomposition) as $B = UW$, where $U \in \mathbb{F}^{q \times r},\; W \in \mathbb{F}^{r \times p}$ with $r < q$. Then $AV = VT + (V'U)W$ from which it follows that $\Img{AV} \subseteq \Img{\begin{bmatrix}V & V'U \end{bmatrix}}$,
where $\rank(V'U) \leq r < q$. But $\S=\Img{V}$, and so $\S+A\S \subseteq \Img{\begin{bmatrix}V & V'U \end{bmatrix}} = \S + \Img{V'U}$.
This implies that
$\dim{\S + A\S} \leq \dim{\S} + r$,
which is a contradiction.

Now suppose (ii) holds. Since $\S = \Img{V}$, and $A\S = \Img{AV}$, by assumption it follows that $\S + A\S = \{ Vx + (VT + V'B)y \mid x, y \in \Fp{p} \}
= \{ V(x + Ty) + V' B y \mid x, y \in \Fp{p} \}
= \{ Vw + V'z \mid w \in \Fp{p},\; z \in \Fp{q}\}
= \Img{\begin{bmatrix} V & V' \end{bmatrix}}$ (the second last equality follows because $B$ is full rank). Since $\begin{bmatrix}V & V' \end{bmatrix}$ is semi-unitary, we conclude that $\dim{\S+A\S} = p+q = \dim{\S}+q$.
\end{proof}

\begin{remark}
It should be noted that \Cref{lemma:basic_properties,lemma:rank_B} are also true when the semi-unitarity condition of $V$ and $\begin{bmatrix}V & V' \end{bmatrix}$ is replaced by the condition that $V$ and $\begin{bmatrix}V & V' \end{bmatrix}$ are full rank.
\end{remark}

\Cref{lemma:rank_B} has an important consequence that we state next, which will play a key role later in the proof of the main theorem of this paper.

\begin{corollary}
\label{cor:tridiag-blk-decomp}
Let $A \in \Fnn$, $\S \in \Gr{p}{n}$, and $\Indsimple{A}{\S} = q$. Let $V \in \Fpq{n}{p}$, $V' \in \Fpq{n}{q}$, and $V'' \in \Fpq{n}{(n-p-q)}$ be such that $\begin{bmatrix} V & V' & V'' \end{bmatrix}$ is unitary, $\S = \Img{V}$, and $\S + A\S = \Img{\begin{bmatrix}V & V' \end{bmatrix}}$. Then $A$ has the following block decomposition
\vspace*{-0.4cm}
\begin{equation}
\begin{aligned}
\label{eq:property_VTAV} 
    \begin{tikzpicture}
    \node[] at (1.2, 1) {$p$};
    \node[] at (1.9, 1) {$q$};
    \node[] at (3.7, 0.4) {$p$};
    \node[] at (3.7, 0) {$q$};
    \node[] {
    $\begin{bmatrix} 
        V^* \\ V'^* \\ V''^* 
    \end{bmatrix} 
        A 
    \begin{bmatrix} 
        V & V' & V'' 
    \end{bmatrix} 
    = 
    \left[
    \begin{array}{c|c|c} 
        T & P & Q \\ \hline
        B & C & R \\ \hline
        0 & D & E 
    \end{array}
    \right],
    $
    };
    \end{tikzpicture}
\end{aligned}
\vspace*{-0.3cm}
\end{equation}
where $T \in \Fpq{p}{p}$, $C \in \Fpq{q}{q}$, $E \in \Fpq{(n-p-q)}{(n-p-q)}$, with the shapes of the other blocks being compatible. If $p=0$, or $q=0$, or $n-p-q=0$, then \eqref{eq:property_VTAV} holds with the non-existent blocks and the corresponding non-existent $V,V',V''$ removed. If $q \geq 1$, then $B$ is of full rank $q$. Additionally
\begin{enumerate}[(i)]
    \item If $A \in \Gl{n}$ and $p \geq 1$, then $H :=\begin{bmatrix} T \\ B \end{bmatrix}$ is of full rank $p$.
    
    \item If $A \in \Sym{n}$, one has $P = B^\ast$, $R = D^\ast$, $Q = 0$, and $T,C,E$ Hermitian. 
\end{enumerate}
\end{corollary}

\begin{proof}
The decomposition follows from \Cref{lemma:rank_B} (which also gives $\rank(B) = q$), by noting that $V''^*(AV) = 0$, using the unitarity of $\begin{bmatrix} V & V' & V'' \end{bmatrix}$ and $AV = VT + V'B$. For (i), let $\bar A$ be the right-hand-side of \eqref{eq:property_VTAV}; so $A \in \Gl{n}$ implies $\bar A \in \Gl{n}$, which means the first $p$ columns of $\bar A$ are linearly independent. But if $\rank(H) < p$, the first $p$ columns of $\bar A$ are linearly dependent, giving a contradiction. For (ii), note that when $A \in \Sym{n}$, both sides of \eqref{eq:property_VTAV} are Hermitian, and so the conclusion follows.
\end{proof}

\begin{remark}
When $A \in \Sym{n}$, the decomposition given by \Cref{cor:tridiag-blk-decomp} will be called the \textit{tridiagonal block decomposition}. Note that (i) this decomposition exists regardless of whether $A$ is invertible or positive, (ii) even if $A$ is invertible, the diagonal blocks $T$, $C$ and $E$ need not be, (iii) if however $A \in \Pos{n}$, then $T$, $C$ and $E$ are in fact positive, but $D$ need not be full rank\footnote{For example, if $\Indsimple{A}{\S + A\S} = 0$ (i.e., $\S + A\S$ is invariant), then $D = 0$.}. We also note that this decomposition is similar to the block Lanczos decomposition \cite[page 567]{golub2013matrix} as used in block Krylov methods (amongst many, \cite{o1980block, sadkane1993block1, sadkane1993block2, dietl2007linear}).
\end{remark}

It is worth noting some special cases. Consider the case when $\S$ is $A-$invariant, and $A \in \Sym{n}$. Then $\Indsimple{A}{\S} = 0$, and so in the tridiagonal block decomposition \eqref{eq:property_VTAV}, $V'$ has $0$ columns (i.e., $q=0$), and we can simply write
\begin{equation} 
\label{eq:property_invariant} 
\begin{bmatrix} V^* \\ V''^* 
\end{bmatrix} A 
\begin{bmatrix} V & V'' 
\end{bmatrix} = 
\begin{bmatrix} T & 0 \\ 0 & E 
\end{bmatrix}. 
\end{equation}
This is the block-diagonal Schur decomposition of a Hermitian matrix for a given invariant subspace \cite[page 443]{golub2013matrix}. Consider similarly the case when $\S = \K_p(A,b)$, such that $\S$ is not $A-$invariant, and so $\Indsimple{A}{\S} = 1$. We then know that $B \in \Fpq{1}{p}$ is rank-1. In fact, if we build $V$ by the Arnoldi process (that is the first $k$ columns of $V$ span $\K_k(A,b)$ for $1 \leq k \leq p$), then $B = V'^* A V = \beta e_p^\ast$, where $\beta \neq 0$ with $(e_p)_i = 0$ for $i < p$ and $(e_p)_p = 1$.

The next three results build upon \Cref{cor:tridiag-blk-decomp}.

\begin{lemma}
\label{lemma:rank-S-perp}
Let $A \in \Fnn$, $\S \in \Gr{p}{n}$, and $\Indsimple{A}{\S} = q$, with $1 \leq p 
< n$. Let $\begin{bmatrix}V_1 & V_2 \end{bmatrix}$ be invertible, such that $\Img{V_1} = \S$, and $\Img{V_2} = \S^\perp$. Then $AV_1 = V_1S_1 + V_2S_2$ for unique $S_1 \in \Fpq{p}{p}, S_2 \in \Fpq{(n-p)}{p}$, and $\rank(S_2) = q$.
\end{lemma}

\begin{remark}
Note that since there always exist $S_1$ and $S_2$ such that $AV_1 = V_1 S_1 + V_2 S_2$, \Cref{lemma:rank-S-perp} is necessary and sufficient: if $\rank(S_2) = q$, $\Indsimple{A}{\S} = q$.
\end{remark} 

\begin{proof}
Existence and uniqueness of $S_1$ and $S_2$, such that $AV_1 = V_1S_1 + V_2S_2$, follows from the invertibility of $\begin{bmatrix}V_1 & V_2 \end{bmatrix}$, as the columns form a basis of $\Fn$. Now $A$ has the decomposition \eqref{eq:property_VTAV} by  \Cref{cor:tridiag-blk-decomp}, where $\Img{V} = \S$, and since $\begin{bmatrix}V & V' & V'' \end{bmatrix}$ is unitary, we also have $\Img{\begin{bmatrix}V' & V'' \end{bmatrix}} = \S^\perp$. Thus there exist $M_1 \in \Gl{p}$ and $M_2 \in \Gl{n-p}$, such that $V_1 = V M_1$ and $V_2 = \begin{bmatrix}V' & V'' \end{bmatrix} M_2$, and so we have $A V  = V M_1 S_1 M_1^{-1} + \begin{bmatrix}V' & V'' \end{bmatrix} M_2 S_2 M_1^{-1}$. But $AV = VT + \begin{bmatrix}V' & V'' \end{bmatrix} \begin{bmatrix}B \\ 0 \end{bmatrix}$ also, from which it follows that 
\vspace*{-0.3cm}
\begin{equation}
    T = M_1 S_1 M_1^{-1}, \text{ and } M_2 S_2 M_1^{-1} = \begin{bmatrix}B \\ 0 \end{bmatrix}.
\end{equation}
The latter gives that $\rank(S_2) = \rank(B) = q$, as $M_1, M_2$ are invertible.
\end{proof}

\begin{corollary}
\label{cor:index-nonincreasing}
Let $A \in \Fnn$, and $\S$ be a subspace. Define the nested sequence of subspaces $\S_0 \subseteq \dots \subseteq \S_i \subseteq \S_{i+1} \subseteq \dots$, as $\S_0 = \S$, and $\S_{i+1} = \S_i + A\S_i$. Then $\Indsimple{A}{\S_i} \geq \Indsimple{A}{\S_{i+1}}$ for all $i \geq  0$, and there exists $j \geq 0$ such that $\Indsimple{A}{\S_j} = 0$.
\end{corollary}

\begin{proof}
If $\Indsimple{A}{\S_0} = 0$, then $\S_i = \S_0$ for all $i \geq 0$, and the statement follows. Now assume $\Indsimple{A}{\S_0} \geq 1$. Let us just show that $\Indsimple{A}{\S_0} \geq \Indsimple{A}{\S_{1}}$; repeated application of the same argument proves that the sequence $\{\Indsimple{A}{\S_i}\}_{i=0}^\infty$ is non-increasing. If $\S_1 = \Fn$ we are again done as $\S_i = \S_1$ for all $i \geq 1$, so assume this is not the case. Consider the decomposition of $A$ in \eqref{eq:property_VTAV}, from which we have $\Img{\begin{bmatrix}V & V' \end{bmatrix}} = \S + A\S$, and $\Img{V''} = (\S+A\S)^\perp$; thus defining $V_1 := \begin{bmatrix}V & V' \end{bmatrix}$ and $V_2 := V''$ we obtain $AV_1 = V_1 S_1 + V_2 S_2$, with $S_2 = \begin{bmatrix}0 & D \end{bmatrix}$ (and $S_1$ similarly determined by \eqref{eq:property_VTAV}). Now $\rank(S_2) = \rank(D) \leq \Indsimple{A}{\S}$, and thus applying \Cref{lemma:rank-S-perp} gives $\Indsimple{A}{\S + A\S} \leq \Indsimple{A}{\S}$. To prove that there exists $j \geq 0$ such that $\Indsimple{A}{\S_j} = 0$, notice that if this was false then there would exist $k \geq 0$ such that $\dim{\S_k} > n$, which would give a contradiction.
\end{proof}

\begin{lemma}
\label{lemma:index-bounds}
Let $A \in \Fnn$, $\S \in \Gr{p}{n}$, and $\Indsimple{A}{\S} = q$. Let us also define $\S' := \S^\perp \cap (\S + A\S)$. Then we have the following.
\begin{enumerate}[(i)]
    \item If $A_\omega := A + \omega I$ for $\omega \in \F$, then $\Indsimple{A_\omega}{\S} = q$.
    
    \item $\Indsimple{A}{\S^\perp} \leq \min \{p, \; n-p\} \leq \lfloor n/2 \rfloor$, $\Indsimple{A}{\S + A\S} \leq \min \{q, \; n-p-q\} \leq \lfloor (n-p)/2 \rfloor$, and $\Indsimple{A}{\S'} \leq q$.
    
    \item If $A \in \Gl{n}$, then $\Indsimple{A^{-1}}{\S} = q$. 
    
    \item If $A \in \Sym{n}$, then $\Indsimple{A}{\S^\perp} = q$, and $\Indsimple{A}{\S'} = q$. Thus if $q = 0$, both $\S$ and $\S^\perp$ are $A-$invariant, and if $A \in \Sym{n} \cap \Gl{n}$, both $\S$ and $\S^\perp$ are also $A^{-1}-$invariant\footnote{The fact that $\S$ being $A-$invariant implies $\S^\perp$ is $A-$invariant for Hermitian $A$ is well known.}.
    
    \item If $A \in \Pos{n}$ and $\Indsimple{A}{\S} = 0$, then for any $s \in \R$, $\Indsimple{A^s}{\S} = 0$.
\end{enumerate}
\end{lemma}

\begin{proof}
\begin{enumerate}[(i)]
    \item This follows because $\S + A_\omega \S = \S + A\S$.
    
    \item $\Indsimple{A}{\S^\perp} \leq \min \{p, \; n-p\} \leq \lfloor n/2\rfloor$ follows by applying \Cref{lemma:basic_properties}(i) to $\S^\perp$, and noticing that $\dim{\S^\perp} = n - p$. 
    $\Indsimple{A}{\S + A\S} \leq q$ was proved in \Cref{cor:index-nonincreasing}. Applying \Cref{lemma:basic_properties}(i) to $\S + A\S$ gives $\Indsimple{A}{\S + A\S} \leq \min \{p+q, \; n-p-q\}$, as $\dim{\S + A\S} = p+q$; so combining gives $\Indsimple{A}{\S + A\S} \leq \min \{q, \; p+q, \; n-p-q\} = \min \{q, \; n-p-q\}$. Finally $\min \{q, \; n-p-q\} \leq \lfloor (n-p)/2 \rfloor$.
    $\Indsimple{A}{\S'} \leq q$ follows from \Cref{lemma:basic_properties}(i): $\Indsimple{A}{\S'} \leq \dim{\S'} = \Indsimple{A}{\S} = q$.
    
    \item The $p=0$ case is clear, so assume $p \geq 1$. Denote by $\hat A$ the right-hand side of \eqref{eq:property_VTAV}. Since $A \in \Gl{n}$, $\hat{A} \in \Gl{n}$ and we have $A^{-1} \begin{bmatrix} V & V' & V''\end{bmatrix} = \begin{bmatrix} V & V' & V'' \end{bmatrix} \hat{A}^{-1}$. We use the subscript $1$ (resp. $2$) to denote the first $p$ (resp. last $n-p$) rows or columns.  From the nullity theorem (Theorem 2.1 in \cite{strang2004interplay}), $\nullity{(\hat{A}^{-1})_{21}} = \nullity{\hat{A}_{21}}$, and so $\rank{(\hat{A}^{-1})_{21}} = \rank{(\hat{A}_{21})} = q$. We then have $A^{-1} V = V (\hat{A}^{-1})_{11} + \begin{bmatrix} V' & V'' \end{bmatrix} (\hat{A}^{-1})_{21}$, and using \Cref{lemma:rank-S-perp} we conclude $\Indsimple{A^{-1}}{\S} = q$.
    
    \item Assuming $A \in \Sym{n}$, \eqref{eq:property_VTAV} gives $P = B^\ast$, $Q = 0$, $\Img{\begin{bmatrix}V' & V''\end{bmatrix}} = \S^\perp$, $\S' = \Img{V'}$, and $\S'^\perp = \Img{\begin{bmatrix}V & V'\end{bmatrix}}$. We also have $A\begin{bmatrix}V' & V''\end{bmatrix} = \begin{bmatrix}V' & V''\end{bmatrix} S_1 + VS_2$, and $AV' = V'\widetilde{S}_1 + \begin{bmatrix}V & V'\end{bmatrix} \widetilde{S}_2$, with $S_1, S_2, \widetilde{S}_1, \widetilde{S}_2$ determined by \eqref{eq:property_VTAV}. In particular $S_2 = \begin{bmatrix}P & 0\end{bmatrix}$ and $\widetilde{S}_2 = \begin{bmatrix} P \\ D \end{bmatrix}$, and note that $\rank(P) = q$, by \Cref{lemma:rank_B}. Now $\rank(S_2) = \rank(P)$ trivially, while $\rank(\widetilde{S}_2) = q$ as $\rank(\widetilde{S}_2) \geq \rank(P)$, and also $\rank(\widetilde{S}_2) \leq q$ since $\widetilde{S}_2 \in \Fpq{(n-q)}{q}$. So by \Cref{lemma:rank-S-perp}  $\Indsimple{A}{\S^\perp} = \Indsimple{A}{\S'} = q$. Finally by (iii), if $A \in \Sym{n} \cap \Gl{n}$ and $q = 0$, then $\Indsimple{A^{-1}}{\S} = \Indsimple{A^{-1}}{\S^\perp} = 0$.
    
    \item Note that from assumptions, $A\S = \S$, using both invertibility of $A$ and $\Indsimple{A}{\S} = 0$. By an argument similar to that already used in \Cref{lemma:solutions-invariant-S} we see that $A^s \S = \S$ also (since $\S$ is spanned by eigenvectors of $A$, which are also eigenvectors of $A^s$), and the conclusion follows.
\end{enumerate}
\end{proof}

The next lemma shows that the index of invariance is subadditive in both its arguments.
\begin{lemma}[Subadditivity]
\label{lemma:subadditivity}
Let $A, B \in \Fnn$, and $\S, \S' \in \G{n}$. Then
\begin{enumerate}[(i)]
    \item $\Ind{A}{\S + \S'} \leq \Ind{A}{\S} + \Ind{A}{\S'}$.
    
    \item $\Ind{A+B}{\S} \leq \Ind{A}{\S} + \Ind{B}{\S}$.
    
    \item $\Ind{AB}{\S} \leq \Ind{A}{\S} + \Ind{B}{\S}$.
\end{enumerate}
\end{lemma}

\begin{proof}
In this proof we will use the fact that if $\mathcal{A} \in \G{n}$ and $T \in \Fnn$, then $\Indsimple{T}{\mathcal{A}} = \dim{\mathcal{A}^\perp \cap (\mathcal{A} + T\mathcal{A})}$, by \Cref{lemma:basic_properties}(ii).

\begin{enumerate}[(i)]
    \item Let $\mathcal{Q} = \S + \S'$. Notice that $\S + A\S = \S \oplus (\S^\perp \cap (\S + A\S))$, and similarly $\S' + A\S' = \S' \oplus (\S'^\perp \cap (\S' + A\S'))$; so adding gives $\mathcal{Q} + A\mathcal{Q} = \S + \S' + (\S^\perp \cap (\S + A\S)) + (\S'^\perp \cap (\S' + A\S'))$. It follows that $\dim{\mathcal{Q} + A\mathcal{Q}} \leq \dim{\mathcal{Q}} + \Indsimple{A}{\S} + \Indsimple{A}{\S'}$.
    
    \item We again have $\S + (A+B)\S = \S + A\S + \S + B\S = \S + (\S^\perp \cap (\S + A\S)) + (\S^\perp \cap (\S + B\S))$. Thus $\dim{\S + (A+B)\S} \leq \dim{\S} + \Indsimple{A}{\S} + \Indsimple{B}{\S}$.
    
    \item We have $\S + AB \S \subseteq \S + A(\S + B\S) = \S + A(\S \oplus (\S^\perp \cap (\S + B\S))) = \S + A\S + A(\S^\perp \cap (\S + B\S)) = \S \oplus (\S^\perp \cap (\S + A\S)) + A(\S^\perp \cap (\S + B\S))$. Now $\dim{A(\S^\perp \cap (\S + B\S))} \leq \dim{\S^\perp \cap (\S + B\S)} = \Indsimple{B}{\S}$, and so we have $\dim{\S + AB\S} \leq \dim{\S} + \Indsimple{A}{\S} + \Indsimple{B}{\S}$.
\end{enumerate}
\end{proof}

We now return to the central question of the paper, which is to provide tighter bounds on the dimensions of the affine subspace $\X_b$ and the subspace $\X$, introduced in \eqref{eq:X-X_b}. 
\section{Proof of the main result}
\label{sec:main-result}

The goal of this section is to prove \Cref{thm:main_result}. We will work under the assumptions established in \Cref{ssec:problem-statement}, so we briefly remind the reader that we are working over an arbitrary field $\F = \C \text{ or } \R$,  $A \in \Sym{n} \cap \Gl{n}$, $b \in \Fn$, and $\S \in \Gr{p}{n}$ for $1 \leq p \leq n$ (more precisely it was shown in \Cref{ssec:problem-statement} that it suffices to only consider subspaces, so in fact we have assumed that $\T = \Gam{p}{\T} = \S$). We have also assumed that $V \in \Fpq{n}{p}$ is semi-unitary, such that $\Img{V} = \S$, and we have defined $\omin := -\lambda_{\text{min}}(A)$, and $A_\omega := A + \omega I$. Finally, we are interested in solutions $x_{b,\omega}$ to problem \eqref{eq:general-min-prob}, with $s=-1$. 

Let $\Indsimple{A}{\S} = q$. Notice that if $q = 0$, then the statement of \Cref{thm:main_result} already follows by \Cref{lemma:solutions-invariant-S}, because by the lemma $x_{b,\omega} - x_{b,\mu} = 0$, for all $\omega, \mu > \omin$, and $b \in \Fn$; so $\mathcal{Y} = \{0\}$. Thus for the proof of \Cref{thm:main_result} we assume $q \geq 1$. Now there are two cases: $n = p+q$, and $n > p+q$. In \Cref{app:appB}, we reduce the proof of \Cref{thm:main_result} in the $n = p+q$ case, to the case where $n > p+q$; thus we can further assume for the proof, without loss of generality, that $n > p+q$. Then using the tridiagonal block decomposition (\Cref{cor:tridiag-blk-decomp}), we will choose $V' \in \Fpq{n}{q}$, and $V'' \in \Fpq{n}{(n-p-q)}$ such that $\begin{bmatrix} V & V' & V'' \end{bmatrix}$ is unitary, $\S + A\S = \Img{\begin{bmatrix}V & V' \end{bmatrix}}$, and
\vspace*{-0.1cm}
\begin{equation}
\label{eq:tridiag-decomp-hermitian}
    \begin{bmatrix} V^* \\ V'^* \\ V''^* \end{bmatrix} A \begin{bmatrix} V & V' & V'' \end{bmatrix} = \begin{bmatrix} T & B^* & 0 \\ B & C & D^* \\ 0 & D & E \end{bmatrix},
\end{equation}
where $T \in \Fpq{p}{p}$, $C \in \Fpq{q}{q}$, $E \in \Fpq{(n-p-q)}{(n-p-q)}$ are all Hermitian, and the shapes of the other blocks are compatible, and we denote $H := \begin{bmatrix} T \\ B \end{bmatrix} \in \Fpq{(p+q)}{p} $ which is of full rank $p$. We let $b = Vc + V' c' + V'' c''$, for some $c \in \Fp{p}$, $c' \in \Fp{q}$, and $c'' \in \Fp{n-p-q}$, the representation being unique for the given choice of $V, V'$, and $V''$, and existing for any $b \in \Fn$, because $\Img{\begin{bmatrix} V & V' & V'' \end{bmatrix}} = \Fn$. 

To simplify the presentation of this section, we make a few observations. Using \eqref{eq:tridiag-decomp-hermitian} and the unitarity of $\begin{bmatrix} V & V' & V'' \end{bmatrix}$, we obtain
\vspace*{-0.1cm}
\begin{equation}
\label{eq:tridiag-decomp-hermitian1}
    \begin{bmatrix} V^* \\ V'^* \\ V''^* \end{bmatrix} A_\omega \begin{bmatrix} V & V' & V'' \end{bmatrix} = \begin{bmatrix} T + \omega I & B^* & 0 \\ B & C + \omega I & D^* \\ 0 & D & E + \omega I \end{bmatrix},
\end{equation}
and since $A_\omega \in \Pos{n}$ for $\omega > \omin$, the right-hand side of \eqref{eq:tridiag-decomp-hermitian1} is also positive. Thus in particular $E+\omega I$ is positive, which allows us to define $F_\omega \in \Fpq{q}{q}$ and $G_\omega \in \Fpq{(p+q)}{(p+q)}$ for any $\omega \in (\omin, \infty)$, as follows
\vspace*{-0.1cm}
\begin{equation}
\label{eq:F-G}
    F_\omega := D^\ast (E + \omega I)^{-1} D, \text{ and } G_\omega := \begin{bmatrix} T & B^\ast \\ B & C - F_{\omega} \end{bmatrix} + \omega I.
\end{equation}
The positivity of $E + \omega I$ directly ensures that $F_\omega \in \Pos{q}$, while $G_\omega \in \Pos{p+q}$ as it is the Schur complement of the $E + \omega I$ block of the right-hand side of \eqref{eq:tridiag-decomp-hermitian1}. Finally, we note a couple of key identities that follow from the 2-by-2 block matrix inversion formula \cite{lu2002inverses}, whenever $\omega > \omin$:
\vspace*{-0.3cm}
\begin{equation} 
\label{eq:block_inverse} 
\begin{split}
G_\omega^{-1} &= \begin{bmatrix} V^* \\ V'^* \end{bmatrix} A_\omega^{-1} 
\begin{bmatrix} V & V' \end{bmatrix}, \\
- G_\omega^{-1} \begin{bmatrix} 0 \\ D^\ast (E + \omega I)^{-1} \end{bmatrix} &= \begin{bmatrix} V^* \\ V'^* \end{bmatrix} A_\omega^{-1} V''.
\end{split}
\end{equation}
We are now ready to prove the following lemma, which is the first step in proving \Cref{thm:main_result}.

\begin{lemma}
\label{lemma:HTSH}
Define $d_{b,\omega,\mu} := V^*(x_{b,\omega} - x_{b,\mu}) \in \Fp{p}$, whenever $\omega, \mu \in (\omin, \infty)$, and let $N \in \Fpq{(p+q)}{q}$ be any full rank matrix\footnote{Existence of $N$ is guaranteed as $\rank(H^\ast) = \rank(H) = p$, hence the nullspace of $H^\ast$ has dimension $q$.} whose columns span the nullspace of $H^\ast$. Also define 
\begin{equation}
\label{eq:z-b-omega-mu}
    z_{b,\omega,\mu}(t) := D^\ast(E + \omega I)^{-1}c'' - D^\ast(E + \mu I)^{-1}c'' + \begin{bmatrix} B & C - F_{\mu} \end{bmatrix}Nt.
\end{equation}
Then there exists a unique $d \in \Fp{p}$ and $t \in \Fp{q}$ satisfying the system of equations
\begin{equation}
\label{eq:equivalent-system-eqs}
\begin{cases}
    & H^{\ast} G_\omega^{-1} \left ( Hd + \mu Nt + \begin{bmatrix} 0 \\ z_{b,\omega,\mu}(t) \end{bmatrix} \right ) = 0 \\
    & Nt = G_\mu^{-1} \left (H (H^\ast G_\mu^{-1} H)^{-1} H^\ast G_\mu^{-1} - I \right) \begin{bmatrix} c \\ c' - D^\ast (E + \mu I)^{-1} c'' \end{bmatrix},
\end{cases}
\end{equation}
where the solution $d$ satisfies $d =  d_{b,\omega,\mu}$.
\end{lemma}

\begin{proof}
To show uniqueness, suppose $(d,t), (d',t') \in \Fp{p} \times \Fp{q}$ are two solutions of \eqref{eq:equivalent-system-eqs}. Then from the second equation we get $N(t-t') = 0$; but since $N$ is of full rank $q$, we have $t = t'$. The first equation then gives $H^{\ast} G_\omega^{-1}H (d - d') = 0$. Now as $G_\omega \in \Pos{p+q}$, we have $H^{\ast} G_\omega^{-1}H \in \Pos{p}$, which gives $d = d'$ proving uniqueness.

To prove the existence of a solution to \eqref{eq:equivalent-system-eqs}, we start by expressing $d_{b,\omega,\mu}$ using \eqref{eq:x_b_w_explicit_S}, and obtain $d_{b,\omega,\mu} = (V^* A A_\omega^{-1} A V)^{-1} V^* A A_\omega^{-1} b - (V^* A A_\mu^{-1} A V)^{-1} V^* A A_\mu^{-1} b$, which after multiplying both sides by $V^* A A_\omega^{-1} AV$ and rearranging is equivalent to
\begin{equation}
\label{eq:equivalent-system-eqs-proof-1}
    V^* A A_\omega^{-1} \left \{ A V d_{b,\omega,\mu} - b + AV (V^* A A_\mu^{-1} A V)^{-1} V^* A A_\mu^{-1} b \right \} = 0.
\end{equation}
Next observe that as $A$ is Hermitian, we can express $V^* A A_\mu^{-1} A V$ and $V^* A A_\mu^{-1} b$, as $(AV)^\ast A_\mu^{-1} (AV)$ and $(AV)^\ast A_\mu^{-1} b$ respectively, and so firstly using the fact that $AV = \begin{bmatrix} V & V' \end{bmatrix} H$ from \eqref{eq:tridiag-decomp-hermitian1}, and secondly using the identities in \eqref{eq:block_inverse} one obtains
\vspace*{-0.3cm}
\begin{equation}
\label{eq:equivalent-system-eqs-proof-2}
    V^* A A_\mu^{-1} A V = H^\ast G_\mu^{-1} H, \;\; V^\ast A A_\mu^{-1} b = H^\ast G_\mu^{-1} \begin{bmatrix} c \\ c' - D^\ast (E + \mu I)^{-1} c'' \end{bmatrix},
\end{equation}
with similar expressions holding for $\mu$ replaced by $\omega$. Using \eqref{eq:equivalent-system-eqs-proof-2} one can then equivalently write \eqref{eq:equivalent-system-eqs-proof-1} as
\vspace*{-0.3cm}
\begin{equation}
\label{eq:equivalent-system-eqs-proof-3}
\begin{split}
    & H^\ast G_\omega^{-1} \left \{ H d_{b,\omega,\mu} + \begin{bmatrix} 0 \\ D^\ast(E + \omega I)^{-1}c'' - D^\ast(E + \mu I)^{-1}c'' \end{bmatrix} \right \} \\
    & + H^\ast G_\omega^{-1}  \left( H \left ( H^\ast G_\mu^{-1} H \right)^{-1} H^\ast G_\mu^{-1} - I \right) \begin{bmatrix} c \\ c' - D^\ast (E + \mu I)^{-1} c'' \end{bmatrix} = 0.
\end{split}
\end{equation}
Now let $s := \left( H \left ( H^\ast G_\mu^{-1} H \right)^{-1} H^\ast G_\mu^{-1} - I \right) \begin{bmatrix} c \\ c' - D^\ast (E + \mu I)^{-1} c'' \end{bmatrix}$. Then it follows that $H^\ast G_\mu^{-1} s = 0$, or equivalently $G_\mu^{-1} s = Nt$ for some $t \in \Fp{q}$, as the columns of $N$ form a basis for the nullspace of $H^\ast$. But this then implies that
\begin{equation}
\label{eq:equivalent-system-eqs-proof-4}
    s = G_\mu Nt = \mu Nt + \begin{bmatrix} 0 \\ \begin{bmatrix} B & C - F_\mu \end{bmatrix} Nt \end{bmatrix},
\end{equation}
using the fact that $\begin{bmatrix} T & B^\ast  \end{bmatrix} N = H^\ast N = 0$. Plugging $s$ back into \eqref{eq:equivalent-system-eqs-proof-3} then shows that $(d_{b,\omega,\mu}, t)$ is a solution of \eqref{eq:equivalent-system-eqs}, finishing the proof.
\end{proof}

It is worth noting an important special case of \Cref{lemma:HTSH}, when $A \in \Pos{n}$ or equivalently $\omin < 0$. In this case, for any $\omega \in (\omin, \infty)$, we set $\mu = 0$ in \Cref{lemma:HTSH} and obtain

\begin{corollary}
\label{cor:HTSH-positive}
When $A \in \Pos{n}$, with $d_{b,\omega, \mu}$ defined as in \Cref{lemma:HTSH}, there exists a unique solution $d = d_{b,\omega, 0}$ to the equation
\begin{equation}
    H^{\ast} G_\omega^{-1} \left ( Hd + \begin{bmatrix} 0 \\ D^\ast(E + \omega I)^{-1}c'' + BT^{-1}c - c' \end{bmatrix} \right ) = 0.
\end{equation}
\end{corollary}

\begin{proof}
We start by observing that using \eqref{eq:F-G} we have $\begin{bmatrix} I & 0 \end{bmatrix} G_0 = H^\ast$, or equivalently $H^\ast G_0^{-1} = \begin{bmatrix} I & 0 \end{bmatrix}$, from which we get $H^\ast G_0^{-1} H = T$, $T$ being positive. A simple computation then shows that
\vspace*{-0.3cm}
\begin{equation}
\label{eq:HTSH-positive-proof1}
    \left( H (H^\ast G_0^{-1} H)^{-1} H^\ast G_0^{-1} - I \right) \begin{bmatrix} c \\ c' - D^\ast E^{-1} c'' \end{bmatrix} = \begin{bmatrix} 0 \\ BT^{-1}c -c' + D^\ast E^{-1} c'' \end{bmatrix}.
\end{equation}
Using \eqref{eq:HTSH-positive-proof1} in the second equation of \eqref{eq:equivalent-system-eqs} gives
\vspace*{-0.3cm}
\begin{equation}
    G_0 Nt = \begin{bmatrix} 0 \\ BT^{-1}c -c' + D^\ast E^{-1} c'' \end{bmatrix},
\end{equation}
and so using the expression of $G_0$ from \eqref{eq:F-G}, we first conclude that $\begin{bmatrix} B & C-F_0 \end{bmatrix} Nt = BT^{-1}c -c' + D^\ast E^{-1} c''$, and then using the definition of $z_{b,\omega,\mu}(t)$ in \eqref{eq:z-b-omega-mu} we get $z_{b,\omega,0}(t) = D^\ast(E + \omega I)^{-1}c'' + BT^{-1}c - c'$. The corollary is now proved by applying \Cref{lemma:HTSH}, after setting $\mu = 0$ in the first equation of \eqref{eq:equivalent-system-eqs}.
\end{proof}

We now prove this paper's main result, stated in \Cref{sec:intro}.

\begin{proof}[Proof of \Cref{thm:main_result}]
From \Cref{lemma:HTSH} $(d_{b,\omega,\mu}, t)$ is the unique solution of \eqref{eq:equivalent-system-eqs}, for some $t \in \Fp{q}$.
Hence there exist $t' \in \F^q$ such that
\begin{equation} 
H d_{b,\omega,\mu} + \mu N t + \begin{bmatrix} 0 \\ z_{b,\omega,\mu}(t) \end{bmatrix} = G_\omega N t' = \omega N t' + \begin{bmatrix} 0 \\ \begin{bmatrix}  B & C - F_\omega \end{bmatrix} N t' \end{bmatrix}, \end{equation}
where we used the fact that $H^* N = \begin{bmatrix} T & B^* \end{bmatrix} N = 0$.
So $H d_{b,\omega,\mu} = N(\omega t' - \mu t) + \begin{bmatrix} 0 \\ z'_{b,\omega,\mu}(t,t')\end{bmatrix}$ where $z'_{b,\omega,\mu}(t,t') = \begin{bmatrix} B & C - F_\omega \end{bmatrix} N t' - z_{b,\omega,\mu}(t)$. Since $H$ is full column rank, $H^* H$ is invertible and $H^* H d_{b,\omega,\mu} = H^* \begin{bmatrix} 0 \\ z'_{b,\omega,\mu}(t,t') \end{bmatrix}$, and we conclude that 
\begin{equation}
\label{eq:d-b-omega-mu-solution}
d_{b,\omega,\mu} = (H^* H)^{-1} B^* z'_{b,\omega,\mu}(t,t').
\end{equation}
Noticing that $x_{b,\omega}, x_{b,\mu} \in \S$, we have $V d_{b,\omega,\mu} = x_{b,\omega} - x_{b,\mu}$, and \eqref{eq:d-b-omega-mu-solution} then gives $x_{b,\omega} - x_{b,\mu} \in \Img{V (H^* H)^{-1} B^*}$, for all $\omega, \mu > \omin$ and $b \in \Fn$. Since $B^\ast$ has full column rank $q$, and $V (H^* H)^{-1}$ has full column rank $p$, $V (H^* H)^{-1} B^*$ has full column rank $q$\footnote{Multiplication of a $\F$ valued matrix from the left by a full column rank matrix does not change its rank.}, and so defining $\Y := \Img{V (H^* H)^{-1} B^*}$ gives $\dim{\Y} = q$.

The theorem is proved if we can show that $\Y$ does not depend on the choice of $V, V'$. So suppose that $\overline{V} \in \Fpq{n}{p}$, $\overline{V'} \in \Fpq{n}{q}$ is a different choice of semi-unitary matrices such that $\begin{bmatrix} \overline{V} & \overline{V'} \end{bmatrix}$ is semi-unitary, $\Img{\overline{V}} = \S$, and $\Img{\begin{bmatrix} \overline{V} & \overline{V'} \end{bmatrix}} = \S + A \S$. Let $\overline{T} := \overline{V}^* A \overline{V}$, $\overline{B}^* := \overline{V}^* A \overline{V'}$, and $\overline{H}^\ast := \begin{bmatrix}
\overline{T} & \overline{B}^\ast \end{bmatrix}$ be analogously defined. Then there exists $U \in \Uni{p}$ and $U' \in \Uni{q}$, such that $V = \overline{V} U$ and $V' = \overline{V'} U'$. A simple computation then shows that $V (H^* H)^{-1} B^* = (\overline{V} (\overline{H}^* \overline{H})^{-1} \overline{B}^*) U'$, and as $U'$ is unitary this shows that $\Img{V (H^* H)^{-1} B^*} = \Img{\overline{V} (\overline{H}^* \overline{H})^{-1} \overline{B}^*}$.
\end{proof}

This proof immediately gives us bounds on the dimensions of the affine subspace $\X_b$ and the subspace $\X$, as stated in the next corollary.

\begin{corollary}
\label{cor:strong-bound}
The sets $\X_b$ and $\X$ introduced in \eqref{eq:X-X_b} satisfy
\begin{enumerate}[(i)]
    \item $\dim{\X_b} \leq \Indsimple{A}{\S}$ for all $b \in \Fn$.
    
    \item $\X \subseteq \Img{V (H^* H)^{-1} B^*}$, and so $\dim{\X} \leq \Indsimple{A}{\S}$.
\end{enumerate}
\end{corollary}

\begin{proof}
Both (i) and (ii) follow by applying \Cref{thm:main_result}, because $x_{b,\omega} - x_{b,\mu} \in \Img{V (H^* H)^{-1} B^*}$, for all $\omega, \mu > \omin$, and for all $b \in \Fn$.
\end{proof}

One should compare the bound above with that provided for $\dim{\X_b}$ by \Cref{cor:weak-bound}. We see that \Cref{cor:strong-bound}(i) provides a stricter bound, because as mentioned in the paragraph below the proof of \Cref{cor:weak-bound}, $\Indsimple{A}{\S} \leq \dim{\S'}$ with $\S'$ defined in \Cref{sssec:minprob-prop}. For example, when $\S = \K_p(A,c)$ that is not $A-$invariant, for some $c \in \Fn$, we now get that $\dim{\X_b} \leq 1$. As mentioned in \Cref{sec:intro}, the particular case, when $\F = \R$ and $c = b$, follows from the results proved in \cite{hallman2018lsmb}. One can ask whether the bounds in \Cref{cor:strong-bound} are tight, or whether they can be improved. As we will show in the next section, $\dim{\X} = \Indsimple{A}{\S}$ when $A$ belongs to certain families of matrices, for example $\Pos{n}$; thus the bound in \Cref{cor:strong-bound}(ii) cannot be improved without further assumptions. On the other hand, we will also show that there are examples where $\dim{\X_b} < \Indsimple{A}{\S}$, for all $b \in \Fn$.
\section{Tightness of bounds}
\label{sec:converse}

In this section we explore the converse of the main theorem. We continue using the notations already introduced in \Cref{ssec:defn} and \Cref{sec:main-result}. In \Cref{ssec:zero-dim}, we explore how tight is the bound $\dim{\X_b} \leq \Indsimple{A}{\S}$ for fixed $b \in \Fn$. In \Cref{ssec:weak-converse} we formulate some sufficient conditions under which $\dim{\X} = \Indsimple{A}{\S}$.

\subsection{Bounds on \normalfont{\texorpdfstring{$\dim{\X_b}$}{bounds1}}}
\label{ssec:zero-dim}

The result that motivated this whole section is the following observation.

\begin{lemma}
\label{lemma:invariant-subspace-condition}
The following conditions are equivalent:

\begin{enumerate}[(i)]
    \item  $\Ind{A}{\S} = 0$.
    \item $x_{b,\mu} = x_{b,\omega}$, for all $b \in \Fn$, and for all $\mu, \omega > \omin$.
    \item There exists distinct $\mu, \omega > \omin$, such that $x_{b,\mu} = x_{b,\omega}$, for all $b \in \Fn$.
\end{enumerate}
\end{lemma}

\begin{proof}
(i) $\rightarrow$ (ii) was proved in \Cref{lemma:solutions-invariant-S}, and (ii) $\rightarrow$ (iii) is straightforward.

We prove (iii) $\rightarrow$ (i). If $\mathcal{S} = \Fn$ we have $\Ind{A}{\Fn}=0$, so assume that $\mathcal{S} \neq \Fn$. Pick any $b \in A_{\omega} A^{-1}\mathcal{S}^{\perp}$. Then $A A_{\omega}^{-1}b \in \mathcal{S}^{\perp}$, and so $V^{\ast}A A_{\omega}^{-1}b = 0$. Since $x_{b,\mu} = x_{b,\omega}$, it follows using \eqref{eq:x_b_w_explicit_S} that $V^{\ast}A A_{\mu}^{-1}b = 0$, or $AA_{\mu}^{-1}b \in \mathcal{S}^{\perp}$. Now $A A_{\omega}^{-1}: A_{\omega} A^{-1} \mathcal{S}^{\perp} \rightarrow \mathcal{S}^{\perp}$ is an isomorphism as both $A, A_{\omega} \in \Gl{n}$, and notice that $A A_{\mu}^{-1} = A_{\mu}^{-1} A_{\omega} (A A_{\omega}^{-1})$ using the fact that $A, A_\mu, A_\omega \in \Sym{n}$; so we have in fact proved that $\Ind{A_{\mu}^{-1} A_{\omega}}{\mathcal{S}^{\perp}}=0$. Finally notice that $A_{\mu}^{-1} A_{\omega} = I + (\omega - \mu) A_{\mu}^{-1}$, which means that for any $x \in \mathcal{S}^{\perp}$, $x + (\omega - \mu)A_{\mu}^{-1}x \in \mathcal{S}^{\perp}$, and so $A_{\mu}^{-1}x \in \mathcal{S}^{\perp}$ (as $\mu \neq \omega$). Thus we conclude that $\Ind{A_{\mu}^{-1}}{\mathcal{S}^{\perp}}=0$, and by applying \Cref{lemma:index-bounds}(iii), (iv), and (i) successively, we get $\Ind{A}{\S}=0$.
\end{proof}

The equivalence of the conditions (i), (ii), and (iii) of \Cref{lemma:invariant-subspace-condition}, leads to the following corollary.

\begin{corollary}
\label{cor:converse-zero-dim}
The following statements are true.
\begin{enumerate}[(i)]
    \item $\Indsimple{A}{\S} = 0$ if and only if $\dim{\X} = 0$.
    \item $\Indsimple{A}{\S} = 0$ if and only if $\dim{\X_b} = 0$, for all $b \in \Fn$.
    \item The map $\mathbf{D}_A$ is a constant map if $\Indsimple{A}{\S} = 0$, and injective otherwise.
\end{enumerate}
\end{corollary}

\begin{proof}
First notice that $\X = \{0\}$ if and only if $\X_b$ is a singleton, for all $b \in \Fn$. Both (i) and (ii) now follow from \Cref{lemma:invariant-subspace-condition}. For (iii), \Cref{lemma:solutions-invariant-S} implies that $\mathbf{D}_A$ is a constant map if $\Indsimple{A}{\S} = 0$, while if $\Indsimple{A}{\S} > 0$ and $\mathbf{D}_A$ is not injective, there exists distinct $\mu, \omega > \omin$, such that $\D{A}{\mu} = \D{A}{\omega}$ implying that $x_{b,\mu} = x_{b,\omega}$ for all $b \in \Fn$, thereby contradicting \Cref{lemma:invariant-subspace-condition}.
\end{proof}

An interesting consequence of \Cref{cor:converse-zero-dim} is that when $\Indsimple{A}{\S} = 1$, there must exist $b \in \Fn$ such that $\dim{\X_b} = 1$, since we know that $\dim{\X_b} \leq 1$ by \Cref{thm:main_result}. One can then ask whether this pattern holds in general, that is if $\Indsimple{A}{\S} \geq 1$, whether there always exists $b \in \Fn$ such that $\dim{\X_b} = \Indsimple{A}{\S}$. However this turns out to not be true as shown by the following example, which shows that one can have cases where $\dim{\X_b} \leq 1$ for all $b \in \Fn$, even though $\Indsimple{A}{\S}$ is arbitrarily large.

\begin{example} \label{example:converse}
For $\alpha \in \R \setminus \{1, -1\}$, let $p = q \geq 1$, $n = p+q$ and consider
\begin{equation}
    A = \begin{bmatrix} \alpha I & I \\ I & \alpha I \\ \end{bmatrix} \in \F^{n \times n}
\end{equation}
with $\S = \Span{\{e_1, \dots, e_p\}}$, where $e_k \in \Fn$ is given by $(e_k)_i = \delta_{ik}$. Notice that $\det(A) = (\alpha^2 - 1)^p$, so $A \in \Gl{n}$. Furthermore, for $\alpha > 1$, $A \in \Pos{n}$ since its eigenvalues are given by $\alpha \pm 1$. With $k(\alpha,\omega) := (\alpha + \omega)^2 - 1$ (note that $\omin = 1 - \alpha$, so $k(\alpha,\omega) > 0$ for $\omega > \omin$), we find
\begin{equation}
\begin{split}
    A_\omega^{-1} &= k(\alpha,\omega)^{-1} \begin{bmatrix}
        (\alpha + \omega) I & -I \\
        -I & (\alpha + \omega) I 
    \end{bmatrix}, \\
    A A_\omega^{-1} A &= k(\alpha,\omega)^{-1} \begin{bmatrix}
(\alpha^3+\alpha^2 \omega - \alpha + \omega) I & (\alpha^2+2\alpha \omega - 1) I \\
(\alpha^2+2\alpha \omega - 1) I & (\alpha^3+\alpha^2 \omega - \alpha + \omega) I \end{bmatrix}.
\end{split}
\end{equation}
Since $A A_\omega^{-1} A \in \Pos{p}$, first note that $\alpha^3+\alpha^2 \omega - \alpha + \omega > 0$ for $\omega > \omin$, and it follows by choosing $V^* = \begin{bmatrix} I & 0\end{bmatrix}$, and $V'^\ast = \begin{bmatrix} 0 & I\end{bmatrix}$ that
\begin{equation} \label{eq:example-alpha}
\begin{split}
    V^* x_{b,\omega} &= \frac{(\alpha^2 + \alpha \omega - 1)c + \omega c'}{\alpha^3+\alpha^2 \omega - \alpha + \omega}, \\
    d_{b,\omega,\mu} &= V^*(x_{b,\omega} - x_{b,\mu}) = \frac{(\mu - \omega)(\alpha^2 - 1)(c - \alpha c')}{(\alpha^3+\alpha^2 \omega - \alpha + \omega)(\alpha^3+\alpha^2 \mu - \alpha + \mu)}.
\end{split}
\end{equation}
It is clear that $d_{b,\omega,\mu} \in \Img{c - \alpha c'}$, so $\dim{\X_b} \leq 1$ (equality holds if and only if $c - \alpha c' \neq 0$), while $\Indsimple{A}{\S} = p$ using \Cref{lemma:rank-S-perp}.
\end{example}

The above example raises the question of whether it is possible to characterize the set of $b \in \Fn$, given $A \in \Sym{n} \cap \Gl{n}$, and $\S \in \Gr{p}{n}$, such that $\dim{\X_b} \leq s$, for some $0 \leq s \leq \Indsimple{A}{\S}$. Currently we only know a satisfactory answer when $s=0$, that we now present. As preparation, we need the following lemma.

\begin{lemma}
\label{lemma:function-1d-prop}
Consider the function $f : (\omin, \infty) \rightarrow \mathbb{R}$ defined by
\begin{equation}
    f(\omega) = \frac{\alpha_1}{\lambda_1 + \omega} + \dots + \frac{\alpha_t}{\lambda_t + \omega},
\end{equation}
with $\alpha_i, \lambda_i \in \R$ for all $1 \leq i \leq t$, and $\lambda_1 > \dots > \lambda_t \geq -\omin$. Then the following conditions are equivalent.
\begin{enumerate}[(i)]
    \item  $\alpha_1 = \dots = \alpha_t = 0$.
    \item $f(\omega) = 0$ for all $\omega \in (\beta_1, \beta_2)$,  where $\omin \leq \beta_1 < \beta_2$.
    \item There exists $\omega \in (\omin,\infty)$ such that $\left( \frac{\partial^m f}{\partial \omega^m} \right) (\omega) = 0$, for all $0 \leq m \leq t-1$.
\end{enumerate}
\end{lemma}

\begin{proof}
(i) $\rightarrow$ (ii) is clear as $f$ is identically zero when $\alpha_1 = \dots = \alpha_t = 0$. (ii) $\rightarrow$ (iii) follows by choosing any $\omega \in (\beta_1, \beta_2)$, as $f$ is infinitely differentiable. We now prove (iii) $\rightarrow$ (i). For $m = 0,\dots,t-1$, we have the linear system
\begin{equation}
\label{eq: vandermonde-system}
    \begin{bmatrix}
        \frac{1}{\lambda_1 + \omega} & \frac{1}{\lambda_2 + \omega} & \dots & \frac{1}{\lambda_t + \omega} \\
        \frac{1}{(\lambda_1 + \omega)^2} & \frac{1}{(\lambda_2 + \omega)^2} & \dots & \frac{1}{(\lambda_t + \omega)^2} \\
        \vdots & \vdots & \ddots & \vdots \\
        \frac{1}{(\lambda_1 + \omega)^t} & \frac{1}{(\lambda_2 + \omega)^t} & \dots & \frac{1}{(\lambda_t + \omega)^t}
    \end{bmatrix}
    \begin{bmatrix}
        \alpha_1 \\
        \alpha_2 \\
        \vdots \\
        \alpha_t
    \end{bmatrix}
    = 0.
\end{equation}
The matrix on the left-hand side of \eqref{eq: vandermonde-system} is a Vandermonde matrix whose determinant is non-zero because $\lambda_i + \omega \neq \lambda_j + \omega$ whenever $i \neq j$, and $\lambda_i + \omega > 0$ for all $1 \leq i \leq t$, as $\omega > \omin$. Thus $\alpha_1 = \dots = \alpha_t = 0$.
\end{proof}

We can now present the first theorem of this section, which uses the notion of strong orthogonality of vectors introduced in \Cref{def:strong-orthogonality}.

\begin{theorem}
\label{thm:dim-Xb-zero-condition}
Let $A$ has $t$ distinct eigenvalues, and denote by $\S_1, \dots, \S_t$ the eigenspaces\footnote{If $\lambda$ is an eigenvalue of $A$, then the eigenspace corresponding to it is the span of all eigenvectors of $A$ with eigenvalue $\lambda$, and it is $A-$invariant.} corresponding to each distinct eigenvalue. Then the following are equivalent.
\begin{enumerate}[(i)]
    \item $\dim{\X_b} = 0$.
    \item $b - Ax_{b,\omega}$ and $v$ are strongly orthogonal with respect to $\S_1, \dots, \S_t$, for all $v \in \S$, and $\omega \in (\omin, \infty)$.
    \item There exists $\omega \in (\omin, \infty)$, such that $b - Ax_{b,\omega}$ and $v$ are strongly orthogonal with respect to $\S_1, \dots, \S_t$, for all $v \in \S$.
\end{enumerate}
\end{theorem}

\begin{proof}
We recall that for a Hermitian matrix, the eigenvectors are complete and the eigenspaces corresponding to distinct eigenvalues are orthogonal. Thus we have an orthogonal direct sum decomposition $\Fn = \S_1 \oplus \dots \oplus \S_t$, and the notion of strong orthogonality with respect to $\S_1,\dots,\S_t$ is well defined. Let $\lambda_1 > \dots > \lambda_t$ be the distinct eigenvalues of $A$, and without loss of generality assume that $\S_i$ is the eigenspace corresponding to $\lambda_i$, and let $Q_i$ be a semi-unitary matrix such that $\Img{Q_i} = \S_i$, for all $1 \leq i \leq t$. Notice that for any vectors $v_1, v_2 \in \Fn$, their orthogonal projection on $\S_i$ is given by $Q_i Q_i^\ast v_1$ and $Q_i Q_i^\ast v_2$ respectively; so we have 
\begin{equation}
\label{eq:dim-Xb-zero-condition-proof-0}
    (Q_i Q_i^\ast v_1)^\ast (Q_i Q_i^\ast v_2) = (Q_i^\ast v_1)^\ast (Q_i^\ast v_2).
\end{equation}
Also note that (ii) $\rightarrow$ (iii) is clear as (ii) is strictly stronger than (iii).

We first prove (i) $\rightarrow$ (ii). From \eqref{eq:x_b_w_explicit_S}, $\dim{\X_b} = 0$ implies that for all $\omega, \mu \in (\omin, \infty)$, $V ( V^* A A_\mu^{-1} A V)^{-1} V^* A A_\mu^{-1} b = x_{b,\omega}$. Fixing $\omega \in (\omin, \infty)$ we get that $V^* A A_\mu^{-1} b = ( V^* A A_\mu^{-1} A V) V^* x_{b,\omega}$, or $V^* A A_\mu^{-1} (b - Ax_{b,\omega}) = 0$, and then fixing any $v \in \S$, so $v = Vy$ for some $y \in \Fp{p}$, we get 
\begin{equation}
\label{eq:dim-Xb-zero-condition-proof-1}
    v^\ast A A_\mu^{-1} (b - Ax_{b,\omega}) = 0,
\end{equation}
for all $\mu \in (\omin, \infty)$. Now since $A \in \Sym{n}$, from the spectral theorem we have $A = \sum_{i=1}^{t} \lambda_i Q_i Q_i^\ast$, $A_\mu^{-1} = \sum_{i=1}^{t} (\lambda_i + \mu)^{-1} Q_i Q_i^\ast$, and $AA_\mu^{-1} = \sum_{i=1}^{t} \lambda_i (\lambda_i + \mu)^{-1}  Q_i Q_i^\ast$; thus plugging into \eqref{eq:dim-Xb-zero-condition-proof-1} gives $\sum_{i=1}^{t} \lambda_i (\lambda_i + \mu)^{-1} (Q_i^\ast v)^\ast (Q_i^\ast (b - Ax_{b,\omega})) = 0$, for all $\mu \in (\omin, \infty)$. It follows from the equivalence of conditions (i) and (ii) in \Cref{lemma:function-1d-prop}, and since $\lambda_i \neq 0$ as $A \in \Gl{n}$, that $(Q_i^\ast v)^\ast (Q_i^\ast (b - Ax_{b,\omega})) = 0$, for all $1 \leq i \leq t$. Using \eqref{eq:dim-Xb-zero-condition-proof-0} and since $v$ and $\omega$ are arbitrary, we now conclude that $b - Ax_{b,\omega}$ and $v$ are strongly orthogonal with respect to $\S_1, \dots, \S_t$, for all $v \in \S$, and $\omega \in (\omin, \infty)$.

It remains to prove (iii) $\rightarrow$ (i), so assume (iii) is true. Then first fixing $v \in \S$, and for all $\mu \in (\omin, \infty)$, we have that $\sum_{i=1}^{t} \lambda_i (\lambda_i + \mu)^{-1} (Q_i^\ast v)^\ast (Q_i^\ast (b - Ax_{b,\omega})) = 0$, i.e. \eqref{eq:dim-Xb-zero-condition-proof-1} holds, and now since $v$ is arbitrary this implies $V^* A A_\mu^{-1} (b - Ax_{b,\omega}) = 0$. Since $V V^\ast x_{b,\omega} = x_{b,\omega}$, it finally follows that $V^* A A_\mu^{-1} b = (V^* A A_\mu^{-1} A V) V^\ast x_{b,\omega}$, or $x_{b,\mu} = x_{b,\omega}$, for all $\mu \in (\omin, \infty)$. This finishes the proof of the theorem.
\end{proof}

As a consequence, we also show the following result, proving that, if $\Indsimple{A}{\S} \geq 1$, then the set $\{b \mid \dim{\X_b} = 0\}$ is of measure zero
(if $\Indsimple{A}{\S} = 0$, then $\dim{\X_b} = 0$ for all $b$ by \Cref{cor:converse-zero-dim}(ii)).

\begin{corollary}
\label{cor:zero-dim-characterization}
Assume $\Indsimple{A}{\S} \geq 1$.
Then the set $\{b \in \F^n \mid \dim{\X_b} = 0\}$ is a non-trivial subspace of $\Fn$ that has $n$-dimensional (resp. $2n$-dimensional) Lebesgue measure zero in $\Fn$, for $\F = \R$ (resp. $\F = \C$).
\end{corollary}

\begin{proof}
Fix any $\omega \in (\omin, \infty)$. Assume $A$ has $t$ distinct eigenvalues, and 
let $\S_1, \dots, \S_t$ be the $t$ associated eigenspaces, and $Q_1, \dots, Q_t$ semi-unitary such that $\Img{Q_i} = \S_i$, for all $1 \leq i \leq t$. We first note the following fact. Let $b \in \F^n$, and assume that $b - Ax_{b,\omega}$ and $v$ are strongly orthogonal with respect to $\S_1,\dots,\S_t$, for all $v \in \S$. 
Then letting $\{v_j\}_{j=1}^p$ be a basis for $\S$, it implies $(Q_i^* v_j)^* (Q_i^* (b - A x_{b,\omega})) = 0$,
or equivalently $f_{\omega,ij} b = 0$ for all $1 \leq i \leq t$, and $1 \leq j \leq p$,
where $f_{\omega,ij} := v_j^* Q_i Q_i^* (I - A \D{A}{\omega})$
and $\D{A}{\omega}$ is defined in \eqref{eq:Dww'}.
If we now define the linear operator $F_{\omega}: \F^n \to \F^{p \times t}$ as $[F_\omega(b')]_{ij} = f_{\omega,ij} b'$, for any $b' \in \Fn$, it follows from \Cref{thm:dim-Xb-zero-condition} and because $\{v_j\}_{j=1}^p$ is a basis for $\S$, that $\dim{\X_b}=0$ if and only if $F_\omega(b)=0$, or $b \in \Ker{F_\omega}$.

The discussion above already shows that the set $\{b \in \F^n \mid \dim{\X_b} = 0\}$ is a subspace, namely $\Ker{F_\omega}$. Since $\Indsimple{A}{\S} \geq 1$, we know from \Cref{cor:converse-zero-dim}(ii) that there exist $\bar b \neq 0$ such that $\dim{\X_{\bar b}} \geq 1$; so it also follows that $\bar b \not \in \Ker{F_\omega}$, or $\Ker{F_{\omega}} \neq \F^n$. Thus $\Ker{F_{\omega}}$ is at least of codimension $1$, and so it has $n$-dimensional (resp. $2n$-dimensional) Lebesgue measure zero in $\Fn$, when $\F = \R$ (resp. $\F = \C$). It remains to prove that $\Ker{F_\omega} \neq \{0\}$. Since $\Indsimple{A}{\S} \geq 1$, we have $\dim{\S} \geq 1$, so there exist $\hat{b} \in A\S$ such that $\hat{b} \neq 0$. Then $y := A^{-1} \hat{b} \in \S$, and we have $x_{\hat{b},\omega} = y$ for all $\omega \in (\omin,\infty)$ as
$\hat{b} - Ay = 0$, or $\dim{\X_{\hat{b}}} = 0$. Thus $\hat{b} \in \Ker{F_\omega}$, and we conclude.
\end{proof}
\subsection{Conditions when \normalfont{\texorpdfstring{$\dim{\X} = \Indsimple{A}{\S}$}{bounds2}}}
\label{ssec:weak-converse}

The purpose of this subsection is to provide sufficient conditions under which $\dim{\X} = \Indsimple{A}{\S}$. From \Cref{cor:converse-zero-dim}(ii) we already know that $\dim{\X} = 0$ if and only if $\Indsimple{A}{\S} = 0$, so it follows that $\dim{\X} \geq 1$ implies $\Indsimple{A}{\S} \geq 1$. Thus, throughout this subsection we will assume that $q \geq \dim{\X} \geq 1$. Moreover, by the discussion in \Cref{app:appB} and \Cref{lemma:reduction1}, we can also assume without loss of generality, that $n > p+q$ (otherwise in the $n = p+q$ case we consider the modified problem).

We start with the observation that a sufficient condition to ensure $\dim{\X} = \Indsimple{A}{\S}= q$ is that for every $y \in \Y$ (with $\Y = \Img{V (H^\ast H)^{-1} B^\ast}$ as defined in \Cref{thm:main_result}), there exist $b \in \Fn$ and $\omega, \mu > \omin$, such that $y = x_{b,\omega} - x_{b,\mu}$, or equivalently $V^\ast y = V^\ast (x_{b,\omega} - x_{b,\mu}) = d_{b,\omega,\mu}$. This is because $\dim{\Y} = q$, and $\X \subseteq \Y$. But since $\Y = \Img{V (H^\ast H)^{-1} B^\ast}$, this is equivalent to showing that for every $u \in \Fp{q}$, there exist $b, \omega, \mu$ such that $d_{b,\omega,\mu} = (H^\ast H)^{-1} B^\ast u$. For convenience, let us define for all $\omega > \omin$, $J_\omega \in \Fpq{(p+q)}{(p+q)}$ and $E_\omega \in \Fpq{(n - p - q)}{(n - p - q)}$ as
\begin{equation}
\label{eq:Jmu}
    J_\omega := G_\omega^{-1} \left (H (H^\ast G_\omega^{-1} H)^{-1} H^\ast G_\omega^{-1} - I \right), \;\;\; E_\omega := E + \omega I.
\end{equation}
Let $N$ be defined as in \Cref{lemma:HTSH}, and suppose $N = \begin{bmatrix}N_1 \\ N_2\end{bmatrix}$ be a partitioning of $N$, where $N_1 \in \Fpq{p}{q}$ and $N_2 \in \Fpq{q}{q}$. Then by \Cref{lemma:HTSH} and the quantities defined therein, we deduce the following sufficient condition:

\begin{lemma} 
\label{lemma:sufficient-condition}
Let $u \in \Fp{q}$ be fixed. If there exist $b \in \Fn$, $\omega, \mu \in (\omin, \infty)$, and $t, t' \in \Fp{q}$ satisfying the system
\begin{equation}
\label{eq:sufficient-cond}
\begin{cases}
    & T (H^\ast H)^{-1} B^\ast u = N_1 (\omega t' - \mu t) \\
    & B (H^\ast H)^{-1} B^\ast u = N_2 (\omega t' - \mu t) + z'_{b,\omega,\mu}(t,t') \\
    & Nt = J_\mu \begin{bmatrix} c \\ c' - D^\ast E_\mu^{-1} c'' \end{bmatrix},
\end{cases}
\end{equation}
with $z'_{b,\omega,\mu}(t,t') = \begin{bmatrix} B & C - F_\omega \end{bmatrix} N t' - z_{b,\omega,\mu}(t)$ (as defined in \eqref{eq:z-b-omega-mu}), then $d_{b,\omega,\mu} = (H^\ast H)^{-1} B^\ast u$. 
Conversely, for fixed $(b,\omega,\mu)$, if $d_{b,\omega,\mu} = (H^\ast H)^{-1} B^\ast u$, then there exists $(t,t')$ satisfying \eqref{eq:sufficient-cond}.
Finally, if a $(b,\omega,\mu,t,t')$ exists for every $u \in \Fp{q}$ solving \eqref{eq:sufficient-cond}, then  $\dim{\X} = q$. 

\end{lemma}

\begin{proof}
Combining the first two equations of \eqref{eq:sufficient-cond} gives $H (H^\ast H)^{-1} B^\ast u + \mu N t + \begin{bmatrix} 0 \\ z_{b,\omega,\mu}(t) \end{bmatrix} = G_\omega N t'$.
Since $N$ is a basis for the nullspace of $H^\ast$, this is equivalent to the first equation of \eqref{eq:equivalent-system-eqs} with $d = (H^\ast H)^{-1} B^\ast u$.
Now let $u \in \F^q$ be fixed. By applying \Cref{lemma:HTSH} we find $d_{b,\omega,\mu} = (H^\ast H)^{-1} B^\ast u$.
For the converse, suppose $d_{b,\omega,\mu} = (H^\ast H)^{-1} B^\ast u$. Then by
\Cref{lemma:HTSH}, $t$ exists such that \eqref{eq:equivalent-system-eqs} is satisfied. Since $N$ is a basis for the nullspace of $H^\ast$, there exist $t'$ such that $Hd+\mu N t + \begin{bmatrix} 0 \\ z_{b,\omega,\mu}(t)\end{bmatrix} = G_\omega N t'$ and the conclusion follows. Finally, we have already argued the last statement in the paragraph immediately before this lemma.
\end{proof}

Because of this result, our task now reduces to finding conditions that guarantee solutions to \eqref{eq:sufficient-cond}. It turns out that the first and third equations of \eqref{eq:sufficient-cond} pose no obstructions, which we show in the next lemma, and recall that we denote $c = V^\ast b$, $c' = V'^\ast b$ and $c'' = V''^\ast b$. So a choice of $(c,c',c'')$ uniquely defines $b$, and vice-versa.

\begin{lemma}
\label{lemma:first-third}
The following statements are true:
\begin{enumerate}[(i)]
    \item For every $u \in \Fp{q}$, there exists a unique $t'' \in \Fp{q}$, such that $T (H^\ast H)^{-1} B^\ast u = N_1 t''$, showing that the first equation of \eqref{eq:sufficient-cond} also admits a solution for some $\omega, \mu > \omin$, and $t,t' \in \Fp{q}$.
    \item For every $\mu > \omin$, $t \in \Fp{q}$, and $c'' \in \Fp{n-p-q}$, there exists $c \in \Fp{p}$, and $c' \in \Fp{q}$ solving the third equation of \eqref{eq:sufficient-cond}. Conversely for every $\mu > \omin$, and $b \in \Fn$, there exists a unique $t \in \Fp{q}$ solving the third equation of \eqref{eq:sufficient-cond}.
\end{enumerate}
\end{lemma}

\begin{proof}
\begin{enumerate}[(i)]
    \item This proof relies on the results of \Cref{app:appC}. Fix $u \in \Fp{q}$. Firstly, it has already been argued in \Cref{cor:nullspace-H*}(i) that $\rank{(N_1)} = q$, so it follows that $t''$ is unique, if it exists. Let $u' = (H^\ast H)^{-1} B^\ast u$, and so $(H^\ast H) u' = (T^2 + B^\ast B) u' = B^\ast u$, or equivalently $T^2u' = B^\ast(u - Bu')$. But this then implies that $T^2 u' \in \Img{T} \cap \Img{B^\ast}$, and since $Tu' \in \Img{T}$ also, we further deduce that $Tu' \in T^{\dagger} (\Img{T} \cap \Img{B^\ast})$, where $T^{\dagger}$ denotes the pseudoinverse of $T$, defined in \Cref{cor:nullspace-H*}. Finally, \Cref{cor:nullspace-H*}(ii) shows that $T^{\dagger} (\Img{T} \cap \Img{B^\ast}) \subseteq \Img{N_1}$, and so $Tu' \in \Img{N_1}$, proving existence of $t''$. One can now choose $t = t' = t''$ and $\omega, \mu > \omin$ such that $\omega - \mu = 1$, which ensures $\omega t' - \mu t = t''$, and solves the first equation of \eqref{eq:sufficient-cond}.
    \item Fix any $\mu > \omin$. Let $H_\mu = G_\mu^{-1/2} H$, and notice that since $\rank{(H)} = p$, we have $\rank{(H_\mu)} = p$, and $\dim{\Img{H_\mu}^\perp} = q$. Also note that $J_\mu = - G_\mu^{-1/2} (I - H_\mu(H_\mu^\ast H_\mu)^{-1} H_\mu^\ast) G_\mu^{-1/2}$, from which we may observe that the inner factor is an orthogonal projector onto $\Img{H_\mu}^\perp$, and so $\rank{(J_\mu)} = \dim{\Img{J_\mu}} = q$ using invertibility of $G_\mu^{-1/2}$. Moreover, one can compute that $H^\ast J_\mu = 0$, which gives $\Img{J_\mu} \subseteq \Img{N}$, since the columns of $N$ span the null space of $H^\ast$. But $\rank{(N)} = q$, and so in fact 
    \begin{equation}
    \label{eq:Jmu-N-images}
        \Img{J_\mu} = \Img{N}.
    \end{equation}
    Now if $b \in \Fn$ is given, \eqref{eq:Jmu-N-images} implies the  existence of $t \in \Fp{q}$ satisfying the third equation of \eqref{eq:sufficient-cond}, and it is unique as $N$ is full column rank (by construction). Next suppose that $t$ and $c''$ are given. Then \eqref{eq:Jmu-N-images} again implies the existence of $v \in \F^{p+q}$ such that $Nt = J_\mu v$. We can then choose $c,c'$ such that $\begin{bmatrix}c \\ c' \end{bmatrix} =  v + \begin{bmatrix}0 \\ D^\ast E_\mu^{-1} c'' \end{bmatrix}$, and this is a solution to the third equation of \eqref{eq:sufficient-cond}. Since $\mu$ is arbitrary, this completes the proof.
\end{enumerate}
\end{proof}

We can now state the first condition that ensures that one can find for all $u \in \Fp{q}$, $(b, \omega, \mu, t, t')$ satisfying \eqref{eq:sufficient-cond}.

\begin{lemma}
\label{lemma:first-sufficient-condition}
If $\Img{T} \cap \Img{B^\ast} = \{0\}$, for each $u \in \Fp{q}$, there exists a $(b, \omega, \mu, t, t')$ satisfying \eqref{eq:sufficient-cond}.
\end{lemma}

\begin{proof}
The main ingredient of this proof is the characterization of $N$ provided by \Cref{lemma:nullspace-special-cases}(iii), by which if $\Img{T} \cap \Img{B^\ast} = \{0\}$, then one must choose $N_1$ so that $\Img{N_1} = \Img{T}^\perp$, and $N_2 = 0$. We claim that $BN_1 \in \Fpq{q}{q}$ is invertible, which we prove later. Now fix any $u \in \Fp{q}$, $\omega, \mu > \omin \; (\omega \neq \mu)$, and $c'' \in \Fp{n-p-q}$. Then by \Cref{lemma:first-third}(i) there exists a unique $t'' \in \Fp{q}$ such that $T (H^\ast H)^{-1} B^\ast u = N_1 t''$. Let $\mathcal{A}_1 := \{(t,t') \mid t,t' \in \Fp{q}, \omega t' - \mu t = t''\}$, and notice that it is non-empty. Also, since $N_2 = 0$, the second equation of \eqref{eq:sufficient-cond} reduces to
\begin{equation}
\label{eq:first-sufficient-condition-proof-1}
    BN_1 (t'-t) = B(H^\ast H)^{-1} B^\ast u + D^\ast (E_\omega^{-1} - E_\mu^{-1}) c'',
\end{equation}
and then by the invertibility of $BN_1$, there exists a unique $t''' \in \Fp{q}$ such that the set $\mathcal{A}_2 := \{(t,t') \mid t,t' \in \Fp{q}, t'- t = t''', \text{\eqref{eq:first-sufficient-condition-proof-1} holds}\}$ is non-empty. So the set $\mathcal{A}_1 \cap \mathcal{A}_2$ contains exactly one element $(t,t')$ which is the solution to the equation 
\begin{equation}
\label{eq:first-sufficient-condition-proof-2}
    \begin{bmatrix}
        \omega I & -\mu I \\
        I & -I
    \end{bmatrix}
    \begin{bmatrix}
        t' \\
        t
    \end{bmatrix}=
    \begin{bmatrix}
        t'' \\
        t'''
    \end{bmatrix},
\end{equation}
because the matrix on the left hand side has determinant $(\mu - \omega)^q \neq 0$ (by assumption). Finally with the choices for $\mu, c''$ and the solution $t$ of \eqref{eq:first-sufficient-condition-proof-2}, we can by \Cref{lemma:first-third}(ii) find $c,c'$ satisfying the third equation of \eqref{eq:sufficient-cond}. Thus we have found $(b, \omega, \mu, t, t')$ satisfying \eqref{eq:sufficient-cond}, and this proves the lemma as $u$ is arbitrary.

Now we prove the claim $BN_1 \in \Gl{q}$. Suppose for the sake of contradiction this is not true, and there exists $0 \neq y \in \Fp{q}$ such that $BN_1 y = 0$. Then $N_1 y \in \Img{T}^\perp$ (as $\Img{N_1} = \Img{T}^\perp$), and $N_1 y \in \Ker{B} = \Img{B^\ast}^\perp$ simultaneously, so in fact $N_1 y \in (\Img{T} + \Img{B^\ast})^\perp$. Finally, we also know that $\rank{(H^\ast)} = p$, so $\Fp{p} = \Img{H^\ast} = \Img{T} + \Img{B^\ast}$ implying $N_1 y \in (\F^p)^\perp$, hence $N_1 y = 0$. Since $N_1$ is full column rank by \Cref{cor:nullspace-H*}(i), this implies $y = 0$ and gives a contradiction.
\end{proof}

A second condition that guarantees for all $u \in \Fp{q}$ existence of $(b, \omega, \mu, t, t')$ satisfying \eqref{eq:sufficient-cond}, is that $T$ is invertible. In fact we state a much stronger theorem below, from which our result will follow. We also recall the definition of $\dD{A}{\omega}{\mu}$ from \Cref{def:Dww'}, so $V d_{b,\omega,\mu} = x_{b,\omega} - x_{b,\mu} = \dD{A}{\omega}{\mu}b$, and we know that $\Img{\dD{A}{\omega}{\mu}} \subseteq \Y$ by \Cref{thm:main_result}. In the theorem below, we are concerned about when this can actually be an equality.

\begin{theorem}
\label{thm:second-condition}
Suppose $T \in \Gl{p}$, and consider the open set $\mathcal{U} := (\omin, \infty) \times (\omin, \infty) \subseteq \R^2$. Then there exists a closed subset $\mathcal{V} \subseteq \mathcal{U}$ of $2$-dimensional Lebesgue measure zero, such that for all $(\omega, \mu) \in \mathcal{U} \setminus \mathcal{V}$, we have $\Img{\dD{A}{\omega}{\mu}} = \Y$ (with $\Y$ defined as in \Cref{thm:main_result}). Moreover if $A \in \Pos{n}$, then $\Img{\dD{A}{\omega}{\mu}} = \Y$ for all $\omega, \mu \geq 0$ and $\omega \neq \mu$.
\end{theorem}

The proof of \Cref{thm:second-condition}, relies on the following lemma which we first state and prove below, and then supply the proof of \Cref{thm:second-condition}.

\begin{lemma}
\label{lemma:KLprop}
Let $T \in \Gl{p}$, and $\mathcal{U}$ be defined as in \Cref{thm:second-condition}. Define $K(\omega, \mu) \in \Fpq{(n-p-q)}{(n-p-q)}$, and $L(\omega, \mu) \in \Fpq{q}{q}$, for all distinct $\omega, \mu \in (\omin,\infty)$ as
\vspace*{-0.1cm}
\begin{equation}
\label{eq:KLdef}
    K(\omega, \mu) := \left( \frac{\mu E_{\omega}^{-1} - \omega E_{\mu}^{-1}}{\mu - \omega} \right), \; L(\omega, \mu) := C - B T^{-1} B^\ast - D^\ast K(\omega, \mu) D.
\end{equation}
Then we have the following:
\begin{enumerate}[(i)]
    \item There exists a closed subset $\mathcal{V} \subseteq \mathcal{U}$ of $2$-dimensional Lebesgue measure zero, such that for all $(\omega, \mu) \in \mathcal{U} \setminus \mathcal{V}$, $L(\omega, \mu)$ is invertible.
    
    \item Suppose additionally that $A \in \Pos{n}$. Then $L(\omega, \mu) \in \Pos{q}$ for all $\omega, \mu \geq 0$, and $\omega \neq \mu$.
\end{enumerate}
\end{lemma}

\remark{Notice that in part (ii) of \Cref{lemma:KLprop}, the assumption $A \in \Pos{n}$ guarantees that $T \in \Pos{p}$, hence invertible. So in fact, for (ii) the assumption $T \in \Gl{p}$ is not needed.}

\begin{proof}
In this proof, whenever we say that a set has measure zero, it will mean that it has 2-dimensional Lebesgue measure zero. Also for both parts (i) and (ii), we note that as $A \in \Gl{n}$, the matrix on the right hand side of \eqref{eq:tridiag-decomp-hermitian} is invertible, and since $T \in \Gl{p}$, this implies that $M := \begin{bmatrix} C - B T^{-1} B^\ast & D^\ast \\ D & E \end{bmatrix} \in \Gl{n-p} \cap \Sym{n-p}$, being the Schur complement of $T$. \Cref{fig:measure_0_sets} illustrates some of the sets used in this proof.
\begin{enumerate}[(i)]
    \item Suppose $E$ has $k$ distinct eigenvalues $\{\xi_i\}_{i=1}^{k}$, for some $1 \leq k \leq n-p-q$, ordered as $\xi_1 > \xi_2 > \dots > \xi_k \geq -\omin$, which are real as $E \in \Sym{n-p-q}$, while the lower bound $-\omin$ is due to the Cauchy interlacing theorem (Theorem 8.1.7 in \cite{golub2013matrix}). Let $E = Q \Lambda Q^\ast $ be the eigenvalue decomposition of $E$, for some $Q \in U(n-p-q)$, and a diagonal matrix $\Lambda$ (thus $\Lambda_{jj} \in \{\xi_i \mid 1 \leq i \leq k\}$ for all $1 \leq j \leq n-p-q$). Now define the set $\mathcal{W} := \{(\omega, \mu) \in \R^2 \mid \omega + \mu \in \{-\xi_i \mid 1 \leq i \leq k\}\}$, and the function $f : \R^2 \setminus \mathcal{W} \rightarrow \R$ as
    \begin{equation}
    \label{eq:f-def}
        f(\omega, \mu) = \det \left( \begin{bmatrix} C - B T^{-1} B^\ast & D^\ast \\ D & E \end{bmatrix} + \begin{bmatrix} 0 & 0 \\ 0 & Q \Delta(\omega,\mu) Q^\ast \end{bmatrix} \right),
    \end{equation}
    where $\Delta(\omega,\mu)$ is a diagonal matrix having the same shape as $\Lambda$, and diagonal entries $(\Delta(\omega,\mu))_{jj} = \omega \mu / (\Lambda_{jj} + \omega + \mu)$, for all $1 \leq j \leq n-p-q$. Notice that $f$ is indeed $\R$-valued as the matrix on the right hand side of \eqref{eq:f-def} is Hermitian by construction, so the determinant is just the product of the eigenvalues.\\
    
    Next notice that $\mathcal{W}$ is a closed subset of measure zero, as it is the union of $k$ parallel lines $\omega + \mu + \xi_i = 0$, for $1 \leq i \leq k$. Hence $\R^2 \setminus \mathcal{W}$ is open and is a disjoint union of $k+1$ connected components $\mathcal{R}_1, \mathcal{R}_2, \dots, \mathcal{R}_{k+1}$ (each of which is also open) separated by these lines. Our immediate goal is to show that $f(\omega, \mu) \neq 0$ almost everywhere in $\R^2 \setminus \mathcal{W}$. To do this, we claim that $f$ is a real analytic map on each connected component, which we prove later. Now fix a connected component $\mathcal{R}_i$, and let $\mathcal{A}_i := \{(\omega, \mu) \in \mathcal{R}_i \mid \mu = 0\}$, noting that $\mathcal{A}_i$ is non-empty. Then for any $(\omega, \mu) \in \mathcal{A}_i$, we have 
    \begin{equation}
        f(\omega, \mu) = f(\omega, 0) = \det \left( \begin{bmatrix} C - B T^{-1} B^\ast & D^\ast \\ D & E \end{bmatrix} \right) \neq 0,
    \end{equation}
    thus $f|_{\mathcal{R}_i}$ is either non-zero everywhere in $\mathcal{R}_i$, or otherwise non-constant. So in both cases there exists a closed subset $\mathcal{V}_i \subseteq \mathcal{R}_i$ of measure zero, such that $f(\omega, \mu) \neq 0$ for all $(\omega, \mu) \in \mathcal{R}_i \setminus \mathcal{V}_i$ --- in the first case $\mathcal{V}_i$ is the empty set, while in the second case it follows by applying \Cref{lemma:real-analyticity-prop}(i), as $f|_{\mathcal{R}_i}$ is a non-constant real analytic map. Thus defining $\mathcal{V}' := \left( \bigcup_{i=1}^{k+1} \mathcal{V}_i \right) \bigcup \mathcal{W}$, we conclude that $\mathcal{V}'$ is a closed subset of measure zero such that $f(\omega,\mu) \neq 0$ for all $(\omega,\mu) \in \R^2 \setminus \mathcal{V}'$.\\
    
    Returning to the proof, note that whenever $\omega, \mu \in (\omin,\infty)$ with $\omega \neq \mu$, it follows using $E = Q \Lambda Q^\ast$ that the eigenvalues of $K(\omega, \mu)$ belong to the set $\left \{ (\xi_i + \omega + \mu) (\xi_i + \omega)^{-1}(\xi_i + \mu)^{-1} \mid 1 \leq i \leq k \right \}$. Thus for all $(\omega, \mu) \in \mathcal{U} \setminus \left( \mathcal{W} \cup \{(\omega,\omega) \mid \omega \in \R\} \right)$, $K(\omega, \mu)$ is invertible and we have
    \begin{equation}
    \label{eq:Kinv-def}
    \begin{split}
        K(\omega, \mu)^{-1} &= Q \left( \frac{\mu (\Lambda + \omega I)^{-1} - \omega (\Lambda + \mu I)^{-1}}{\mu - \omega} \right)^{-1} Q^\ast \\
        &= Q \Lambda Q^\ast + Q \Delta(\omega,\mu) Q^\ast = E + Q \Delta(\omega,\mu) Q^\ast.
    \end{split}
    \end{equation}
    Now define $\mathcal{V} := \mathcal{U} \cap \left( \mathcal{V}' \cup \; \{(\omega,\omega) \mid \omega \in \R\} \right)$, and note that $\mathcal{V}$ is a closed subset of $\mathcal{U}$ of measure zero. Combining \eqref{eq:f-def} and \eqref{eq:Kinv-def}, it follows from the previous paragraph that $f(\omega, \mu) = \det \left( \begin{bmatrix} C - B T^{-1} B^\ast & D^\ast \\ D & K(\omega, \mu)^{-1} \end{bmatrix} \right) \neq 0$ for all $(\omega,\mu) \in \mathcal{U} \setminus \mathcal{V}$. Since $K(\omega, \mu)^{-1}$ is also invertible in $\mathcal{U} \setminus \mathcal{V}$, we can finally conclude that $L(\omega, \mu)$ is invertible for all $(\omega,\mu) \in \mathcal{U} \setminus \mathcal{V}$, as it is the Schur complement of $K(\omega, \mu)^{-1}$ in $\begin{bmatrix} C - B T^{-1} B^\ast & D^\ast \\ D & K(\omega, \mu)^{-1} \end{bmatrix}$.\\
    
    It remains to prove that $f$ is a real analytic map on each connected component. Consider an open, connected component $\mathcal{R}_i$.
    Rewrite \eqref{eq:f-def} as
    \begin{equation*}
        f(\omega, \mu) = \det(P(\omega,\mu)) = \det \left( \begin{bmatrix} C - B T^{-1} B^\ast & D^\ast Q \\ Q^\ast D & Q^\ast E Q \end{bmatrix} + \begin{bmatrix} 0 & 0 \\ 0 & \Delta(\omega,\mu) \end{bmatrix} \right),
    \end{equation*}
    where $P(\omega,\mu) \in \Sym{n}.$
    The left matrix is constant and $\Delta(\omega,\mu)$ is diagonal where $\Delta(\omega,\mu)_{jj} \in \R$ is the quotient of real analytic functions ($\omega\mu$ and $\Lambda_{jj}+\omega+\mu$) and since $\Lambda_{jj}+\omega+\mu \neq 0$ for $(\omega,\mu) \in \mathcal{R}_i$, it is real analytic (see Proposition 2.2.2 in \cite{krantz2002primer}). Using the Leibniz formula (Theorem 2.4 in \cite{kwak2004linear}), let $S_n$ be all the permutations of $\{1,\dots,n\}$; then $\det(P(\omega,\mu)) = \sum_{\sigma \in \S_n} \sign(\sigma) \prod_i P(\omega,\mu)_{i,\sigma_i} \in \R$, where $\sign(\sigma) \in \{-1,1\}$ is the sign of the permutation. For a given permutation $\sigma$, $\prod_i P(\omega,\mu)_{i,\sigma_i}$ contains the product of real diagonal entries ($i$ such that $i = \sigma_i$), call it $g_\sigma(\omega,\mu) \in \R$, and other constant non-diagonal entries, call it $h_\sigma \in \F$ (with $\F$ either $\R$ or $\C$). Again, by Proposition 2.2.2 in \cite{krantz2002primer}, $g_\sigma$ is the product of real analytic functions and is real analytic, and since $\det(P(\omega,\mu)) \in \R$, we have $\det(P(\omega,\mu)) = \sum_{\sigma \in S_n} \sign(\sigma) g_\sigma(\omega,\mu) \Re(h_\sigma)$, which is a sum of real analytic functions and is real analytic.
    
    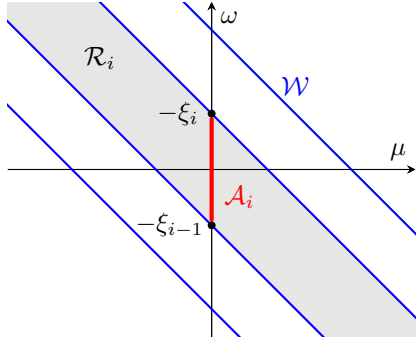
\begin{figure}
        \centering
        \begin{tikzpicture}
            \begin{axis}[
                axis equal,width=7cm,ticks=none,domain=-10:10,
                ymin=-6,ymax=6,xmin=-6,xmax=6,axis lines=middle,
            ]
                \fill [black!10] (axis cs:-8,6) -- (axis cs:8,-10) -- (axis cs: 8,-6) -- (axis cs:-8,10) -- (axis cs: -8,6);
            \end{axis}
            \begin{axis}[
                axis equal,width=7cm,ticks=none,domain=-10:10,
                ymin=-6,ymax=6,xmin=-6,xmax=6,axis lines=middle,xlabel={$\mu$},ylabel={$\omega$},
            ]
                \addplot[blue,thick] (x,-5-x);
                \addplot[blue,thick] (x,-2-x);
                \addplot[blue,thick] (x,2-x);
                \addplot[blue,thick] (x,5-x);
                \node[circle,fill=black,inner sep=0pt,minimum size=3pt] (a) at (axis cs:0,-2) {};
                \node[circle,fill=black,inner sep=0pt,minimum size=3pt] (b) at (axis cs:0,2) {};
                \node (wtext) at (axis cs:3,3) {\textcolor{blue}{$\mathcal{W}$}};
                \draw[red,ultra thick] (a) -- (b);
                \node at (axis cs: 1,-1) {\textcolor{red}{$\mathcal{A}_i$}};  
                \node[] at (axis cs:-1.5,-2) {$-\xi_{i-1}$};
                \node[] at (axis cs:-1.2,2) {$-\xi_{i}$};
                \node[] at (axis cs:-4,4) {$\mathcal{R}_i$};
            \end{axis}
        \end{tikzpicture}
        \caption{Illustration of some of the sets used in the proof of \Cref{lemma:KLprop}.}
        \label{fig:measure_0_sets}
    \end{figure}

    \item We will use some of the notations introduced in the proof of (i). Since $A \in \Pos{n}$, the matrix on the right hand side of \eqref{eq:tridiag-decomp-hermitian} is also positive, from which we firstly have $T \in \Pos{p}$, and so its Schur complement $M \in \Pos{n-p}$, and secondly $E \in \Pos{n-p-q}$, so $\xi_i > 0$ for all $1 \leq i \leq k$. Let $\mathcal{U}' := [0, \infty) \times [0, \infty)$, and $\mathcal{W}' := \{(\omega,\omega) \mid \omega \in \R\}$. Now for all $(\omega, \mu) \in \mathcal{U}' \setminus \mathcal{W}'$, we have $\xi_i + \omega + \mu > 0$, $\xi_i + \omega > 0$, and $\xi_i + \mu > 0$, for all $1 \leq i \leq k$; thus all the eigenvalues of $K(\omega,\mu)$ are positive, and all the eigenvalues of  $Q \Delta(\omega,\mu) Q^\ast$ are non-negative. It follows that both $K(\omega,\mu), K(\omega,\mu)^{-1} \in \Pos{n-p-q}$, and
    \begin{equation}
        \begin{bmatrix} C - B T^{-1} B^\ast & D^\ast \\ D & K(\omega, \mu)^{-1} \end{bmatrix} = M + \begin{bmatrix} 0 & 0 \\ 0 & Q \Delta(\omega,\mu) Q^\ast \end{bmatrix} \in \Pos{n-p},
    \end{equation}
    for all $(\omega, \mu) \in \mathcal{U}' \setminus \mathcal{W}'$, and again taking the Schur complement of $K(\omega, \mu)^{-1}$, as in the proof of (i), we conclude that $L(\omega, \mu) \in \Pos{q}$.
\end{enumerate}
\end{proof}

\begin{proof}[Proof of \Cref{thm:second-condition}]
Since $T \in \Gl{p}$, which also holds if $A \in \Pos{n}$, we can by \Cref{lemma:nullspace-special-cases}(i) choose $N_1 = -T^{-1} B^\ast$ and $N_2 = I$. Now choose $\mathcal{V}$ as defined in the proof of \Cref{lemma:KLprop}(i), and note that it is a closed subset of $\mathcal{U}$ of 2-dimensional Lebesgue measure zero. Then for all $(\omega, \mu) \in \mathcal{U} \setminus \mathcal{V}$, we have using \Cref{lemma:KLprop}(i) and $L(\omega,\mu)$ defined therein, that $L(\omega,\mu) \in \Gl{q}$; so let us choose $\omega, \mu \in \mathcal{U} \setminus \mathcal{V}$, and note that $\omega \neq \mu$ as such points are excluded by the construction of $\mathcal{V}$. Let us also fix some $u \in \Fp{q}$, and $c'' \in \Fp{n-p-q}$.  Now as in the proof of \Cref{lemma:first-sufficient-condition}, let $\mathcal{A}_1 := \{(t,t') \mid t,t' \in \Fp{q}, \omega t' - \mu t = t''\}$, where $t'' \in \Fp{q}$ is the unique solution to $T (H^\ast H)^{-1} B^\ast u = N_1 t''$, by \Cref{lemma:first-third}(i), and note that $\mathcal{A}_1$ is non-empty (for example, one can choose $t = t' = t'' / (\omega - \mu)$). The second equation of \eqref{eq:sufficient-cond}, with these choices for $u$, $\omega$, $\mu$, $c''$, and $t''$ then becomes
\begin{equation}
\label{eq:second-condition-proof-1}
    (C-BT^{-1}B^\ast) (t'-t) + (F_\mu t - F_\omega t') = B(H^\ast H)^{-1} B^\ast u + D^\ast (E_\omega^{-1} - E_\mu^{-1}) c'' - t'',
\end{equation}
and let us define $\mathcal{A}_2 := \{(t,t') \mid t,t' \in \Fp{q}, \text{\eqref{eq:second-condition-proof-1} holds}\}$. Our goal is to now show that $\mathcal{A}_1 \cap \mathcal{A}_2$ has a unique element. To do this, first assuming $\omega \neq 0$, we write $t' = (t'' + \mu t) / \omega$ and plug into \eqref{eq:second-condition-proof-1}, which after some rearrangement gives
\begin{equation}
\label{eq:second-condition-proof-2}
    \left(C-BT^{-1}B^\ast - \frac{\mu F_\omega - \omega F_\mu}{\mu -\omega}\right) t = g(u, \omega, \mu, c'', t''),
\end{equation}
for some $g \in \Fp{q}$. But the matrix on the left hand side of \eqref{eq:second-condition-proof-2} is exactly $L(\omega,\mu)$, which is invertible, thus there exists a unique $t \in \Fp{q}$ solving \eqref{eq:second-condition-proof-2}, and then $t' = (t'' + \mu t) / \omega$ is also uniquely determined. Now if $\omega = 0$, then $\mu \neq 0$, so we write $t = (\omega t' - t'')/\mu$, and then \eqref{eq:second-condition-proof-2} holds with $t$ replaced by $t'$, and we get the same uniqueness statement for $t,t'$. Finally by \Cref{lemma:first-third}(ii), the third equation of \eqref{eq:sufficient-cond} now has a solution $(c,c')$ with these choices for $\mu, c''$, and $t$. We have thus shown that for all $(\omega, \mu) \in \mathcal{U} \setminus \mathcal{V}$, and $u \in \Fp{q}$, we can obtain solutions $(b,t,t')$ satisfying \eqref{eq:sufficient-cond}; that is by \Cref{lemma:sufficient-condition}, $\dD{A}{\omega}{\mu} b = Vd_{b,\omega,\mu} = V(H^\ast H)^{-1} B^\ast u$. This proves that $\Img{\dD{A}{\omega}{\mu}} = \Y$.

When $A \in \Pos{n}$, we can repeat the same argument as above, with $\mathcal{U}$ replaced by $\mathcal{U}' := [0, \infty) \times [0, \infty)$, and $\mathcal{V}$ replaced by $\mathcal{W}' := \{(\omega,\omega) \mid \omega \in \R\}$, and in this case we need to use that $L(\omega,\mu) \in \Pos{q}$, for all $(\omega, \mu) \in \mathcal{U}' \setminus \mathcal{W}'$ by \Cref{lemma:KLprop}(ii). This completes the proof.
\end{proof}

Combining \Cref{lemma:sufficient-condition}, \Cref{lemma:first-sufficient-condition,thm:second-condition} we have thus finished the proof of the following corollary:
\begin{corollary}
\label{cor:sufficient-conditions}
$\dim{\X} = \Indsimple{A}{\S}$ if any of the following conditions hold:
\begin{enumerate}[(i)]
    \item $\Img{T} \cap \Img{B^\ast} = \{0\}$,
    \item $T \in \Gl{p}$.
\end{enumerate}
\end{corollary}

We state a surprising consequence of \Cref{thm:main_result}, and \Cref{cor:converse-zero-dim,cor:sufficient-conditions}, below, which is valid even when $n = p+q$, whose part (i) shows that in the very special case of $\Indsimple{A}{\S} = 1$, \Cref{thm:second-condition} can be strengthened significantly.

\begin{corollary}
\label{cor:rank-dD-index1}
Suppose $A \in \Gl{n}$ such that $\Indsimple{A}{\S} = 1$. Then 
\begin{enumerate}[(i)]
    \item For all distinct $\omega , \mu \in (\omin, \infty)$, the matrix $\dD{A}{\omega}{\mu}$ has rank 1, and constant image $\Y$ defined in \Cref{thm:main_result}.
    \item $\dim{\X} = 1$.
\end{enumerate}
\end{corollary}

\begin{proof}
\begin{enumerate}[(i)]
    \item Fix any $\omega , \mu \in (\omin, \infty)$ such that $\omega \neq \mu$. Since $\Indsimple{A}{\S} = 1$, by \Cref{cor:converse-zero-dim}(iii), $\mathbf{D}_A$ is injective, so $\dD{A}{\omega}{\mu} \neq 0$. Thus $\dD{A}{\omega}{\mu}$ at least has rank 1. Moreover, by \Cref{thm:main_result}, $\Img{\dD{A}{\omega}{\mu}} \subseteq \Y$ with $\dim{\Y} = 1$, and so $\Img{\dD{A}{\omega}{\mu}} = \Y$.
    \item There are two cases: either $T \in \Gl{p}$ or $T \not \in \Gl{p}$. In the first case, we conclude by \Cref{cor:sufficient-conditions}(ii). In the second case, we claim that $\Img{T} \cap \Img{B^\ast} = \{0\}$, and then we can again conclude by \Cref{cor:sufficient-conditions}(i). For the claim, note that $\dim{\Img{T}} \leq p-1$, while $\rank{(B^\ast)} = 1$, so $\Img{B^\ast}$ is 1-dimensional. Thus, if $\Img{T} \cap \Img{B^\ast} \neq \{0\}$, it would imply $\Img{B^\ast} \subseteq \Img{T}$, and so $\Img{H^\ast} = \Img{T} + \Img{B^\ast} = \Img{T} \neq \Fp{p}$, giving a contradiction as $\rank{(H^\ast)} = p$.
\end{enumerate}
\end{proof}

\section{Applications}
\label{sec:related-results}

In this final section, we will point out some interesting consequences of the results derived in previous sections. In particular, we look at the limit $\omega \rightarrow \infty$, the question of injectivity of the map $\D{A}{\cdot}b : (\omin, \infty) \rightarrow \Fn$, and investigate some topological aspects of our results.

\subsection{The limit \texorpdfstring{$\omega \rightarrow \infty$}{}}
\label{ssec:minres}

This subsection is mostly for completeness, and we show that for each fixed $b \in \Fn$, the minimizers $x_{b,\omega}$ of \eqref{eq:x_b_w} have a well-defined limit, as $\omega \to \infty$. This is carried out in the next lemma.

\begin{lemma} 
\label{lemma:minres-connection}
Let $A \in \Sym{n} \cap \Gl{n}$, $b \in \F^n$, and $\T \in \Gra{p}{n}$. Let us define $x_{b,\infty} := \argmin_{x \in \T} \| b-Ax \|_2$, which exists uniquely. Then with $\X_b$ and $\X$ defined in \eqref{eq:X-X_b}, we have the following:
\begin{enumerate}[(i)]
    \item As $\omega \to \infty$, $x_{b,\omega} \to x_{b,\infty}$, and $x_{b,\infty} \in \X_b$.
    \item $x_{b,\infty} - x_{b,\omega} \in \X$ for all $\omega \in (\omin, \infty)$, and $b \in \Fn$.
\end{enumerate}
\end{lemma}

\begin{proof}
\begin{enumerate}[(i)]
    \item For $\omega$ sufficiently large, define $\tilde A_w := w^{-1} A + I$, and let $V \in \Fpq{n}{p}$ be semi-unitary such that $\Img{V} = \S$. Then from \Cref{sssec:solution-formula}, with $\T = x_0 + \S$, $x_0 \in \T$, and $\S \in \Gr{p}{n}$ we have 
    \begin{equation}
    \label{eq:minres-formula}
    \begin{split}
        x_{b,\omega} & = x_0 + V(V^\ast A A_\omega^{-1} A V)^{-1} V^\ast A A_\omega^{-1} (b-Ax_0) \\ 
        &= x_0 + V(V^\ast A \tilde A_{\omega}^{-1} A V)^{-1} V^\ast A \tilde A_\omega^{-1} (b - Ax_0).
    \end{split}
    \end{equation}
    Because the map $\T \ni x \mapsto V^\ast (x - x_0) \in \Fp{p}$ is a bijection, we can equivalently write $x_{b,\infty} = x_0 + \argmin_{y \in \Fp{p}} \| (b-Ax_0) - AVy \|_2$, and since $AV$ is full rank, using \Cref{lemma:quadratic_minimization} we have $x_{b,\infty} = x_0 + V (V^\ast A^2 V)^{-1} V^\ast A (b - Ax_0)$ uniquely. Now recall that matrix products are continuous in the matrix entries and the matrix inversion map $\Gl{n} \ni A \mapsto A^{-1} \in \Gl{n}$ is also continuous. Since it is clear that $\tilde A_{\omega} \to I$ as $\omega \to \infty$, by continuity we also have $\tilde A_\omega^{-1} \to I$ as $\omega \rightarrow \infty$. This then implies that $(V^\ast A \tilde A_\omega^{-1} A V)^{-1} \to (V^\ast A^2 V)^{-1}$, and $V^\ast A \tilde A_\omega^{-1} b \to V^\ast A b$. Combining these and using continuity again, we get $V(V^\ast A \tilde A_\omega^{-1} A V)^{-1} V^\ast A \tilde A_\omega^{-1} (b-Ax_0) \to V(V^\ast A^2 V)^{-1} V^\ast A (b-Ax_0)$, which proves $x_{b,\omega} \to x_{b,\infty}$. Finally since any finite dimensional affine subspace is closed, $\X_b$ is closed, and so $x_{b,\infty} \in \X_b$.
    \item By (i), for any $b \in \Fn$, $x_{b,\infty} \in \X_b$ implying that $x_{b,\omega} - x_{b,\infty} \in \Gam{\dim{\X_b}}{\X_b}$, for all $\omega \in (\omin, \infty)$. The result now follows immediately from the definition of $\X$.
\end{enumerate}
\end{proof}

As a direct consequence of part (i) of \Cref{lemma:minres-connection}, if $\T = \K_k(A,b)$ is a Krylov subspace, then $x_{b,\omega}$ converges to the solution of the MINRES subproblem \cite{paige1975solution} as $\omega \to \infty$.

We finish this section by a result extending \Cref{lemma:sufficient-condition} to the case $\mu \to \infty$. Define $J_\infty := H(H^\ast H)^{-1} H - I$ and $d_{b,\omega,\infty} := V^\ast (x_{b,\omega} - x_{b,\infty})$.

\begin{lemma}
\label{lemma:sufficient-condition-infty}
Let $u \in \F^q$ be fixed. If there exist $b \in \F^n$, $\omega \in (\omin, \infty)$, and $t, t' \in \F^q$ such that
\begin{equation}
\label{eq:sufficient-cond-infty}
\begin{cases}
    & T (H^\ast H)^{-1} B^\ast u = N_1 (\omega t' - t) \\
    & B (H^\ast H)^{-1} B^\ast u = N_2 (\omega t' - t) + \begin{bmatrix} B & C-F_\omega \end{bmatrix}Nt' - D^\ast E_\omega^{-1} c'' \\
    & Nt = J_\infty \begin{bmatrix} c \\ c' \end{bmatrix},
\end{cases}
\end{equation}
then $d_{b,\omega,\infty} = (H^\ast H)^{-1} B^\ast u$.
Conversely, if there exist $b \in \F^n$, and $\omega \in (\omin,\infty)$ such that $d_{b,\omega,\infty} = (H^\ast H)^{-1} B^\ast u$, then there exist $t, t' \in \F^q$ such that \eqref{eq:sufficient-cond-infty} holds.
\end{lemma}

\begin{proof}
$x_{b,\infty} \in \X_b$ by \Cref{lemma:minres-connection}(i), and $\X_b \subseteq \Img{V(H^\ast H)^{-1} B^\ast}$ by \Cref{thm:main_result}. So given $d_{b,\omega,\infty}$, there always exists a unique $u \in \F^q$ such that $d_{b,\omega,\infty} = (H^\ast H)^{-1} B^\ast u$. Then from \eqref{eq:x_b_w_explicit_S}, and \eqref{eq:minres-formula} in the proof of \Cref{lemma:minres-connection}, the condition $d_{b,\omega,\infty} = V^\ast (x_{b,\omega} - x_{b,\infty})$ is equivalent to
\begin{equation}
\label{eq:sufficient-condition-infty-proof-1}
V^\ast A A_\omega^{-1} \left\{ AVd_{b,\omega,\infty} - b + AV(V^\ast A^2 V)^{-1} V^\ast A b \right\} = 0.
\end{equation}
Note that $V^\ast A^2 V = H^\ast H$, and $V^\ast A b = H^\ast \begin{bmatrix} c \\ c' \end{bmatrix}$. Thus, with $G_\omega$ defined as in \eqref{eq:F-G}, and using \eqref{eq:equivalent-system-eqs-proof-2}, we find that condition \eqref{eq:sufficient-condition-infty-proof-1} is equivalent to
\begin{equation} 
\label{eq:H_Gomegainv_H_infty} 
H^\ast G_\omega^{-1} \left\{ H d_{b,\omega,\infty} + \begin{bmatrix} 0 \\ D^\ast E_\omega^{-1} c''\end{bmatrix} + (H(H^\ast H)^{-1} H^\ast - I) \begin{bmatrix} c \\ c'\end{bmatrix} \right\} = 0.
\end{equation}
Let $s = (H(H^\ast H)^{-1} H^\ast - I)\begin{bmatrix} c \\ c'\end{bmatrix}$, and note that $I - H(H^\ast H)^{-1} H^\ast$ is an orthogonal projector onto $\Img{H}^\perp$. So $H^\ast s = 0$, or $s = Nt$ for some $t \in \F^q$, and \eqref{eq:H_Gomegainv_H_infty} is equivalent to the following system
\begin{equation}
\label{eq:sufficient-condition-infty-proof-2}
    \begin{cases}
    & H^\ast G_\omega^{-1} \left\{ H d_{b,\omega,\infty} + \begin{bmatrix} 0 \\ D^\ast E_\omega^{-1} c'' \end{bmatrix} + Nt \right\} = 0, \\
    & Nt = (H (H^\ast H)^{-1} H^\ast - I) \begin{bmatrix} c \\ c' \end{bmatrix}.
    \end{cases}
\end{equation}
Finally, the first equation of \eqref{eq:sufficient-condition-infty-proof-2} holds if and only if $H d_{b,\omega,\infty} + \begin{bmatrix} 0 \\ D^\ast E_\omega^{-1} c'' \end{bmatrix} + Nt = G_\omega N t'$, for some $t' \in \Fp{q}$. Expanding $G_\omega$, writing $N = \begin{bmatrix} N_1 \\ N_2 \end{bmatrix}$, and letting $u$ be such that $d_{b,\omega,\infty} = (H^\ast H)^{-1} B^\ast u$ then leads to \eqref{eq:sufficient-cond-infty}. 
Since all steps are equivalences, the converse is true as well which concludes the proof.
\end{proof}

\subsection{Injectivity of the map \texorpdfstring{$\D{A}{\cdot}b : (\omega_{\text{\normalfont min}}, \infty) \rightarrow \Fn$}{}}
\label{ssec:injectivity}

We next give an elegant application of our results: in the setting $\F = \R$, we provide an explanation of the phenomenon first reported and proved in \cite{hallman2018lsmb} (see Section 4.4), that LSMB iterates are a convex combination of LSQR and LSMR iterates. However, we think that our proof is more illuminating and raises other interesting questions. In order to do this, we will first look at the map $\D{A}{\cdot}b : (\omega_{\text{\normalfont min}}, \infty) \rightarrow \Fn$, for some fixed $b \in \Fn$, and specifically ask ourselves when this map is injective.

We can easily formulate a necessary and sufficient condition of injectivity. Suppose $\D{A}{\cdot}b$ is not injective. Then there exists distinct $\omega, \mu \in (\omin,\infty)$ such that $\dD{A}{\omega}{\mu} b = 0$, or equivalently $b \in \Ker{\dD{A}{\omega}{\mu}}$. Conversely, if there exists distinct $\omega, \mu \in (\omin,\infty)$ such that $b \in \Ker{\dD{A}{\omega}{\mu}}$, then $\D{A}{\cdot}b$ is not injective as $\dD{A}{\omega}{\mu} b = 0$. Taking the contrapositive of this statement gives the result that for any $b \in \Fn$, $\D{A}{\cdot}b$ is injective if and only if 
\begin{equation}
\label{eq:global-injectivity-cond}
    b \not \in \underset{\mu \neq \omega}{\bigcup_{\mu, \omega > \omin}} \Ker{\dD{A}{\omega}{\mu}}.
\end{equation}
We note that it can be hard in practice to check condition \eqref{eq:global-injectivity-cond} explicitly; however, whether other simpler equivalent conditions exist or not is not known to us presently. Also recall from \Cref{cor:converse-zero-dim}(iii) that when $\Indsimple{A}{\S} = 0$, we have $\dD{A}{\omega}{\mu} = 0$, for all $\omega, \mu > \omin$; thus the injectivity question of $\D{A}{\cdot}b$ is only interesting when $\Indsimple{A}{\S} \geq 1$. In this case again, \Cref{thm:dim-Xb-zero-condition} provides a sufficient condition on $b$ for $\D{A}{\cdot}b$ to not be injective, but clearly this is not necessary. 

On the other hand, a much easier question that we can resolve almost completely is that of \textit{local injectivity}: we will say that $\D{A}{\cdot}b$ is locally injective at $\omega \in (\omin, \infty)$ if and only if there exists a non-empty open interval $(\omega_1,\omega_2) \subseteq (\omin, \infty)$ containing $\omega$, such that the restriction of $\D{A}{\cdot}b$ to $(\omega_1,\omega_2)$ is injective. Now note that the map $\D{A}{\cdot}b$ satisfies exactly one of the two following conditions:
\begin{enumerate}[(i)]
    \item $\D{A}{\cdot}b$ is a constant map,
    \item $\D{A}{\cdot}b$ is not a constant map.
\end{enumerate}
We already have a complete characterization in \Cref{thm:dim-Xb-zero-condition} of when condition (i) is true, in which case $\D{A}{\cdot}b$ is not locally injective anywhere in $(\omin,\infty)$. It turns out that if condition (ii) holds, then $\D{A}{\cdot}b$ is locally injective almost everywhere. We prove this in the next two lemmas.

\begin{lemma}
\label{lemma:real-analyticity-DA}
If $\F = \R$ (resp. $\F = \C$), $\D{A}{\cdot}b : (\omin, \infty) \rightarrow \Fn$ is a real analytic function (resp. both the real and imaginary parts of the image are real analytic functions).
\end{lemma}

\begin{proof}
We use Proposition 2.2.2 in \cite{krantz2002primer}, repeatedly in this proof. We will also simply say that a function is analytic if its real (if $\F=\R$ and $\F=\C$) and imaginary parts (if $\F = \C$) are both real analytic functions on $(\omin,\infty)$. Let us first consider $A_\omega^{-1}$. Using the adjugate formula (i.e. for $M \in \Gl{n}$, $M^{-1} = \det(M)^{-1} \text{adj}(M)$), we find $A_\omega^{-1} = \det(A_\omega)^{-1} \text{adj}(A_\omega)$. Since $A \in \Sym{n}$, $\det(A_\omega) \in \R$, and since $\omega \in (\omin,\infty)$, $\det(A_\omega) \neq 0$. Hence $\det(A_\omega)^{-1}$ is analytic as $\det(A_\omega)$ is a polynomial in the entries of $A_\omega$, which are in turn affine functions of $\omega$. Furthermore, $(\text{adj}(A_\omega))_{ij}$ is the determinant of a submatrix of $A_\omega$, and thus also analytic. Hence all the entries of $A_\omega^{-1}$ are analytic functions. Since $V$ and $A$ are constant, the entries of $V^\ast A A_\omega^{-1} A V$ are again analytic, and the same is true for its inverse, by the same argument as above (as $\det(V^\ast A A_\omega^{-1} A V) \neq 0$ for $\omega \in (\omin,\infty)$ and $V^\ast A A_\omega^{-1} A V \in \Sym{p}$). Similarly, the entries of $V^\ast A A_\omega^{-1} (b-Ax_0)$ are analytic, and we conclude that $(\omin,\infty) \ni \omega \mapsto x_{b,\omega}$ is analytic.
\end{proof}

\begin{lemma}
\label{lemma:local-injectivity-ae}
Suppose that $\D{A}{\cdot}b$ is not a constant map. Then the set of points in $(\omin,\infty)$ where $\D{A}{\cdot}b$ is not locally injective, has Lebesgue measure (1-dimensional) zero, and is a discrete, countable set.
\end{lemma}

\begin{proof}
Denote $f := \D{A}{\cdot}b$, and $\mathcal{U} := (\omin,\infty)$. By \Cref{lemma:real-analyticity-DA}, $f$ is a real analytic function on $\mathcal{U}$ (in the sense defined in \Cref{lemma:real-analyticity-DA} for the case $\F = \C$, meaning that both the real and imaginary parts are real analytic), and hence so is $f' := \partial_\omega f$, by Proposition 1.1.14 in \cite{krantz2002primer}. We claim that $f'$ is not identically zero. Let $\mathcal{Z} := \{\omega \in (\omin,\infty) \mid \partial_\omega f (\omega) = 0\}$. By the claim, $f'$ is either non-zero or non-constant, and in the first case $\mathcal{Z}$ is empty. In the second case, by \Cref{lemma:real-analyticity-prop,lemma:real-analyticity-complex-range}, $\mathcal{Z}$ has 1-dimensional Lebesgue measure zero, and is a discrete, countable set. It now follows from the inverse function theorem (see for e.g. Theorem C.34 in \cite{lee2013smooth}) that $f$ is locally injective at all points in $\mathcal{U} \setminus \mathcal{Z}$ in both cases, proving the lemma. To prove the claim, note that if $f' = 0$ identically in $\mathcal{U}$, then $\partial_\omega^n f = 0$ identically in $\mathcal{U}$ also, for all $n \geq 1$. By real analyticity of $f$, this would imply that $f$ is constant in $\mathcal{U}$, which is a contradiction.
\end{proof}

We return to the explanation of the convexity phenomenon concerning the LSMB iterates, that was alluded to at the beginning of this subsection. We first provide a lemma below that captures some conditions ensuring convexity, holding when $\dim{\X_b} = 1$ and $\F = \R$.

\begin{lemma}
\label{lemma:convexity-cond}
Let $\F = \R$ and $\dim{\X_b} = 1$. Then we have the following:
\begin{enumerate}[(i)]
    \item Let $\omin < \omega < \mu \leq \infty$ be such that $x_{b,\omega} \neq x_{b,\mu}$. Also assume that for all $\eta \in (\omega, \mu)$, $x_{b,\eta} \neq x_{b,\omega}$ and $x_{b,\eta} \neq x_{b,\mu}$. Then for all $\eta \in (\omega, \mu)$, $x_{b,\eta}$ is a convex combination of $x_{b,\omega}$ and $x_{b,\mu}$.
    \item Let $\omin < \omega < \mu \leq \infty$, and $\D{A}{\cdot}b$ is injective in $(\omega,\mu)$. Then for all $\eta \in (\omega, \mu)$, $x_{b,\omega} \neq x_{b,\eta} \neq x_{b,\mu}$, $x_{b,\omega} \neq x_{b,\mu}$, and $x_{b,\eta}$ is a convex combination of $x_{b,\omega}$ and $x_{b,\mu}$.
\end{enumerate}
\end{lemma}

\begin{proof}
Since $\dim{\X_b} = 1$, there exist $x_0 \in \R^n$, $\R^n \ni w \neq 0$ such that $\X_b = \{ x_0 + w t \mid t \in \R\}$. Let $y_{b,\omega} \in \R$ be uniquely defined such that $x_{b,\omega} = x_0 + w y_{b,\omega}$. Notice that if $\D{A}{\cdot}b$ is injective, so is the map $(\omin,\infty) \ni \xi \mapsto y_{b,\xi}$.
\begin{enumerate}[(i)]
    \item First note that $y_{b,\omega} \neq y_{b,\mu}$ because $x_{b,\omega} \neq x_{b,\mu}$. Now assume the result is not true. Then there exists $\eta \in (\omega,\mu)$ such that either (1) $y_{b,\eta} > \text{max}(y_{b,\omega},y_{b,\mu})$, or (2) $y_{b,\eta} < \text{min}(y_{b,\omega},y_{b,\mu})$. For (1), by continuity of $\xi \mapsto y_{b,\xi}$, this implies there exists $\eta^{-} \in (\omega,\eta)$ if $y_{b,\omega} < \text{max}(y_{b,\omega},y_{b,\mu})$ such that $y_{b,\eta^{-}} = \text{max}(y_{b,\omega},y_{b,\mu})$, and $\eta^{+} \in (\eta,\mu)$ if $y_{b,\mu} < \text{max}(y_{b,\omega},y_{b,\mu})$ such that $y_{b,\eta^{+}} = \text{max}(y_{b,\omega},y_{b,\mu})$, a contradiction in both cases.
    For (2), there exists $\eta^{-} \in (\omega,\eta)$ if $y_{b,\omega} > \text{min}(y_{b,\omega},y_{b,\mu})$ such that $y_{b,\eta^{-}} = \text{min}(y_{b,\omega},y_{b,\mu})$, and $\eta^{+} \in (\eta,\mu)$ if $y_{b,\mu} > \text{min}(y_{b,\omega},y_{b,\mu})$ such that $y_{b,\eta^{+}} = \text{min}(y_{b,\omega},y_{b,\mu})$, and we again obtain a contradiction.
    \item Assume there exist $\eta \in (\omega,\mu)$ such that $y_{b,\eta} = y_{b,\mu}$ (resp. $y_{b,\eta} = y_{b,\omega}$). But $\xi \mapsto y_{b,\xi}$ is continuous and injective over $(\eta,\mu)$ (resp. $(\omega,\eta)$). So this cannot happen. Also by the same reasoning $x_{b,\omega} \neq x_{b,\mu}$; otherwise $\xi \mapsto y_{b,\xi}$ cannot be both injective and continuous in $(\omega,\mu)$. We conclude by applying (i) that for all $\eta \in (\omega, \mu)$, $x_{b,\eta}$ is a convex combination of $x_{b,\omega}$ and $x_{b,\mu}$.
\end{enumerate}
\end{proof}

We now state the theorem that explains the LSMB convexity result.

\begin{theorem}
\label{thm:convexity-spd-case}
Let $\F=\R$, and in addition to the assumptions of \Cref{thm:main_result}, also assume that $A \in \Pos{n}$, and $\T = \S = \K_k(A,b)$. Then for all $\omega \in (0,\infty)$, $x_{b,\omega}$ is a convex combination of $x_{b,0}$ and $x_{b,\infty}$. If $\Indsimple{A}{\S} = 1$, then $\dim{\X_b} = 1$.
\end{theorem}

\begin{proof}
If $\Indsimple{A}{\S} = 0$, then $\dim{\X_b} = 0$ by \Cref{thm:main_result}, and there is nothing to prove. So assume $q = \Indsimple{A}{\S} = 1$ (recall that for Krylov subspaces generated by any $A \in \Fnn$, the index of invariance is at most $1$), which also implies $b \neq 0$. We will show that the following conditions hold: (1) for all $\omega \in (0, \infty)$, $x_{b,\omega} \neq x_{b,0}$, and (2) for all $\omega \in [0, \infty)$, $x_{b,\omega} \neq x_{b,\infty}$. Then by (1) we get that $\dim{\X_b} = 1$ (as $\dim{\X_b} \leq 1$ by \Cref{thm:main_result}), and applying \Cref{lemma:convexity-cond}(i) proves the convex combination part.

For the proof of (1) and (2), assume $n > p+q$. Moreover, to facilitate the proof, we choose $\begin{bmatrix}V & V'\end{bmatrix}$ in a very specific way: the columns of $\begin{bmatrix}V & V'\end{bmatrix}$ are chosen to be the Lanczos vectors with the first column equal to $b / \|b\|_2$, which can be done using the Lanczos tridiagonalization process (see Algorithm 10.1.1 in \cite{golub2013matrix}). With this choice $T \in \Pos{p}$ is tridiagonal, with non-zero sub-diagonal entries (this uses the fact that $\S$ is Krylov and $\Indsimple{A}{\S} = 1$), i.e. $T_{i,i-1} \neq 0$ for $1 \leq i \leq p-1$, while the diagonal entries are positive. It also ensures that $B = \beta e_p^\ast$ for some $\beta \neq 0$ (see Section 10.1.2 in \cite{golub2013matrix} why $B$ and $T$ have these properties), and $c = \|b\|_2 e_1$, where we denote $e_k \in \R^p$ to be the vector satisfying $(e_k)_k = 1$ and all other entries zero. In addition, since $b \in \K_k(A,b)$, we have $c' = 0$, and $c'' = 0$. Now since $T$ is invertible, we choose $N$ according to \Cref{lemma:nullspace-special-cases}(i), i.e. $N_1 = -T^{-1}B^\ast$, and $N_2 = I$. Finally, we will need a claim: $BT^{-1}c \neq 0$. To prove the claim, notice that given the tridiagonal structure of $T$, it follows from Theorem 2.3 in \cite{meurant1992review} that all entries of $T^{-1}$ are non-zero, and so $BT^{-1}c = \beta \|b\|_2 (e_p^\ast T^{-1} e_1) \neq 0$.
\vspace{0.1cm}
\begin{enumerate}[(1)]
    \item Suppose the statement is false, so there exists $\omega \in (0,\infty)$ with $x_{b,\omega} = x_{b,0}$. Then with $d_{b,\omega,\mu}$ defined in \Cref{lemma:HTSH}, we have $d_{b,\omega,0} = 0$, so by \Cref{thm:main_result} $u=0$ is the unique solution to $d_{b,\omega,0} = (H^\ast H)^{-1} B^\ast u$. By \Cref{lemma:sufficient-condition}, there exists $t,t' \in \R$ satisfying the system \eqref{eq:sufficient-cond}, and since $N_1$ is full column rank, the first equation implies $t' = 0$ uniquely. Plugging this into the second equation of \eqref{eq:sufficient-cond} gives $z_{b,\omega,0}(t) = 0$, and then using \eqref{eq:z-b-omega-mu} and $c'' = 0$, we get $\begin{bmatrix} B & C - F_{0} \end{bmatrix}Nt = 0$. Now with $N_1$ and $N_2$ chosen as above, this is equivalent to $(C - B T^{-1} B^\ast - D^\ast E^{-1} D) t = 0$, and since \Cref{lemma:KLprop}(ii) implies that $(C - B T^{-1} B^\ast - D^\ast E^{-1} D)$ is invertible, we get $t=0$ uniquely. Finally the third equation of \eqref{eq:sufficient-cond} implies $J_0 \begin{bmatrix} c \\ 0 \end{bmatrix} = 0$, and now multiplying both sides of this equation by $G_0$, using the definition of $J_0$ in \eqref{eq:Jmu}, and using \eqref{eq:HTSH-positive-proof1}, we obtain $BT^{-1}c = 0$. By the claim above, this is a contradiction.
    \item Suppose again that the statement is false. Then there exists $\omega \in [0,\infty)$ such that $V^\ast(x_{b,\omega} - x_{b,\infty}) = d_{b,\omega,\infty} = 0$. Since $x_{b,\infty} \in \X_b$ by \Cref{lemma:minres-connection}(i), we get from \Cref{thm:main_result} that $u = 0$ is the unique solution to $d_{b,\omega,\infty} = (H^\ast H)^{-1} B^\ast u$. Now consider \Cref{lemma:sufficient-condition-infty}. The first equation of \eqref{eq:sufficient-cond-infty} implies $N_1(\omega t' - t) = 0$, or $\omega t' - t = 0$, as $N_1$ is full column rank. Using this in the second equation of \eqref{eq:sufficient-cond-infty}, and as $c'' = 0$, we get $\begin{bmatrix} B & C-F_\omega\end{bmatrix} N t' = 0$, which is equivalent to $(C - B T^{-1} B^\ast - F_\omega) t' = 0$ with $N_1$ and $N_2$ chosen as above.
    Since $(C - B T^{-1} B^\ast - F_\omega)$ is the Schur complement of the $C$ block of 
    \begin{equation} 
    \begin{bmatrix} T & B^\ast & \\ B & C & D^\ast \\ & D & E + \omega I \end{bmatrix},
    \end{equation}
    which is positive since $\omega \geq 0$ and $A \in \Pos{n}$, it is invertible, from which we conclude that $t' = 0$ uniquely. Plugging this into the third equation of \eqref{eq:sufficient-cond-infty} then gives $H (H^\ast H)^{-1} H^\ast \begin{bmatrix} c \\ 0 \end{bmatrix} = \begin{bmatrix} c \\ 0 \end{bmatrix}$. But $H (H^\ast H)^{-1} H^\ast$ is an orthogonal projector onto $\Img{H}$, so this implies $\begin{bmatrix} c \\ 0 \end{bmatrix} \in \Img{H}$, i.e. there exists $v \in \R^p$ such that $\begin{bmatrix} T \\ B \end{bmatrix}v = \begin{bmatrix} c \\ 0 \end{bmatrix}$, or equivalently $BT^{-1}c = 0$, contradicting the claim above.
\end{enumerate}
\vspace{0.1cm}

Finally, for the case $n = p+q$, we can use the discussion in \Cref{app:appB} to reduce to the $n > p+q$ case, and then repeat the above proof.
\end{proof}

The conclusion of the convexity part of \Cref{thm:convexity-spd-case} can be restated as follows: for all $\omega \in (0,\infty)$, the solutions $x_{b,\omega}$ are a convex combination of the CG and MINRES solutions (note that substituting $\omega = 0$ in the minimization problem \eqref{eq:x_b_w}, when $A \in \Pos{n}$, is the definition of the CG subproblem \cite{hestenes1952methods}). The fact that the LSMB solution is a convex combination of the LSQR and LSMR solutions now also follows by making appropriate substitutions in \eqref{eq:x_b_w}, as was mentioned in \Cref{sec:intro}.

Let us finally provide an example where we have both convexity and injectivity of $\D{A}{\cdot}b$ for almost all $b \in \Fn$, in the setting when $\F = \R$ and $\dim{\X_b} = 1$.

\begin{example}
Consider \Cref{example:converse} with the only added restriction of $\F = \R$. We built $A$ and $\S$ (with $p = q \geq 1$, $n = p + q$) such that, with $\alpha \in \R \backslash \{-1,1\}$
\[ d_{b,\omega,\mu} = V^*(x_{b,\omega} - x_{b,\mu}) = \frac{(\mu - \omega)(\alpha^2 - 1)(c - \alpha c')}{(\alpha^3+\alpha^2 \omega - \alpha + \omega)(\alpha^3+\alpha^2 \mu - \alpha + \mu)}, \]
where $c = V^\ast b$, $c' = V'^\ast b$, and $\omega,\mu \in (\omin, \infty)$. From the analysis in \Cref{example:converse}, we know the denominator is positive and hence non-zero, and $\alpha^2 - 1 \neq 0$. So we conclude that, if $c - \alpha c' \neq 0$, then for any $\omega, \mu \in (\omin,\infty)$ with $\omega \neq \mu$, we have $x_{b,\omega} \neq x_{b,\mu}$. This implies that if $c - \alpha c' \neq 0$, the map $\D{A}{\cdot}b : \omega \in (\omin,\infty) \to x_{b,\omega} \in \Fn$ is injective. By applying \Cref{lemma:convexity-cond}(ii) one obtains a convexity statement (for any choice of $\omega > \omin$ in the lemma). Now the set $\mathcal{U} := \{b \in \Fn \mid c - \alpha c' = 0 \}$ is clearly closed under addition and scalar multiplication, hence a subspace of $\Fn$ with $\dim{\mathcal{U}} = p < n$. Thus, $\mathcal{U}$ has $n$-dimensional Lebesgue measure zero in $\Fn$.
\end{example}

\subsection{Topological consequences}
\label{ssec:top}

We point out some topological consequences of  the tridiagonal block decomposition obtained in
\eqref{eq:property_VTAV}. In what follows, we will identify $\Rnn$ and $\Cnn$ with the smooth manifolds $\R^{n^2}$ and $\R^{2n^2}$ respectively. Similarly, $\Gl{n}$, $\Sym{n}$, and $\Pos{n}$ will denote the corresponding subsets of $\R^{n^2}$ (resp. $\R^{2n^2}$) under this identification, when $\F = \R$ (resp. $\F = \C$). We also state two useful facts: (i) $\Sym{n}$ is diffeomorphic to $\R^{n(n-1)/2 + n}$ (resp. $\R^{2n(n-1)/2 + n}$) if $\F = \R$ (resp. $\F = \C$), and (ii) the set of full rank matrices in $\Fpq{p}{q}$ is an open subset. Finally recall the function $\vec{\cdot}$ introduced in \Cref{sssec:hilbert-spaces}. We remind the reader of the following well known property: for compatible shapes, $AXB=C$ if and only if $(B^\top \otimes A)\vec{X} = \vec{C}$, where $\otimes$ is the Kronecker product.

We first define a few matrix sets that we will study below. Let $\S, \S' \in \G{n}$ be such that $\S \subseteq \S' \subseteq \Fn$, and $n \geq 1$. Define the sets
\begin{equation}
\label{eq:matrix-sets}
\begin{split}
    \M(\S, \S') &:= \{ A \in \Fnn \mid \S + A\S = \S' \} \\
    \Minv (\S, \S') &:= \M(\S, \S') \cap \Gl{n} \\
    \Msym (\S, \S') &:= \M(\S, \S') \cap \Sym{n} \\
    \Msyminv (\S, \S') &:= \Msym(\S,\S') \cap \Gl{n} \\
    \Mpos (\S, \S') &:= \M(\S, \S') \cap \Pos{n}.
\end{split}
\end{equation}
We give each of these sets the subset topology from $\Fnn$. The following lemma concerns the existence of these matrix sets.
\begin{lemma}
\label{lemma:existence}
Let $\S, \S' \in \G{n}$ be such that $\S \subseteq \S' \subseteq \Fn$ with $n \geq 1$, and let $p = \dim{\S}$, and $q = \dim{\S'} - \dim{\S}$. Then all the matrix sets in \eqref{eq:matrix-sets} are non-empty if $q \leq p$, and empty if $q > p$.
\end{lemma}

\begin{proof}
For the case $q > p$, $\M(\S, \S')$ is empty because for any $A \in \M(\S, \S')$, $\Indsimple{A}{\S} \leq \dim{\S}$ by \Cref{lemma:basic_properties}(i). We now assume $q \leq p$. Since $\Mpos (\S, \S') \subseteq \Msym (\S, \S') \subseteq \M(\S, \S')$, and $\Mpos (\S, \S') \subseteq \Msyminv(\S,\S') \subseteq \Minv (\S, \S')$, it suffices to show that $\Mpos (\S, \S')$ is non-empty. If $q = 0$, we have $I \in \Mpos (\S, \S')$ because $\Ind{I}{\S}=0$ for all $\S \in \G{n}$; so we can assume $p \geq q \geq 1$. Now choose $V \in \F^{n \times p}$, $V' \in \F^{n \times q}$ semi-unitary, such that $\Img{V} = \S$, $\Img{\begin{bmatrix} V & V' \end{bmatrix}} = \S'$, and denote the columns of $V$ by $v_1,\dots,v_p$, and those of $V'$ by $v_{p+1},\dots,v_{p+q}$. If $\S' \neq \Fn$, choose $V'' \in \Fpq{n}{(n-p-q)}$ semi-unitary with $\Img{V''}=\S'^\perp$, and denote its columns $v_{p+q+1},\dots,v_n$. Now let $A$ be the linear map defined by its action on the orthonormal basis set $\{v_1,\dots,v_n\}$ as follows:
\begin{equation}
\begin{split}
    A(v_i) &= v_{p+i}, \;\;\; i \in \{1,\dots,q\}, \\
    A(v_{p+i}) &= v_i, \;\;\;\;\;\;\; i \in \{1,\dots,q\}, \\
    A(v_i) &= v_i, \;\;\;\;\;\;\; i \in \{q+1,\dots,p\} \cup \{p+q+1,\dots,n\}.
\end{split}
\end{equation}
Then by construction $\S + A\S = \S'$, so $\Ind{A}{\S} = q$, and moreover we have $v_j^\ast A v_i = v_i^\ast A v_j$, for all $1 \leq i,j \leq n$. By linearity we get $u^\ast A w = w^\ast A u$, for all $u,w \in \Fn$, and $A \in \Sym{n}$. For example, for $q \geq 1$, $p-q \geq 1$ and $n-p-q \geq 1$, this linear map has the following tridiagonal block decomposition (see \eqref{eq:property_VTAV}):
\begin{equation}
    \begin{bmatrix} V & V' & V'' \end{bmatrix}^\ast A \begin{bmatrix} V & V' & V'' \end{bmatrix} =
    \left[
    \begin{array}{c c|c|c} 
        0_q & & I_q  & \\
         & I_{p-q} & & \\ \hline
        I_q & & & \\ \hline
        & & & I_{n-p-q}
    \end{array}
    \right],
\end{equation}
where $0_q \in \Fpq{q}{q}$ is a zero-matrix, $I_r \in \Fpq{r}{r}$ is the identity matrix for $r \in \{q,p-q,n-p-q\}$, and empty blocks are zero. 
Finally for any $\omega > -\lambda_{\text{min}}(A)$, $A_\omega := A+ \omega I \in \Pos{n}$, and  $\S + A_\omega \S = \S + A \S = \S'$, i.e. $A_\omega \in \Mpos (\S, \S')$. 
\end{proof}

Given the above result, in the case $q \leq p$, it begs the question whether the sets $\M(\S, \S')$, $\Minv (\S, \S')$, $\Msym (\S, \S')$, $\Msyminv(\S,\S')$ and $\Mpos (\S, \S')$ have any interesting structure or not. The next few lemmas and corollaries will show that these are all in fact smooth (real) manifolds. For the proofs in the rest of this subsection, we define $\alpha = 1$, if $\F = \R$, and $\alpha = 2$, if $\F = \C$.

\begin{lemma}
\label{lemma:embedding-M}
Assume the notations of \Cref{lemma:existence}, $n \geq 1$ and $q \leq p$. Then $\M(\S, \S')$ is a smooth embedded submanifold of $\R^{\alpha n^2}$ of dimension $\alpha(n^2 - p(n-p-q))$.
\end{lemma}

\begin{proof}
\label{proof:embedding-M}
If $p=0$ then $q = 0$, so $\M(\S, \S') = \R^{\alpha n^2}$, which proves the lemma for the case $p = 0$. So assume $p \geq 1$, and by \Cref{lemma:existence} $\M(\S, \S')$ is non-empty.  Let us also assume that $q \geq 1$ and $n - p - q \geq 1$, and first prove the lemma in this setting. So let $V \in \F^{n \times p}$, $V' \in \F^{n \times q}$ and $V'' \in \F^{n \times (n-p-q)}$ be such that $\Img{V} = \S$, $\Img{\begin{bmatrix} V & V' \end{bmatrix}} = \S + \S'$, and $W := \begin{bmatrix} V & V' & V''\end{bmatrix}$ is unitary. Now define the set of matrices $\M := \{ A \in \Fnn \mid \rank{(A_{p+1:p+q,1:p})} = q, A_{p+q+1:n,1:p} = 0\}$. Then for any $K \in \M(\S, \S')$, by \Cref{cor:tridiag-blk-decomp}, $W^\ast K W \in \M$, so $\M$ is non-empty.
For clarity, we split the proof into three steps.

\paragraph{Step 1: Smooth manifold structure}
Our first goal is to endow $\M$ with a smooth manifold structure. Let $f(n,p,q) := \alpha (n^2 - p(n-p-q))$. Clearly, we can identify $\M$ with the set $\mathcal{U} \times \R^{\alpha(n(n-p) + p^2)}$, where $\mathcal{U} \subseteq \R^{\alpha pq}$ is the set of all $q \times p$ matrices of rank $q$ (i.e. full rank). Now it is well known that $\mathcal{U}$ is a non-empty open subset of $\R^{\alpha pq}$; so $\M$ is an open subset of $\R^{f(n,p,q)}$, and we equip $\M$ with the subspace topology from $\R^{f(n,p,q)}$. Then with the standard smooth manifold structure inherited from $\R^{f(n,p,q)}$, $\M$ is a smooth manifold of dimension $f(n,p,q)$. Now define the map $F : \M \rightarrow \Fnn$ as $F(A) = W A W^\ast$, for all $A \in \M$. We claim that $F$ is a smooth embedding, which we prove below. Assuming the claim and applying Proposition 5.2 in \cite{lee2013smooth}, we conclude that with the subset topology inherited from $\Fnn$, $F(\M)$ is a smooth embedded submanifold of $\Fnn$, and moreover $F(\M)$ admits a unique smooth structure with the property that $F: \M \rightarrow F(\M)$ is a diffeomorphism. If we pick any $K \in \M(\S,\S')$, we note that $K = W (W^\ast K W) W^\ast$ as $W$ is unitary, and moreover $W^\ast K W \in \M$; so $K = F(W^\ast K W)$, and we conclude that $F(\M) = \M(\S,\S')$. This proves that $\M(\S,\S')$ is a smooth embedded submanifold of $\Fnn$ of dimension $f(n,p,q)$.

\paragraph{Step 2: $F$ is smooth and a topological embedding}
It remains to prove the smooth embedding claim. Recall that a smooth embedding is a map that is smooth, a topological embedding and an immersion (see Chapter 4 of \cite{lee2013smooth}). We first show that $F$ is smooth and a topological embedding. For the latter, it suffices to show that when $F(\M)$ is equipped with the subspace topology, then $F:\M \rightarrow F(\M)$ is continuous and there exists a continuous map $F^{-1}: F(\M) \rightarrow \M$ such that $FF^{-1} = \text{Id}$. Now define the map $\bar{F}: \R^{f(n,p,q)} \rightarrow \Fnn$ as $A \mapsto W A W^\ast$, and notice that $\bar{F}$ is a linear map between (real) vector spaces, and hence smooth as a map between smooth manifolds. Thus $F = \bar{F} |_{\M}$ is also a smooth map between manifolds, as $\M \subseteq \R^{f(n,p,q)}$ is open. In particular, if $F(\M)$ is equipped with the subspace topology, then $F:\M \rightarrow F(\M)$ is continuous. Next define the map $F^{-1}: F(\M) \rightarrow \M$ as $K \mapsto W^\ast K W$, and notice that both $FF^{-1} = \text{Id}$ and $F^{-1}F = \text{Id}$ (by unitarity of $W$). Also $F^{-1}$ is continuous: for any convergent sequence $K_i \rightarrow K$ in $F(\M)$, we have $||F^{-1}(K_i) - F^{-1}(K)||_{\text{F}} = ||K_i - K||_{\text{F}} \rightarrow 0$ in the Frobenius norm, so $F^{-1}(K_i) \rightarrow F^{-1}(K)$ in $\M$. Thus we have proved that $F$ is smooth and a topological embedding.

\paragraph{Step 3: $F$ is an immersion}
We now show that $F$ is an immersion, i.e. for all $A \in \M$, the differential $(dF)_{A}$ is injective. First notice that if $A \in \M$, then $\vec{F(A)} = \vec{WAW^*} = (\overline{W} \otimes W) \vec{A}$, where $\overline{W}$ denotes the complex conjugate of $W$. Also as $W$ is unitary, so is $\overline{W}$, and it follows that $\widehat{W} := \overline{W} \otimes W$ is unitary, because $\widehat{W}^\ast \widehat{W} = (\overline{W}^\ast \otimes W^\ast) (\overline{W} \otimes W) = (\overline{W}^\ast \overline{W}) \otimes (W^\ast W) = I$. Now pick any $A \in \M$, and pick a smooth chart $(\mathcal{V},\text{Id}_\mathcal{V})$ in $\M$, so that $A \in \mathcal{V}$, and similarly pick a smooth chart $(\mathcal{V}',\text{Id}_{\mathcal{V}'})$ in $\Fnn$ such that $F(A) \in \mathcal{V}'$. In these local coordinates, for all $A' \in \mathcal{V}$, the map $L := \text{Id}_{\mathcal{V}'} \circ F \circ \text{Id}_{\mathcal{V}}^{-1} : \text{Id}_{\mathcal{V}} (\mathcal{V}) \rightarrow \text{Id}_{\mathcal{V}'} (\mathcal{V}')$ has the form $L(\text{Id}_{\mathcal{V}} (A')) = \widehat{W} \vec{A'}$ if $\F = \R$, and $L(\text{Id}_{\mathcal{V}} (A')) = \begin{bmatrix} \Re(\widehat{W} \vec{A'}) \\ \Im(\widehat{W} \vec{A'}) \end{bmatrix}$ if $\F = \C$. The constraint $A'_{p+q+1:n,1:p} = 0$ means the corresponding entries in $\vec{A'}$ are fixed to zero. 
So let $\widetilde{W} \in \F^{n^2 \times (n^2 -p(n-p-q))}$ be the matrix formed by taking the subset of columns of $\widehat{W}$, corresponding to the entries of $\vec{A}$ not fixed to zero. Now for $\F = \R$, the differential of $L$ at $\text{Id}_{\mathcal{V}} (A)$ is exactly $\widetilde{W}$, which implies $\Ker{\widetilde{W}} = \{0\}$. In the case $\F = \C$, we have $\begin{bmatrix} \Re(\widehat{W} \vec{A}) \\ \Im(\widehat{W} \vec{A}) \end{bmatrix} = \begin{bmatrix} \Re(\widehat{W}) & - \Im(\widehat{W}) \\ \Im(\widehat{W}) & \Re(\widehat{W})\end{bmatrix} \begin{bmatrix} \Re(\vec{A}) \\ \Im(\vec{A})\end{bmatrix}$ and the differential of $L$ at $\text{Id}_{\mathcal{V}}(A)$ is
$\begin{bmatrix} \Re(\widetilde{W}) & -\Im(\widetilde{W}) \\
\Im(\widetilde{W}) & \Re(\widetilde{W}) \end{bmatrix}$,
which is again full rank since $\widetilde{W}$ is.
This finishes the proof of the lemma for the case $q \geq 1$ and $n - p - q \geq 1$.

For the case $q = 0$ and $n-p-q \geq 1$, the above proof is modified as follows. In this case, there is no $V'$; so $\S = \S'$, $\Img{V} = \S$, and $W := \begin{bmatrix} V & V''\end{bmatrix}$. We define $\M := \{ A \in \Fnn \mid A_{p+q+1:n,1:p} = 0\}$, and in the proof we identify $\M$ with $\R^{f(n,p,q)}$, which gives $\M$ the standard topology and smooth manifold structure of $\R^{f(n,p,q)}$. The rest of the proof remains the same. Finally, for the case $n-p-q=0$ ($q=0$ or $q \geq 1$), there is no $V''$ in both cases; so we simply redefine $W := \begin{bmatrix} V & V'\end{bmatrix}$ (resp. $W = V$) if $q \geq 1$ (resp. $q = 0$), and we repeat the entire argument above. In particular, $\M$ is defined as $\M := \{ A \in \Fnn \mid \rank{(A_{p+1:p+q,1:p})} = q\}$, and $\M = \{ A \in \Fnn\}$ in the respective cases.
\end{proof}

\begin{corollary}
\label{cor:embedding-Minv}
Assume the notations of \Cref{lemma:existence}, $n \geq 1$ and $q \leq p$. Then 
\begin{enumerate}[(i)]
    \item $\Minv(\S, \S')$ is a smooth embedded submanifold of $\Gl{n}$ of dimension $\alpha(n^2 - p(n-p-q))$.
    \item $\Minv(\S, \S')$ is an open, dense subset of $\M(\S, \S')$.
\end{enumerate}
\end{corollary}

\begin{proof}
\begin{enumerate}[(i)]
    \item \label{item:embedding-Minv-i} Let $f(n,p,q)$ be defined as in the proof of \Cref{lemma:embedding-M}. The proof is a simple consequence of the existence of slice charts for smooth embedded submanifolds (see Chapter 5 of \cite{lee2013smooth} for definitions). We know from \Cref{lemma:embedding-M} that $\M(\S,\S')$ is a smooth embedded submanifold of $\R^{\alpha n^2}$, of dimension $f(n,p,q)$. Now pick any $A \in \Minv(\S, \S')$. Since $A \in \M(\S, \S')$ also, by Theorem 5.8 in \cite{lee2013smooth}, there exists a smooth slice chart $(\mathcal{V},\phi)$ for $\M(\S,\S')$ in $\R^{\alpha n^2}$, such that $A \in \mathcal{V}$. Define $\mathcal{V}' = \mathcal{V} \cap \Gl{n}$, $\phi' = \phi |_{\mathcal{V} \cap \Gl{n}}$, and note that $A \in \mathcal{V}'$. Since $\Gl{n}$ is open in $\R^{\alpha n^2}$, this implies that $(\mathcal{V}',\phi')$ is a smooth slice chart for $\Minv(\S, \S')$ in $\Gl{n}$ satisfying the local $f(n,p,q)$-slice condition. Finally by Theorem 5.8 in \cite{lee2013smooth} (the converse part), we conclude that $\Minv(\S, \S')$ is a smooth embedded submanifold of $\Gl{n}$ of dimension $f(n,p,q)$.
    \item Since $\M(\S, \S')$ has the subset topology from $\R^{\alpha n^2}$, and $\Gl{n}$ is open in $\R^{\alpha n^2}$, the set $\M(\S, \S') \cap \Gl{n}$ is open in $\M(\S, \S')$. To show density, first if $p=q=0$, then $\M(\S, \S') = \R^{\alpha n^2}$, and $\Minv(\S, \S') = \R^{\alpha n^2} \cap \Gl{n}$; the result is now true as $\R^{\alpha n^2} \cap \Gl{n}$ is dense in $\R^{\alpha n^2}$. Now let $p \geq 1$, and consider $\M$ and $F$ as defined in the proof of \Cref{lemma:embedding-M}. Since $F : \M \rightarrow F(\M)$ is a topological embedding, and for all $A \in \M$ we have $\det{(A)} = \det{(F(A))}$, it suffices to prove that $\M \cap \Gl{n}$ is dense in $\M$. So let $A \in \M$, $A \not\in \Gl{n}$, and consider the Schur decomposition $A = U T U^\ast$, with unitary $U$ and upper-triangular $T$. Then $A$ has some zero eigenvalues, located on the diagonal of $T$. Hence there exists $\{\epsilon_i\}_{i=1}^\infty$ such that $\epsilon_i \to 0$, and for all $i \geq 1$, $T + \epsilon_i I \in \Gl{n}$ implying $A_i = U(T + \epsilon_i I)U^\ast = A + \epsilon_i I\in \Gl{n}$. Thus $A_i \rightarrow A$, and furthermore $A_i$ and $A$ agree at all off-diagonal entries, so $A_i \in \M \cap \Gl{n}$ for all $i \geq 1$, finishing the proof.
\end{enumerate}
\end{proof}

\begin{lemma}
\label{lemma:embedding-Msym}
Assume the notations of \Cref{lemma:existence}, $n \geq 1$, and $q \leq p$. Then $\Msym(\S, \S')$ is a smooth embedded submanifold of $\Sym{n}$ (which is diffeomorphic to $\R^{\alpha n(n-1)/2 + n}$) having dimension $\alpha (n(n-1)/2 - p(n-p-q) ) + n$.
\end{lemma}

\begin{proof}
This proof follows the structure of the proof of \Cref{lemma:embedding-M}, so the quantities appearing in this proof, unless redefined, are the same. The two differences are different dimensions of the manifolds involved, and the proof of immersion, which has to take the Hermitian property into account.
As stated before, $\Sym{n}$ can be identified with $\R^{\alpha n(n-1)/2 + n}$, as it is diffeomorphic to it. If $p = 0$, then $q = 0$, and $\Msym(\S,\S') = \Sym{n}$ which proves the lemma. So now let $q \geq 1$, and $n-p-q \geq 1$.
Let $\M := \{A \in \Sym{n} \mid \rank{(A_{p+1:p+q,1:p})} = q, A_{p+q+1:n,1:p} = 0\}$. Then by \Cref{lemma:existence} $\Msym(\S, \S')$ is non-empty; so for any $K \in \Msym(\S, \S')$, $W^\ast K W \in \M$ by \Cref{cor:tridiag-blk-decomp}, and $\M$ is also non-empty.

\paragraph{Step 1: Smooth manifold structure} We need to account for the Hermitian property, so now define $f(n,p,q) := \alpha ( n(n-1)/2 - p(n-p-q) ) + n$. $\M$ can then be identified with $\mathcal{U} \times \R^{g(n,p,q)}$, where $g(n,p,q)=\alpha(p(p-1)/2)+p+\alpha((n-p)(n-p-1)/2)+(n-p)$; thus $\M$ is an open subset of $\R^{f(n,p,q)}$, from which it inherits the subset topology and smooth manifold structure.
Defining $F: \M \to \Sym{n}$ by $A \mapsto W A W^\ast$, and assuming it is a smooth embedding, we conclude similarly as in \Cref{lemma:embedding-M} that $F$ is a diffeomorphism, $F(\M) = \Msym(\S,\S')$, and $F(\M)$ is a smooth embedded submanifold of $\Sym{n}$ of dimension $f(n,p,q)$.

\paragraph{Step 2: $F$ is smooth and a topological embedding} We redefine $\bar{F}: \R^{f(n,p,q)} \rightarrow \Sym{n}$ by $A \mapsto W A W^\ast$. This step otherwise remains unchanged.

\paragraph{Step 3: $F$ is an immersion}
We finally prove that for all $A \in \M$, $(dF)_{A}$ is injective. As in \Cref{lemma:embedding-M}, $\vec{F(A)} = \widehat{W} \vec{A}$, where $\widehat{W}$ is unitary. Fixing $A \in \M$, pick smooth charts $(\mathcal{V}, \text{Id}_{\mathcal{V}})$ in $\M$ and $(\mathcal{V'}, \text{Id}_{\mathcal{V'}})$ in $\Sym{n}$, so that $A \in \mathcal{V}$ and $F(A) \in \mathcal{V'}$, and in these local coordinates define $L$ similarly as in \Cref{lemma:embedding-M}. Our goal is again to show that the differential of $L$ is injective.
We first prove the statement for $\F = \C$. 
Let $a = \vec{A}$, and $h = \vec{F(A)}$. We use the indices $d$, $l$, and $u$ to denote an enumeration of the diagonal, strictly lower triangular and strictly upper triangular entries respectively, of both $A$ and $F(A)$. The enumerations $l$ and $u$ satisfy the property that if $(i,j)$ is the $k$\textsuperscript{th} entry of $l$, then $(j,i)$ is the $k$\textsuperscript{th} entry of $u$. We also use exponents $R$ and $I$ (there will be no scope of confusion here with the identity matrix) to denote the real and imaginary parts of vectors and matrices.
After an appropriate reordering, we can rewrite $\vec{F(A)} = \widehat{W} \vec{A}$ as
\begin{equation} 
\label{eq:h-hatW-a}
\begin{bmatrix} h_d \\ h_l \\ h_u \end{bmatrix} = \begin{bmatrix}  \W_{dd} & \W_{dl} & \W_{du} \\ \W_{ld} & \W_{ll} & \W_{lu} \\ \W_{ud} & \W_{ul} & \W_{uu} \end{bmatrix} \begin{bmatrix} a_d \\ a_l \\ a_u \end{bmatrix}.
\end{equation}
Then using $a_l = \overline{a_u}$, and by separating the real and imaginary parts, we find
\begin{equation}
\begin{bmatrix} h_d^R \\ h_l^R \\ h_l^I \end{bmatrix} = 
\underbrace{\begin{bmatrix}
\W_{dd}^R & \W_{dl}^R + \W_{du}^R & -\W_{dl}^I + \W_{du}^I \\
\W_{ld}^R & \W_{ll}^R + \W_{lu}^R & -\W_{ll}^I + \W_{lu}^I \\
\W_{ld}^I & \W_{ll}^I + \W_{lu}^I &  \W_{ll}^R - \W_{lu}^R \end{bmatrix}}_{=\widetilde{W}} \begin{bmatrix}
a_d^R \\ a_l^R \\ a_l^I 
\end{bmatrix}.
\end{equation}
We now show that $\widetilde{W} \in \R^{n^2 \times n^2}$ is full column rank. Starting with $\widehat{W} \in \Gl{n^2}$, we show that $\widehat{W}$ can be transformed into $\widetilde{W}$ while remaining full column rank. First, since $\widehat{W}$ is full column rank, so is
\begin{equation} 
\widehat{W}^{(1)} := \begin{bmatrix} \widehat{W}^R & -\widehat{W}^I \\
\widehat{W}^I & \widehat{W}^R \end{bmatrix} = \begin{bmatrix}
        \W_{dd}^R & \W_{dl}^R & \W_{du}^R & -\W_{dd}^I & -\W_{dl}^I & -\W_{du}^I \\
        \W_{ld}^R & \W_{ll}^R & \W_{lu}^R & -\W_{ld}^I & -\W_{ll}^I & -\W_{lu}^I \\
        \W_{ud}^R & \W_{ul}^R & \W_{uu}^R & -\W_{ud}^I & -\W_{ul}^I & -\W_{uu}^I \\
        \W_{dd}^I & \W_{dl}^I & \W_{du}^I &  \W_{dd}^R &  \W_{dl}^R &  \W_{du}^R \\
        \W_{ld}^I & \W_{ll}^I & \W_{lu}^I &  \W_{ld}^R &  \W_{ll}^R &  \W_{lu}^R \\
        \W_{ud}^I & \W_{ul}^I & \W_{uu}^I &  \W_{ud}^R &  \W_{ul}^R &  \W_{uu}^R
    \end{bmatrix} \in \R^{2n^2 \times 2n^2}.
\end{equation}
We will now refer to the rows and columns of $\widehat{W}^{(1)}$ in order (i.e. top to bottom, and left to right respectively), by $d^R$, $l^R$, $u^R$, $d^I$, $l^I$, $u^I$.
We first drop the columns $d^I$.
We then replace columns $l^R$ and $u^R$ by their sum, and columns $l^I$ and $u^I$ by their differences. These operations keep the matrix full column rank and so
\begin{equation}
    \widehat{W}^{(2)} := 
    \begin{bmatrix}
        \W_{dd}^R & \W_{dl}^R + \W_{du}^R & -\W_{dl}^I + \W_{du}^I \\
        \W_{ld}^R & \W_{ll}^R + \W_{lu}^R & -\W_{ll}^I + \W_{lu}^I \\
        \W_{ud}^R & \W_{ul}^R + \W_{uu}^R & -\W_{ul}^I + \W_{uu}^I \\
        \W_{dd}^I & \W_{dl}^I + \W_{du}^I &  \W_{dl}^R - \W_{du}^R \\
        \W_{ld}^I & \W_{ll}^I + \W_{lu}^I &  \W_{ll}^R - \W_{lu}^R \\
        \W_{ud}^I & \W_{ul}^I + \W_{uu}^I &  \W_{ul}^R - \W_{uu}^R
    \end{bmatrix}
    \in \R^{2n^2 \times n^2}
\end{equation}
is also full column rank. 
We next argue that the rows $d^I$, $u^R$ and $u^I$ of $\widehat{W}^{(2)}$ can be removed without changing the rank of the result. Since $F(A) \in \Sym{n}$ for any $A \in \Sym{n}$, we have $h_l^R = h_u^R$, $h_l^I = -h_u^I$, and $h_d^I = 0$, for all $a_d^R$, $a_l^R$ and $a_l^I$.
Now from \eqref{eq:h-hatW-a}, one  obtains after regrouping terms
\begin{equation}
\label{eq:lu-real-parts}
\begin{split}
    h_l^R &= \begin{bmatrix} \W_{ld}^R & \W_{ll}^R + \W_{lu}^R & -\W_{ll}^I + \W_{lu}^I \end{bmatrix} \begin{bmatrix} a_d^R \\ a_l^R \\ a_l^I \end{bmatrix}, \\
    h_u^R &= \begin{bmatrix} \W_{ud}^R & \W_{ul}^R + \W_{uu}^R & -\W_{ul}^I + \W_{uu}^I \end{bmatrix}
    \begin{bmatrix} a_d^R \\ a_l^R \\ a_l^I \end{bmatrix},
\end{split}
\end{equation}
and so the rows $l^R$ and $u^R$ of $\widehat{W}^{(2)}$ are equal, using $h_l^R = h_u^R$. Similarly using $h_l^I = - h_u^I$ and reasoning similarly, we get that the rows $l^I$ and $u^I$ of $\widehat{W}^{(2)}$ are negative of each other. Finally $h_d^I = 0$ implies that row $d^I$ of $\widehat{W}^{(2)}$ is zero. So we can remove the three rows $u^R$, $d^I$ and $u^I$ from $\widehat{W}^{(2)}$, and obtain $\widetilde{W}$ with same rank as $\widehat{W}^{(2)}$. But $\rank{(\widehat{W}^{(2)})} = n^2$, and $\widetilde{W}$ is square, so it is invertible.
The constraint $A_{p+q+1:n,1:p} = 0$ is only setting some of the entries in $\vec{A}$ to zero, which corresponds to removing the corresponding columns of $\widetilde{W}$, after which we exactly get the differential of $L$, and we conclude that its kernel is trivial. The case $\F = \R$ is a particular case of $\F = \C$, and we repeat the above argument with appropriate changes, removing all matrices and vectors corresponding to the imaginary parts.

The remaining two cases $q=0$ or $n-p-q=0$ are argued similarly as in the proof of \Cref{lemma:embedding-M}. The proper definitions of $\M$ to use in the proof are now as follows: (i) if $q=0$, and $n-p-q \geq 1$, then let $\M := \{A \in \Sym{n} \mid A_{p+q+1:n,1:p} = 0\}$, (ii) if $q=0$, and $n-p-q = 0$, define $\M := \{A \in \Sym{n}\}$, and (iii) if $q \geq 1$, and $n-p-q = 0$, then define $\M := \{A \in \Sym{n} \mid \rank{(A_{p+1:p+q,1:p})} = q\}$.
\end{proof}

\begin{corollary}
\label{cor:embedding-Msyminv}
Assume the notations of \Cref{lemma:existence}, $n \geq 1$ and $q \leq p$. Then 
\begin{enumerate}[(i)]
    \item $\Msyminv(\S, \S')$ (resp. $\Mpos(\S,\S')$) is a smooth embedded submanifold of $\Gl{n} \cap \Sym{n}$ (resp. $\Pos{n}$) of dimension $\alpha ( n(n-1)/2 - p(n-p-q) ) + n$.
    \item $\Msyminv(\S, \S')$ and $\Mpos(\S,\S')$ are open subsets of $\Msym(\S, \S')$.
    \item $\Msyminv(\S,\S')$ is a dense subset of $\Msym(\S,\S')$.
\end{enumerate}
\end{corollary}

\begin{proof}
Observe that $\Gl{n}\cap\Sym{n}$ and $\Pos{n}$ are both open in $\Sym{n}$.
\begin{enumerate}[(i)]
    \item 
    \begin{sloppypar}
    Define $f(n,p,q)$ as in the proof of \Cref{lemma:embedding-Msym}. Now to prove that $\Msyminv(\S, \S')$ (resp. $\Mpos(\S,\S')$) is a smooth embedded submanifold of $\Gl{n} \cap \Sym{n}$ (resp. $\Pos{n}$), we simply follow the proof of \Cref{cor:embedding-Minv}(i), with the following replacements: $\R^{\alpha n^2}$ by $\Sym{n}$, $\M(\S,\S')$ by $\Msym(\S,\S')$, $\Minv(\S,\S')$ by $\Msyminv(\S,\S')$ (resp. $\Mpos(\S,\S')$), and $\Gl{n}$ by $\Gl{n} \cap \Sym{n}$ (resp. $\Pos{n}$), and use the observation above. Also \Cref{lemma:embedding-Msym} should be used in place of \Cref{lemma:embedding-M}.
    \end{sloppypar}
    \item 
    \begin{sloppypar}
    Since $\Msym(\S,\S')$ has the subset topology from $\Sym{n}$, it then follows from the observation above that $\Msyminv(\S,\S')$ and $\Mpos(\S,\S')$ are both open in $\Msym(\S,\S')$.
    \end{sloppypar}
    \item Follow the same steps as in the proof of the density part of \Cref{cor:embedding-Minv}(ii), with the replacements as stated in the proof of part (i) of this lemma, $F$ and $\M$ now defined as in the proof of \Cref{lemma:embedding-Msym}, and noticing that $A_i = A + \epsilon_i I \in \Sym{n} \cap \Gl{n}$, if $A \in \Sym{n}$.
\end{enumerate}
\end{proof}

Finally, we consider a last related matrix manifold corresponding to Hermitian matrices with a special property.

\begin{lemma}
\label{lemma:density-MsymT}
\begin{sloppypar}
Assume the notations of \Cref{lemma:existence}, $p\geq 1$, $n \geq 1$ and $q \leq p$. Let $V \in \F^{n \times p}$ be semi-unitary such that $\Img{V} = \S$. Define $\MsymT(\S,\S') = \{A \in \Msym(\S,\S') \mid V^\ast A V \in \Gl{p}\}$. Then
\end{sloppypar}
\begin{enumerate}[(i)]
    \item $\MsymT(\S,\S')$ is non-empty, and independent of the choice of $V$.
    \item $\MsymT(\S,\S')$ is a dense open subset of $\Msym(\S,\S')$.
    \item $\MsymT(\S,\S')$ is an embedded submanifold of $\Msym(\S,\S')$ of the same dimension as $\Msym(\S,\S')$.
\end{enumerate}
\end{lemma}

\begin{proof}
\begin{enumerate}[(i)]
    \item 
    \begin{sloppypar}
    Since $\Mpos(\S,\S') \subseteq \MsymT(\S,\S')$ and since $\Mpos(\S,\S')$ is non-empty from \Cref{lemma:existence}, $\MsymT(\S,\S')$ is non-empty.
    To show that $\MsymT(\S,\S')$ does not depend on the choice of $V$, let $\overline{V} \in \F^{n \times p}$ be another semi-unitary matrix such that $\Img{\overline{V}} = \S$. Then there exist $Q \in \Uni{p}$ so that $V = \overline{V}Q$ and $\rank(V^\ast A V) = \rank(Q^\ast \overline{V}^\ast A \overline{V} Q) = \rank(\overline{V}^\ast A \overline{V})$ which shows the result.
    \end{sloppypar}
    \item Let $F$ and $\M$ be defined as in the proof of \Cref{lemma:embedding-Msym}, and define $\MsymT = \{A \in \Sym{n} \mid A_{1:p,1:p} \in \Gl{p}\} \cap \M$. Then from definitions we have $\MsymT(\S,\S') = F(\MsymT)$. It was argued in \Cref{lemma:embedding-Msym} that $F: \M \rightarrow F(\M)$ is a topological embedding, hence we know that it is an open map. So if we show that $\MsymT$ is open in $\M$, then we would conclude that $\MsymT(\S,\S')$ is open in $F(\M) = \Msym(\S,\S')$. But since $\Sym{p} \cap \Gl{p}$ is open in $\Sym{p}$, it follows from the product structure of $\M$ that $\MsymT$ is open in $\M$.\\
    
    To show density, note again that as $F: \M \rightarrow F(\M)$ is a topological embedding, it suffices to show that $\MsymT$ is dense in $\M$. Consider any $A \in \M$. Let $T = A_{1:p,1:p}$. If $T \in \Gl{p}$, we are done. Assume $T \not\in \Gl{p}$. Since $\Gl{p}\cap\Sym{p}$ is dense within $\Sym{p}$, there exist $\{T_i\}_{i=1}^\infty$, $T_i \in \Sym{p}\cap\Gl{p}$ such that $T_i \to T$ as $i \to \infty$. Then define $A_i$ as $(A_i)_{kl} = (T_i)_{kl}$ if $1 \leq k,l \leq p$ and $(A_i)_{kl} = A_{kl}$ otherwise (i.e., $A_i$ is equal to $A$ except in the top-left $p \times p$ block where it equals $T_i$). Then $A_i \in \Sym{n}$, $\rank{((A_i)_{1:p,1:p})} = \rank{(T_i)} = p$, and $A_i \to A$ since $\|A_i - A\|_F^2 = \|T_i - T\|_F^2 \to 0$. So $\MsymT$ is dense within $\M$.
    \item This follows since $\MsymT(\S,\S')$ is open in $\Msym(\S,\S')$, by (ii).
\end{enumerate}
\end{proof}

As a result of \Cref{lemma:density-MsymT} we now have the following consequence after combining with \Cref{cor:sufficient-conditions}.

\begin{lemma}
\label{lemma:X-T-density-prop}
Consider subspaces $\S,\S' \in \G{n}$, such that $\S \subseteq \S' \subseteq \Fn$ with $n \geq 1$. Let $p = \dim{\S}$, $q = \dim{\S'} - \dim{\S}$, and $q \leq p$. Now define the following matrix sets
\begin{equation}
\label{eq:Q-grp1}
\begin{split}
    \mathcal{Q}_{\S} &:= \{A \in \Sym{n} \cap \Gl{n} \mid \dim{\X} = \Indsimple{A}{\S}\} \\
    \mathcal{Q}_{\S,\S'} &:= \{A \in \Sym{n} \cap \Gl{n} \mid \dim{\X} = q, \; \S + A\S = \S'\}.
\end{split}
\end{equation}
Then
\begin{enumerate}[(i)]
    \item $\mathcal{Q}_{\S}$ is dense in $\Sym{n} \cap \Gl{n}$.
    \item $\mathcal{Q}_{\S,\S'}$ is dense in $\Msym(\S,\S')$.
\end{enumerate}
\end{lemma}

\begin{proof}
\begin{enumerate}[(i)]
    \item If $p=0$, then $\mathcal{Q}_{\S} = \Sym{n} \cap \Gl{n}$; so we assume $p \geq 1$. Let $\epsilon > 0$ be arbitrary, and pick any $A \in \Sym{n} \cap \Gl{n}$. Let $\S_1 = \S + A\S$, and $q_1 = \Indsimple{A}{\S}$. Now consider the sets $\Msym(\S,\S_1)$ and $\MsymT(\S,\S_1)$, defined in \eqref{eq:matrix-sets} and \Cref{lemma:density-MsymT} respectively. Then by \Cref{lemma:density-MsymT}(ii), $\MsymT(\S,\S_1)$ is dense in $\Msym(\S,\S_1)$, and one can choose $A' \in \MsymT(\S,\S_1)$, such that $||A - A'||_{\text{F}} \leq \epsilon$. By \Cref{cor:sufficient-conditions}(ii) $A' \in \mathcal{Q}_{\S}$, finishing the proof.
    \item 
    \begin{sloppypar}
    If $p=0$, it implies $\S = \S' = \{0\}$, so $\mathcal{Q}_{\S,\S'} = \Msyminv (\S,\S')$, which is dense in $\Msym(\S,\S')$ by \Cref{lemma:density-MsymT}(ii). Now assume $p \geq 1$. Again $\MsymT(\S,\S')$ is dense in $\Msym(\S,\S')$. So for any $A \in \Msym(\S,\S')$ and $\epsilon > 0$, one can find $A' \in \MsymT(\S,\S')$ such that $||A - A'||_{\text{F}} \leq \epsilon$, and we conclude by applying \Cref{cor:sufficient-conditions}(ii).
    \end{sloppypar}
\end{enumerate}
\end{proof}

\section{Open problems}
\label{sec:future-work}

The results in this paper were proved in the finite dimensional setting, i.e. the matrices and vectors were all finite dimensional. However, we expect many of these results to also hold if the minimization problem \eqref{eq:x_b_w} was instead posed over an infinite dimensional closed subspace (or closed affine subspace) of an infinite dimensional separable Hilbert space. The precise generalizations and proofs are left as future work. However, even in the finite dimensional setting, there are a lot of open questions, which we now state. 

\subsection{Bounds}
\label{ssec:bound-tightness-open-probs}

As shown in \Cref{ssec:zero-dim}, the bound $\dim{\X_b} \leq \Indsimple{A}{\S}$ is not tight. Thus, finding ways to strengthen this bound will be interesting. We also characterized in \Cref{thm:dim-Xb-zero-condition} the precise condition under which $\dim{\X_b} = 0$. It is an open question whether one can find similar conditions that guarantee $\dim{\X_b} = t$, for $1 \leq t \leq \Indsimple{A}{\S}$. Similarly, one would like to know results analogous to \Cref{cor:zero-dim-characterization}, for the sets $\{b \in \Fn \mid \dim{\X_b} = t\}$, when $1 \leq t \leq \Indsimple{A}{\S}$. For example, one could ask how large are these sets, or what is their Hausdorff dimension? It is currently our conjecture that if $\F = \R$ (resp. $\F = \C$), the $n$-dimensional (resp. $2n$-dimensional) Lebesgue measure of these sets are non-decreasing in $t$, over the range $0 \leq t \leq \Indsimple{A}{\S}$. Proving this, or finding a counterexample to this will be interesting. A couple of other related questions, along similar lines, are the following:
\vspace{0.2cm}
\begin{enumerate}[(i)]
    \item Let $\S \subseteq \S' \subseteq \Fn$ be nested subspaces, and let $A \in \Sym{n} \cap \Gl{n}$ be such that $A \in \Msyminv (\S,\S')$. Suppose that $\dim{\X_b} < \Indsimple{A}{\S}$ for all $b \in \Fn$. Then does there exist $B \in \Sym{n} \cap \Gl{n}$, arbitrarily close to $A$ such that both (a) and (b) hold?
    \begin{enumerate}[(a)]
        \item $B \in \Msyminv (\S,\S')$,
        \item There exists $b \in \Fn$ such that $\dim{\X'_b} = \Indsimple{A}{\S}$. Here $\X'_b$ corresponds to the set $\X_b$, but for the matrix $B$.
    \end{enumerate}
    
    \item Let $\S \subseteq \S' \subseteq \Fn$ be nested subspaces, and let $A \in \Sym{n} \cap \Gl{n}$ be such that $A \in \Msyminv (\S,\S')$. Let $\dim{\X_b} < \Indsimple{A}{\S}$ for all $b \in \Fn$. Does there exist $B \in \Sym{n} \cap \Gl{n}$, arbitrarily close to $A$ such that  both (a) and (b) hold?
    \begin{enumerate}[(a)]
        \item $\Indsimple{A}{\S} = \Indsimple{B}{\S}$,
        \item There exists $b \in \Fn$ such that $\dim{\X'_b} = \Indsimple{A}{\S}$. Here $\X'_b$ corresponds to the set $\X_b$, but for the matrix $B$.
    \end{enumerate}
\end{enumerate}
\vspace{0.2cm}

For the next question, we first make the following definition:
\begin{definition}
A 4-tuple $(t,p,q,n)$ of non-negative integers is considered \textit{admissible} if there exists $\S \in \Gr{p}{n}$, and $A \in \Sym{n} \cap \Gl{n}$, such that $\Indsimple{A}{\S} = q$, and $\dim{\X_b} \leq t$ for all $b \in \Fn$, and there exists $c \in \Fn$ such that $\dim{\X_c} = t$.
\end{definition}
We can then ask the following:
\begin{enumerate}[(i)]
    \item Which 4-tuples $(t,p,q,n)$ are admissible? For example, if $q = 0$, then only the tuples $(0,p,0,n)$ are admissible with $n \geq p$. When $q=1$, only the tuples $(1,p,1,n)$ are admissible with $n \geq p+1$ and $p \geq 1$. Our current conjecture is that when $q \geq 1$, a tuple $(t,p,q,n)$ is admissible if and only if $1 \leq t \leq q$, $p \geq q$, and $n \geq p+q$. It will be interesting to know if any obstruction exists that prevents this from happening.
\end{enumerate}

There are also some unresolved questions regarding the tightness of the bound $\dim{\X} \leq \Indsimple{A}{\S}$. For example, we would like to know if there is an explicit example for which this is not an equality. It is also of interest to understand if there are other conditions, similar to the ones listed in \Cref{cor:sufficient-conditions}, that guarantee $\dim{\X} = \Indsimple{A}{\S}$. Finally, it is an interesting open question to ask under which conditions is $\Img{\dD{A}{\omega}{\mu}} = \Y$, for all distinct $\omega,\mu > \omin$. Even in the setting discussed in \Cref{thm:second-condition}, we were only able to prove this everywhere, except on a set of 2-dimensional Lebesgue measure zero, and what exactly is happening on this zero measure set is left to a future analysis.

\subsection{Topological questions}
\label{ssec:topology-open-probs}

We now mention some interesting open problems of a topological nature. It has been observed from numerical experiments that for randomly chosen $A \in \Sym{n} \cap \Gl{n}$, and $\S \in \G{n}$, such that $\Indsimple{A}{\S} \geq 1$, the map $\D{A}{\cdot}b : (\omin,\infty) \to \F^n$ is injective. Based on this, we make two assertions which may or may not be true (in which case a counterexample would be welcome):
\begin{enumerate}[(i)]
    \item For a given $A \in \Sym{n} \cap \Gl{n}$, and $\S = \G{n}$, if $\Indsimple{A}{\S} \geq 1$, there exists a $n-$dimensional (resp. $2n$-dimensional) Lebesgue measure zero set $\mathcal{B}$ if $\F=\R$ (resp. $\F=\C$), such that for all $b \in \F^n \setminus \mathcal{B}$, $\D{A}{\cdot}b$ is injective.
    \item For every $A \in \Sym{n} \cap \Gl{n}$, and $\S = \G{n}$, chosen randomly ensuring that $\Indsimple{A}{\S} \geq 1$, there exists a $n$-dimensional (resp. $2n$-dimensional) Lebesgue measure zero set $\mathcal{B}$ if $\F=\R$ (resp. $\F=\C$), such that for all $b \in \F^n \setminus \mathcal{B}$, $\D{A}{\cdot}b$ is injective.
\end{enumerate}
\vspace{0.2cm}

Another topological question relates to the structure of the sets $\mathcal{Q}_{\S}$ and $\mathcal{Q}_{\S,\S'}$, defined in \eqref{eq:Q-grp1}. How do these sets look? Are these open, closed, or connected? And if not, can these sets be decomposed into simpler sets? In fact, it is natural to additionally define the sets
\begin{equation}
\label{eq:Q-grp2}
\begin{split}
    \mathcal{Q}_{\S}(t) &:= \{A \in \Sym{n} \cap \Gl{n} \mid \dim{\X} = t\}, \hspace{2.5cm} 0 \leq t \leq p, \\
    \mathcal{Q}_{\S,\S'}(t) &:= \{A \in \Sym{n} \cap \Gl{n} \mid \dim{\X} = t, \; \S + A\S = \S'\}, \;\;\; 0 \leq t \leq q,
\end{split}
\end{equation}
and also ask about the structure of these sets, and it will be particularly interesting to know if there are any relationships (for e.g. density type) between these sets, and $\mathcal{Q}_{\S}$ and $\mathcal{Q}_{\S,\S'}$.

We finish by mentioning a topological question about the matrix manifolds, introduced in \eqref{eq:matrix-sets}, that is the topic of current investigation, and for which only partial results are available. The question concerns the number of connected components of each of these manifolds. Currently, we only know the complete answer to this question for the manifold $\M (\S,\S')$. Assuming the setting of \Cref{lemma:embedding-M}, we can state the following (proofs of these statements are simple, and skipped here):
\begin{enumerate}[(i)]
    \item If $\F = \C$, then $\M (\S,\S')$ is connected.
    \item Assume $\F=\R$. If $p=0$, then $\M (\S,\S')$ is connected. If $p \geq 1$, and $q < p$, then $\M (\S,\S')$ is connected. If $p=q\geq 1$, then $\M (\S,\S')$ has exactly 2 connected components.
\end{enumerate}
We are currently working to resolve the connectivity question for the other manifolds.

\section{Acknowledgements}
\label{sec:ack}
We would like to thank Michael Saunders and Eric Hallman for useful discussions at an early stage of this project. We thank Cindy Orozco Bohorquez for pointing out the nullity theorem \cite{strang2004interplay}, and András Vasy for discussions on analytic functions that led to the proof of \Cref{thm:second-condition}.
\appendix
\section{Quadratic forms}
\label{app:appA}

\begin{lemma}
\label{lemma:quadratic_minimization}
Let $m \geq n$, $A \in \F^{m \times n}$ a full column rank matrix, and $b \in \F^n$. The solution to $\min_{x \in \F^n} \|Ax-b\|_2$ is uniquely given by $x = (A^\ast A)^{-1} A^\ast b$.
\end{lemma}

\begin{proof}
Rewrite $f(x) := \|Ax-b\|_2^2 = x^\ast A^\ast A x - (x^\ast A^\ast b) - (b^\ast A x) + b^\ast b$. Since $A$ is full column rank, $P = A^\ast A \in \Pos{n}$. Let $Q$ be such that $P = Q^2$, $Q \in \Pos{n}$. Such $Q$ always exists: let $U\Lambda U^\ast = P$ be the eigenvalue decomposition of $P$, and then one can choose $Q = U \Lambda^{1/2} U^\ast$. Then $f(x) = (Qx)^\ast (Qx) - ((Qx)^\ast (Q^{-1}A^\ast b)) - ((Q^{-1}A^\ast b)^\ast (Qx)) + b^\ast b = \| Q x - Q^{-1}A^\ast b\|_2^2 + b^\ast b - b^\ast A P^{-1} A^\ast b$. The minimum is obtained when $Qx - Q^{-1} A^\ast b = 0$, which happens uniquely (since $Q \in \Pos{n}$) when $Q^2 x = A^\ast b$ or $A^\ast A x = A^\ast b$.
\end{proof}

\begin{proof}[Proof of \Cref{lemma:x-bw}]
Notice that the map $\T \ni x \mapsto V^\ast (x - x_0) \in \Fp{p}$ is a bijection; so each $x \in \T$ can be uniquely written as $x = x_0 + Vy$ for some $y \in \Fp{p}$. Then rewrite the function to minimize in \eqref{eq:general-min-prob} as 
\begin{equation}
\label{eq:appA-proof1}
    \|A_\omega^{s/2}(b-Ax)\|_2 = \| A^{s/2}_\omega (b - A(x_0 + Vy)) \|_2 = \| A^{s/2}_\omega A V y - A_\omega^{s/2}(b-Ax_0) \|_2.
\end{equation}
In this expression, $A_\omega^{s/2} A V$ is full rank since $A_\omega \in \Pos{n}$ (because of the choice of $\omin$), $A \in \Gl{n}$, and  $V$ is full-rank.
Then using \Cref{lemma:quadratic_minimization}, the unique minimizer to \eqref{eq:appA-proof1} is given by $y = (V^\ast A A_\omega^s A V)^{-1} (V^\ast A A^s_\omega) (b - Ax_0)$ or $x = x_0 + Vy = x_0 + VP_\omega^{-1} V^\ast A A_\omega^s (b-Ax_0)$, with $P_\omega = V^\ast A A_\omega^s A V$. For the last part, notice that $x_{b,\omega,s}$ does not depend on the choice of $V$, because if $V' \in \Fpq{n}{p}$ is another full rank matrix whose columns span $\S$, then $V' = V L$ for some $L \in \Gl{p}$, from which it follows that $V P_\omega^{-1} V^\ast = V' P_\omega^{-1} V'^{\ast}$. Similarly $x_{b,\omega,s}$ also does not depend on the choice of $x_0$, because if $\T = x_0' + \S$ is a different representation of $\T$ for some other $x_0' \in \T$, then $x_0 - x_0' = Vy$ for some $y \in \F^p$, and so we have 
\begin{equation}
\begin{split}
    & \left( x_0 + V P_\omega^{-1} V^\ast A A_\omega^{s} (b - Ax_0) \right) - \left( x_0' + V P_\omega^{-1} V^\ast A A_\omega^{s} (b - Ax_0') \right) \\
    & = (x_0 - x_0') - V P_\omega^{-1} V^\ast A A_\omega^{s} A V y = (x_0 - x_0') - V P_\omega^{-1} P_\omega y = 0.
\end{split}
\end{equation}
This completes the proof.
\end{proof}
\section{\normalfont{Reducing the proof of \texorpdfstring{\Cref{thm:main_result}}{Lemma}} for the case \texorpdfstring{$n = p+q$ to the case $n > p+q$, where $p \geq q \geq 1$}{}}
\label{app:appB}

Suppose $n = p+q$, and the quantities $A$, $b$, $\T$, $\S$, etc. defined in the statement of \Cref{thm:main_result}. 
Define $\tilde{A} \in \Gl{n+1}$, $\tilde{b} \in \Fp{n+1}$, and for any $\tilde{c} \in \Fp{n+1}$, also define $\tilde{x}_{\tilde{c},\omega} \in \Fp{n+1}$ as
\vspace*{-0.1cm}
\begin{equation}
    \tilde{A} = 
    \begin{bmatrix}
    A & 0 \\
    0 & - \omin
    \end{bmatrix}, \;\;\;
    \tilde{b} = \begin{bmatrix} b \\ \alpha \end{bmatrix}, \;\;\;
    \tilde{x}_{\tilde{c},\omega} = \argmin_{\tilde{x} \in \tilde{\T}} \|\tilde{A}_\omega^{-\frac{1}{2}}(\tilde{c} - \tilde{A}\tilde{x}) \|_2,
\end{equation}
for any $\omega > \omin$, and $\alpha \in \F$, where $\tilde{A}_\omega = \tilde{A} + \omega I$, and $\tilde{\T}$ is an affine subspace of $\Fp{n+1}$ built as $\tilde{\T} = \{ (x,0) \in \F^{n+1} \mid x \in \T\}$. Also define the subspace $\tilde{\S} = \{(x,0) \mid x \in \S\}$. 
Notice that we have $\lambda_{\text{min}}(\tilde{A}) = \lambda_{\text{min}}(A) = -\omin$, $\dim{\tilde{\T}}=p$, and $\Gam{p}{\tilde{\T}} = \tilde{\S}$. Also note that given the structure of $\tilde A$, 
\begin{equation} 
\label{eq:tilde_non_tilde} \|\tilde{A}_\omega^{-\frac{1}{2}}(\tilde{b} - \tilde{A}\tilde x) \|_2^2 =
\|A_\omega^{-\frac{1}{2}}(b - Ax) \|_2^2 + \|(-\omin+\omega)^{-\frac{1}{2}}(\alpha + \omin y)\|_2^2,
\end{equation}
for $\tilde x \in \F^{n+1}, x \in \F^n, y \in \F$, and $\tilde x = \begin{bmatrix} x \\ y \end{bmatrix}$.
Then given the structure of $\tilde \T$ (i.e., $y = 0$), the last term is a constant and minimizing over $\tilde x \in \tilde \T$ is equivalent to minimizing over $x \in \T$ with
$\tilde x = \begin{bmatrix} x \\ 0 \end{bmatrix}$. Finally, let us also define $\tilde{\X}_{\tilde{b}} := \Aff{\{\tilde{x}_{\tilde{b},\omega} \mid \omega > \omin\}}$ and $\tilde{\X} := \sum_{\tilde{c} \in \Fp{n+1}} \Gam{\dim{\tilde{\X}_{\tilde{c}}}}{\tilde{\X}_{\tilde{c}}}$. Then we can prove the following two lemmas:

\begin{lemma}
\label{lemma:reduction}
Assuming that \Cref{thm:main_result} holds in the $n > p+q$ case, it also holds in the $n = p+q$ case.
\end{lemma}

\begin{proof}
Given this choice of $\tilde{\S}$ and $\tilde A$, we have $\dim{\tilde \S} = \dim{\S}$, $\dim{\tilde \S + \tilde A \tilde \S} = \dim{\S + A\S}$ and $\Indsimple{\tilde A}{\tilde{\S}} = \Indsimple{A}{\S}$.
Furthermore, let $V$, $V'$ be chosen as in \Cref{thm:main_result}. Then $\tilde V = \begin{bmatrix} V \\ 0 \end{bmatrix}$, $\tilde V' = \begin{bmatrix} V' \\ 0 \end{bmatrix}$ are such that $\Img{\tilde V} = \tilde\S$, $\Img{\begin{bmatrix} \tilde V & \tilde V'\end{bmatrix}} = \tilde S + \tilde A \tilde \S$.
We can then apply \Cref{thm:main_result} to $\tilde A$ with $\tilde{\S}$ and $\tilde \T$ (with $\tilde V$ and $\tilde V'$ as basis) to conclude that $\tilde x_{\tilde b,\omega} - \tilde x_{\tilde b,\mu} \in \Img{\tilde V(\tilde H^* \tilde H)^{-1} \tilde B^*}$ with $\tilde H, \tilde B$ defined similarly as $H, B$ in \Cref{thm:main_result}. 
But given the structure of $\tilde A$ and the choices of $\tilde V, \tilde V'$, one easily sees that $\tilde B = B$, $\tilde H = H$ and hence $\tilde x_{\tilde b,\omega} - \tilde x_{\tilde b,\mu} \in \Img{\tilde V(H^* H)^{-1} B^*} = \Img{\begin{bmatrix} V(H^*H)^{-1} B^* \\ 0 \end{bmatrix}}$.
Finally this implies $x_{b,\omega} - x_{b,\mu} \in \Img{V(H^* H)^{-1} B^*}$.
\end{proof}

\begin{lemma}
\label{lemma:reduction1}
$\dim{\tilde{\X}_{\tilde{b}}} = \dim{\X_b}$, and $\dim{\tilde{\X}} = \dim{\X}$.
\end{lemma}

\begin{proof}
Since $\tilde x_{\tilde b,\omega} = \begin{bmatrix} x_{b,\omega} \\ 0\end{bmatrix}$ for all $b \in \Fn$, and $\alpha \in \F$, the conclusion follows readily.
\end{proof}
\section{The nullspace of \texorpdfstring{$H^\ast$}{H*}}
\label{app:appC}

We present here a geometrical relationship between the images of $T$ and $B^\ast$, and the nullspace of $H^{\ast} = \begin{bmatrix}T & B^\ast \end{bmatrix}$, as defined in the statement of \Cref{thm:main_result}. Recall that $T \in \Fpq{p}{p}$ is Hermitian, and $H^\ast$ and $B^\ast \in \Fpq{p}{q}$ are full rank, with $p \geq q$. To do this we first establish a more general result below. In this appendix, for any matrix $M \in \Fpq{p}{q}$, and $\mathcal{A} \in \G{p}$, we define $\text{pre}_{M}(\mathcal{A}) := \{x \in \Fp{q} \mid M x \in \mathcal{A}\}$.

\begin{lemma}
\label{lemma:nullspace-geometry}
Let $M = \begin{bmatrix} M_1 & M_2 \end{bmatrix}$ be a matrix such that $M_1 \in \Fpq{s}{s_1}$, $M_2 \in \Fpq{s}{s_2}$, with $s \geq s_2$, and suppose that $\rank{(M_2)} = s_2$. Let $N$ be any matrix such that the columns of $N$ span the nullspace of $M$, and let us partition $N$ as $N = \begin{bmatrix} N_1 \\ N_2 \end{bmatrix}$, where $N_1$ are the first $s_1$ rows of $N$. Then
\begin{enumerate}[(i)]
    \item $N_2 = -(M_2^\ast M_2)^{-1} M_2^\ast M_1 N_1$.
    \item $\rank{(N)} = \rank{(N_1)}$.
    \item $\Img{N_1} = \text{pre}_{M_1} (\Img{M_2})$.
\end{enumerate}
\end{lemma}

\begin{proof}
\begin{enumerate}[(i)]
    \item As the columns of $N$ span the nullspace of $M$, we have $M_1 N_1 + M_2 N_2 = 0$, and since $\rank{(M_2)} = s_2$, the matrix $M_2^\ast M_2$ is invertible, from which the result follows.
    
    \item Let $N' := -(M_2^\ast M_2)^{-1} M_2^\ast M_1$. From (i), we have $N = \begin{bmatrix} I \\ N' \end{bmatrix} N_1$. Since $\begin{bmatrix} I \\ N' \end{bmatrix}$ is full column rank, the conclusion follows.
    
    \item From (i) $M_1 N_1 = - M_2 N_2$; so each column of $N_1$ is in $\text{pre}_{M_1} (\Img{M_2})$, and hence $\Img{N_1} \subseteq \text{pre}_{M_1} (\Img{M_2})$. Now define $\mathcal{A} := \Img{M_1} \cap \Img{M_2}$, and observe that $\text{pre}_{M_1} (\Img{M_2}) = \text{pre}_{M_1} (\mathcal{A})$. To prove the result it suffices to show that $\dim{\Img{N_1}} = \dim{\text{pre}_{M_1} (\mathcal{A})}$. Let $\rank{(M)} = r$, and suppose $\rank{(N)} = t$, so $\dim{\Img{N_1}} = t$ by part (ii). From the rank-nullity theorem we first have
    \begin{equation}
    \label{eq:nullspace-geometry-proof-1}
        t = s_1 + s_2 - r, \;\; \dim{\Img{M_1}} + \dim{\Ker{M_1}} = s_1,
    \end{equation}
    and since $\Img{M_2} = \mathcal{A} \oplus (\mathcal{A}^\perp \cap \Img{M_2})$, and $\Fp{s} = \Img{M_1} \oplus \Img{M_1}^\perp$ we also have
    \begin{equation}
    \label{eq:nullspace-geometry-proof-2}
    \begin{split}
        & \dim{\mathcal{A}} + \dim{\mathcal{A}^\perp \cap \Img{M_2}} = \dim{\Img{M_2}} = s_2, \\
        & \dim{\Img{M_1}} + \dim{\Img{M_1}^\perp} = s.
    \end{split}
    \end{equation}
    Next notice that $\Img{M_1} + \Img{M_2} = \Img{M_1} + \mathcal{A} \oplus (\mathcal{A}^\perp \cap \Img{M_2}) = \Img{M_1} + (\mathcal{A}^\perp \cap \Img{M_2})$. But $\Img{M_1} \cap (\mathcal{A}^\perp \cap \Img{M_2}) = (\Img{M_1} \cap \Img{M_2}) \cap \mathcal{A}^\perp = \mathcal{A} \cap \mathcal{A}^\perp = \{0\}$, so in fact
    \begin{equation}
    \label{eq:nullspace-geometry-proof-3}
        \dim{\Img{M_1}} + \dim{\mathcal{A}^\perp \cap \Img{M_2}} = \dim{\Img{M_1} + \Img{M_2}} = r.
    \end{equation}
    Thus combining \eqref{eq:nullspace-geometry-proof-2} and \eqref{eq:nullspace-geometry-proof-3} gives $\dim{\mathcal{A}} + \dim{\Img{M_1}^\perp} = s + s_2 - r$, and then using \eqref{eq:nullspace-geometry-proof-1} and \eqref{eq:nullspace-geometry-proof-2} gives 
    \begin{equation}
    \label{eq:nullspace-geometry-proof-4}
        \dim{\mathcal{A}} + \dim{\Ker{M_1}} = s_1 + s_2 - r = t.
    \end{equation}
    Now consider the set $\mathcal{A}' := \Ker{M_1}^\perp \cap \text{pre}_{M_1} (\mathcal{A})$, which is a subspace of $\text{pre}_{M_1} (\mathcal{A})$. Then the restriction to $\mathcal{A}'$ of the linear map given by $M_1$ is an isomorphism $M_1 |_{\mathcal{A}'} : \mathcal{A}' \rightarrow \mathcal{A}$, which gives $\dim{\mathcal{A}} = \dim{\mathcal{A}'}$. Notice that this proves the result as $\text{pre}_{M_1} (\mathcal{A}) = \mathcal{A}' \oplus \Ker{M_1}$, since $\Ker{M_1}$ is also a subspace of $\text{pre}_{M_1} (\mathcal{A})$, giving $\dim{\text{pre}_{M_1} (\mathcal{A})} = \dim{\mathcal{A}'} + \dim{\Ker{M_1}} = t$, using \eqref{eq:nullspace-geometry-proof-4}.
\end{enumerate}
\end{proof}

Let us apply \Cref{lemma:nullspace-geometry} to characterize the nullspace of $H^\ast$, which is of dimension $q$, and derive some consequences. If we choose $N \in \Fpq{(p+q)}{q}$ to be full column rank, as in \Cref{lemma:HTSH}, and write $N = \begin{bmatrix} N_1 \\ N_2 \end{bmatrix}$ as in \Cref{lemma:nullspace-geometry} with $N_1 \in \Fpq{p}{q}$, then we have $N_2 = -(BB^\ast)^{-1} B T N_1$, and $\Img{N_1} = \text{pre}_{T} (\Img{B^\ast})$. Denoting $\mathcal{A} := \Img{T} \cap \Img{B^\ast}$, and $\mathcal{A}' := \Ker{T}^\perp \cap \text{pre}_{T} (\mathcal{A})$, the proof of \Cref{lemma:nullspace-geometry}(iii) shows that $\Img{N_1} = \text{pre}_{T} (\mathcal{A}) = \Ker{T} \oplus \mathcal{A}'$.
Now as $T \in \Sym{p}$, we have $\Ker{T} = \Img{T}^\perp$, which means that $\mathcal{A}' \subseteq \Img{T}$, and $\Img{N_1} = \Img{T}^\perp \oplus \mathcal{A}'$ (an orthogonal direct sum). Thus we have that both $\mathcal{A}, \mathcal{A}'$ are subspaces of $\Img{T}$. Also $T \mathcal{A}' \subseteq \mathcal{A}$ from definition, and moreover $\dim{\mathcal{A}} = \dim{\mathcal{A}'}$ from the proof of \Cref{lemma:nullspace-geometry}(iii), which means $T \mathcal{A}' = \mathcal{A}$. Now by the spectral theorem, the restriction to the subspace $\Img{T}$ of the linear operator $T$, i.e. $T|_{\Img{T}}$, is invertible; thus in fact $T \mathcal{A}' = T|_{\Img{T}} \mathcal{A}' = \mathcal{A}$, or $\mathcal{A}' = T|_{\Img{T}}^{-1} \mathcal{A}$. Denoting the pseudoinverse \cite{moore1920reciprocal,penrose1955generalized} of $T$ by $T^{\dagger}$, it is also easily checked that $T^{\dagger} |_{\Img{T}} = T |_{\Img{T}}^{-1}$, since $T \in \Sym{n}$, and so we have proved the following corollary:

\begin{corollary}
\label{cor:nullspace-H*}
Let $N \in \Fpq{(p+q)}{q}$ be a full column rank matrix, whose columns span the nullspace of $H^\ast$, and let $T^\dagger$ be the pseudoinverse of $T$. If we partition $N$ as $N = \begin{bmatrix} N_1 \\ N_2 \end{bmatrix}$, where $N_1 \in \Fpq{p}{q}$ and $N_2 \in \Fpq{q}{q}$. Then
\begin{enumerate}[(i)]
    \item $N_2 = -(BB^\ast)^{-1} B T N_1$, and $\rank{(N)} = \rank{(N_1)} = q$.
    \item $\Img{N_1} = \Img{T}^\perp \oplus T^{\dagger} (\Img{T} \cap \Img{B^\ast})$.
\end{enumerate}
\end{corollary}

The relationship between some of these subspaces of $\Fp{p}$ appearing in \Cref{cor:nullspace-H*} is illustrated in \cref{fig:nullspace-H*}.
\begin{figure}[ht]
    \centering
    \includegraphics[width=0.75\textwidth]{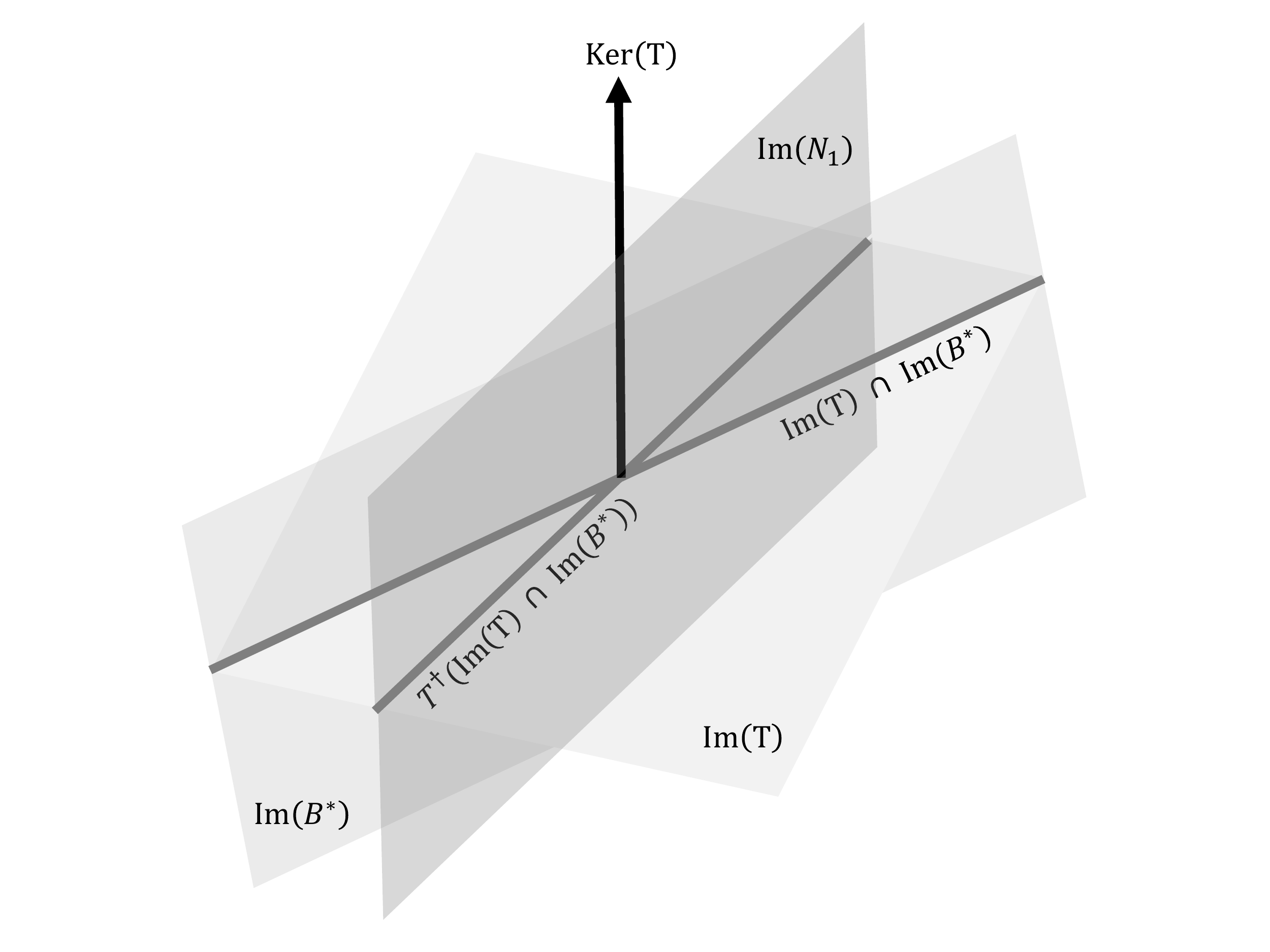}
    \caption{An illustration of the geometrical relationship between the various subspaces of $\Fp{p}$: $\Img{T}$, $\Img{B^\ast}$, $\Ker{T}$ and $\Img{N_1}$.}
    \label{fig:nullspace-H*}
\end{figure}

Finally we note a couple of special cases in the next lemma that follow from \Cref{cor:nullspace-H*}.
\begin{lemma}
\label{lemma:nullspace-special-cases}
Let $N, N_1, N_2$ be as in \Cref{cor:nullspace-H*}. Then
\begin{enumerate}[(i)]
    \item $\Img{B^\ast} \subseteq \Img{T}$ if and only if $T \in \Gl{p}$. In this case, one can choose $N_1 = -T^{-1} B^\ast$, and $N_2 = I$.
    
    \item $\Img{T} \subseteq \Img{B^\ast}$ if and only if $B^\ast \in \Gl{p}$. In this case one can choose $N_1 = Q$, for any $Q \in \Gl{p}$ (for e.g. $Q = I$).
    
    \item If $\Img{T} \cap \Img{B^\ast} = \{0\}$, then $N_1$ should be chosen such that $\Img{N_1} = \Img{T}^\perp$, and in this case $N_2 = 0$.
\end{enumerate}
\end{lemma}

\begin{proof}
\begin{enumerate}[(i)]
    \item If $T \in \Gl{p}$, $\Img{T} = \Fp{p}$ and so $\Img{B^\ast} \subseteq \Img{T}$. If $\Img{B^\ast} \subseteq \Img{T}$, then $\Img{T} = \Fp{p}$ as $\rank{(H^\ast)} = p$, so $T \in \Gl{p}$. In this case, $\Img{T}^\perp = \{0\}$, $\Img{T} \cap \Img{B^\ast} = \Img{B^\ast}$, and $T^\dagger = T^{-1}$. Thus from \Cref{cor:nullspace-H*}(ii) we have $\Img{N_1} = T^{-1} \Img{B^\ast} = \Img{T^{-1} B^\ast}$, and so we can choose $N_1 = -T^{-1} B^\ast$. With this choice, we get $N_2 = I$ by \Cref{cor:nullspace-H*}(i).
    
    \item Interchanging the roles of $T$ and $B^\ast$ in the proof of part (i) proves that $\Img{T} \subseteq \Img{B^\ast}$ if and only if $B^\ast \in \Gl{p}$. In this case $\Img{T} \cap \Img{B^\ast} = \Img{T}$, and so $T^\dagger (\Img{T} \cap \Img{B^\ast}) = T^\dagger \Img{T} = T|_{\Img{T}}^{-1} \Img{T} = \Img{T}$. This gives using \Cref{cor:nullspace-H*}(ii) that $\Img{N_1} = \Img{T}^\perp \oplus \Img{T} = \Fp{p}$.  Thus $N_1$ must be chosen to be an invertible matrix in $\Gl{p}$.
    
    \item It follows directly in this case that $\Img{N_1} = \Img{T}^\perp$ from \Cref{cor:nullspace-H*}(ii). Thus $TN_1 = 0$ and this implies $N_2 = 0$.
\end{enumerate}
\end{proof}
\section{Properties of real analytic maps}
\label{app:appD}

We collect here some useful and well-known facts about real analytic functions that were used in this paper, in \Cref{sec:converse,sec:related-results}.

\begin{lemma}
\label{lemma:real-analyticity-prop}
Let $f$ denote the map $\R^m \supseteq \mathcal{U} \ni (x_1,\dots,x_m) \mapsto (y_1,\dots,y_k) \in \R^k$, which is real analytic \footnote{This means for each $1 \leq i \leq k$, $y_i$ is a real analytic function of $x_1,\dots,x_m$.} on a connected, open subset $\mathcal{U}$. Then
\begin{enumerate}[(i)]
    \item If $f$ is not constant, the zero set $\mathcal{Z} = \{x \in \R^m \mid f(x) = 0\}$ has $m$-dimensional Lebesgue measure zero in $\R^m$.
    
    \item If $m = 1$, then additionally the zero set $\mathcal{Z}$ is discrete\footnote{This means that for each $p \in \mathcal{Z}$, there is an open interval containing $p$, but no other zeros.}, and hence countable.
\end{enumerate}
\end{lemma}

\begin{proof}
\begin{enumerate}[(i)]
    \item We will use $\pi_m (\mathcal{V})$ to denote the $m$-dimensional Lebesgue measure of a measurable subset $\mathcal{V} \subseteq \R^m$. The case $k=1$ is a simple consequence of Lemma 5.22 in \cite{kuchment2016overview} (see also \cite{mityagin2020zero}). For the case $k \geq 1$, there exists $1 \leq i \leq k$ such that $y_i$ is not constant, as $f$ is not constant. Then $\pi_m (\mathcal{Z}) \leq \pi_m (\{x \in \R^m \mid y_i = 0\}) = 0$.
    
    \item It suffices to prove this for the case $k = 1$. For $k \geq 1$, with $y_i$ chosen as in the proof of (i), we have $\mathcal{Z} \subseteq \{x \in \R^m \mid y_i = 0\}$. The proof of the $k=1$ case now follows in a manner similar to the complex setting (i.e. for non-constant analytic maps $\C \supseteq \mathcal{U} \rightarrow \C$), as is done for example in Theorem 3.7 - Corollary 3.10 of \cite{conway2012functions}. Finally, discrete subsets of $\R$ are countable.
\end{enumerate}
\end{proof}

\begin{lemma}
\label{lemma:real-analyticity-complex-range}
The conclusions of \Cref{lemma:real-analyticity-prop} continue to hold if $f$ is instead an analytic map $\R^m \supseteq \mathcal{U} \ni (x_1,\dots,x_m) \mapsto (y_1,\dots,y_k) \in \C^k$ on a connected, open subset $\mathcal{U}$.
\end{lemma}

\begin{proof}
This follows easily because both the real and imaginary components are real analytic maps, which reduces us to the setting of \Cref{lemma:real-analyticity-prop}.
\end{proof}

\bibliographystyle{plain}
\bibliography{bibliography}
\end{document}